\documentclass[11pt,oneside,a4paper,final]{article}

\usepackage{authblk}

\usepackage[T1]{fontenc}
\usepackage{lmodern}
\usepackage{mathtools,amsthm,amssymb}
\usepackage[final]{graphicx}
\usepackage{fixmath}
\usepackage{braket}
\usepackage{physics}
\usepackage{xcolor}
\usepackage[noend]{algorithmic}
\usepackage{pgfpages}
\usepackage{algorithm}
\usepackage{caption}
\usepackage{subcaption}
\usepackage{booktabs}
\usepackage{transparent}

\usepackage{tikz}
\usepackage{pgfplots}
\usetikzlibrary{plotmarks}
\pgfplotsset{compat=1.3}

\usepackage[final,stretch=10]{microtype}
\usepackage{psfrag}

\usepackage[colorlinks=true,linkcolor=black,urlcolor=black,citecolor=black]{hyperref}
\hypersetup{final} 

\newcommand{\C}{\mathbb{C}}

\newcommand{\R}{\mathbb{R}}

\renewcommand{\phi}{\varphi}

\newcommand{\de}{\mathop{}\!\mathrm{d}}

\renewcommand*{\i}{\mathrm{i}}

\DeclarePairedDelimiter{\inner}{(}{)}

\DeclareMathOperator{\diag}{diag}

\DeclareMathOperator{\Id}{Id}

\DeclareMathOperator*{\argmin}{arg\,min}

\newenvironment{curlyeq}{\left\{\quad\begin{aligned}}{\end{aligned}\right.}

\newcommand*{\Lap}{\boldsymbol{\upDelta}}
\newcommand*{\dt}{\Delta t}
\newcommand*{\Hcont}{\mathcal{H}}
\newcommand*{\Jcont}{\mathcal{J}}
\newcommand*{\Fcont}{\mathcal{G}}

\newcommand*{\disc}[1]{\vb*{#1}}

\newcommand*{\Hdisc}{\disc{H}}
\newcommand*{\Jdisc}{\disc{J}}
\newcommand*{\Fdisc}{\disc{G}}

\newcommand*{\Ddiv}{\mathop{\div}_{\grid{E}\to \grid{I}}}
\newcommand*{\Dgrad}{\mathop{\grad}_{\grid{I}\to \grid{E}}}
\newcommand*{\Dcurl}{\mathop{(\hat{k}\vdot\curl)}_{\grid{E}\to \grid{V}}}
\newcommand*{\Dgradp}{\mathop{\grad^\perp}_{\grid{V}\to \grid{E}}}
\newcommand*{\Dkcross}{\mathop{\hat{k}\cross}_{\grid{E}\to \grid{E}}}

\newcommand*{\DLap}{\widetilde{\mathop{\Lap}}_{\grid{E}\to \grid{E}}}

\newcommand*{\grid}[1]{\mathsf{#1}}

\newcommand*{\reference}[1]{{#1}^{\textrm{ref}}}

\newcommand*{\laylowL}{\widehat{L}}
\newcommand*{\laylow}[1]{\widehat{#1}}

\DeclareMathOperator{\Cour}{C}
\newcommand*{\dtC}{\dt_{\mathrm{C}}}

\newcommand*{\arnoldi}[1]{\mathbb{#1}}

\DeclarePairedDelimiter{\inp}{\{}{\}}

\graphicspath{{./figures/}}

\numberwithin{equation}{section}

\theoremstyle{plain}
\newtheorem{theorem}{Theorem}
\numberwithin{theorem}{section}
\newtheorem{proposition}[theorem]{Proposition}

\newtheorem{corollary}[theorem]{Corollary}
\theoremstyle{definition}

\theoremstyle{remark}
\newtheorem{remark}{Remark}

\begin{document}






\author[1]{Konstantin Pieper}
\author[1,2]{K.\ Chad Sockwell}
\author[1]{Max Gunzburger}
\affil[1]{%
Florida State University,
Dept.\ Scientific Computing,
400 Dirac Science Library,
Tallahassee FL 32306
\textup{(\texttt{kpieper@fsu.edu}, \texttt{kcs12j@my.fsu.edu}, \texttt{mgunzburger@fsu.edu})}}
\affil[2]{
Los Alamos National Laboratory,
P.O. Box 1663, T-3, MS-B216,
Los Alamos, NM 87545
\textup{(\texttt{sockwell@lanl.gov})}
}
\date{12 Jul 2019}


\title{Exponential time differencing for mimetic multilayer ocean models}

\maketitle

\begin{abstract}
A framework for exponential time discretization of the multilayer rotating shallow water
equations is developed in combination with a mimetic discretization in space.
The method is based on
a combination of existing exponential time differencing (ETD) methods and a careful choice
of approximate Jacobians. The discrete Hamiltonian structure and conservation properties
of the model are taken into account, in order to ensure stability of the method for large
time steps and simulation horizons. In the case of many layers, further efficiency can be
gained by a layer reduction which is based on the vertical structure of fast
and slow modes. Numerical experiments on the example of a mid-latitude regional
ocean model confirm long term stability for time steps increased by an order of magnitude
over the explicit CFL, while maintaining accuracy for key statistical quantities.
\end{abstract}



\section{Introduction}

Despite their relevance in climate modeling, the numerical solution of the primitive
equations used for the modeling of
global or regional oceanic circulation remains challenging. This is
due to the fact that the partial differential equations
underlying the derivation of the primitive equations are of hyperbolic type,
since physical diffusion terms are negligible at practically feasible grid resolutions.
Concerning time discretization, a particular challenge lies in the presence of multiple
time scales (due to, e.g., fast free-surface wave modes or locally refined meshes near
coastal boundaries), which requires special schemes to take advantage of this
structure. Otherwise, straightforward explicit integrators/Runge-Kutta schemes -- which are
usually very effective for problems of hyperbolic character -- are restricted to an
excessively small time step, degrading performance.
Due to these requirements, specialized implicit methods based on the structure of
the fast vertical mode have been
developed; see, e.g.~\cite{Dukowicz_Smith94jgro}. However, they can be
affected by loss of accuracy due to high frequency error for large
time steps. Moreover, scalability concerns arise on parallel computers due to the requirement of
solving large linear systems.
Subsequently, split-explicit methods~\cite{higdon2005two,ringler2013multi} have been developed and applied with great
success, which treat fast and slow modes with different explicit time
discretization schemes. For an overview over the earlier developments in implicit and
split-explicit time stepping methods for atmosphere and ocean models, we also refer
to~\cite{Madala:1981} and~\cite[Section~5]{Higdon:2006}.

Recently, exponential integrators (see, e.g., \cite{hochbruck2010exponential}), also
called exponential time differencing methods (ETD), have
gained attention in the context of circulation models~\cite{archibald2011multiwavelet,clancy2013use,gaudreault2016efficient,luan2019further}.
Due to the presence of multiple time scales, ETD methods seem well suited to enable
efficient large time step computations together with a reasonably accurate representation of high
frequency dynamics.
For the purposes of this paper, we consider a simplified ocean model which still
exhibits all of the difficulties mentioned above.
Concretely, we restrict attention to the rotating shallow water
equation (RSWE) with multiple horizontal layers, which corresponds to a vertical
discretization of the primitive equations cast in an isopycnal vertical coordinate
system. Concerning the spatial discretization, mimetic
finite difference/finite volume (FD/FV) schemes have proven to be very
effective here. Specifically, we work with the TRiSK scheme~\cite{TRiSK:2009,ringler2010unified},
which has many of the features of classical FD/FV schemes on Cartesian grids
but additionally allows for the use of multi-resolution meshes. We
emphasize that the resulting discretization can be set up to inherit the Hamiltonian
structure of the underlying RSWE, which leads to exact energy conservation of the space
discrete model.
Based on this, we develop a framework for exponential time discretization which relies on
a combination of existing exponential Runge-Kutta (ETD-RK) methods (see, e.g.,
\cite{hochbruck2010exponential,hochbruck2005explicit}), developed for semi-linear equations
of the form
\[
\partial_t \disc{V} = \disc{F}[\disc{V}] = \disc{A} \disc{V} + \disc{r}[\disc{V}]
\]
with an appropriate choice of the linear operator \(\disc{A}\).
Here, we prefer an approximation to the Jacobian of the forcing term over
the full Jacobian \(\disc{F}'\) (which would result in a Rosenbrock-ETD method),
due to favorable properties concerning the implementation and structure
of the linear operator and its numerical treatment. Physically, the proposed choice of
\(\disc{A}\) neglects the linearized advection and potential vorticity dynamics, which
typically evolve on a relatively slow time scale, while still capturing the fast external
(and internal) gravity waves. This leads to a class of explicit exponential Runge-Kutta
methods which can take time steps significantly increased over an explicit integrator,
while still maintaining stability and sufficient accuracy.

On the discrete level, the proposed class of linear operators
inherits the Hamiltonian structure and corresponds to a skew-symmetric matrix with respect
to an appropriate inner product.
In turn, this enables the use of specialized efficient skew-Lanczos methods for the practical evaluation of the matrix exponentials and
\(\phi\)-functions, which are required for the implementation of an ETD method.
Moreover, the matrix exponential \(\exp(\dt\disc{A})\) maintains the linearized energy of the
RSWE for all \(\dt\), which improves numerical stability for large time steps.
Since \(\disc{A}\) describes the linearized
free-surface and internal gravity waves around a reference configuration,
we can additionally use the knowledge of the approximate
structure of the fast and slow wave modes to further reduce the computational complexity.
This is done by performing an additional projection of the linear operator onto the fast
subspace, where we take special care to preserve the symmetry properties of \(\disc{A}\). In a
typical configuration of a global ocean model, the difference in the free-surface speed
and speed of internal gravity waves is greater that an order of magnitude. Thereby, this
projection enables computational savings proportional to the number of layers, while still capturing
the free-surface dynamics in the linear operator. Thus, we obtain a faster
method at the cost of
additional restrictions on the maximal stable time step, due to the neglected internal modes.

In order to enable stable computations for very long simulation horizons
(typically decades, in the context of climate models),
additional diffusion terms have to be incorporated into the discrete model, in order to
prevent a build-up of turbulent energy in the smallest (grid-level) scales. Here, we
employ a variant of the classical biharmonic smoothing. Since optimal
choices of the parameters of these diffusion
terms are typically not stiff when compared to the fastest gravity waves, for efficiency
we treat them explicitly, by adding them to the residual \(\disc{r}\). However, since
\(\disc{A}\) does not take into account the dissipation, this can lead to a spurious
build-up of kinetic energy in the smallest scales for large time step simulations (over
the course of several months). To remedy this, we describe a simple method of adding a
minimal amount of artificial high-frequency dissipation at minimal cost,
by tuning the matrix \(\phi\)-functions occurring in the method. The described procedure
can be set up to maintain the formal order of accuracy of the scheme.

Finally, since exact mass conservation on the discrete level is an essential requirement for
long-running simulations, we take care that the proposed methods fulfill
this basic requirement. This is obtained by proving that the considered exponential integrators preserve linear
invariants for an appropriate choice of \(\disc{A}\).

This paper is structured as follows: In
section~\ref{sec:continuous} we introduce the concrete space and time continuous model and
the underlying Hamiltonian structure. Section~\ref{sec:TRiSK} summarizes the necessary
details on the spatial discretization scheme. In section~\ref{sec:ETD} the relevant background on
exponential integrators, the efficient evaluation of the matrix exponential, and the proposed
artificial dissipation strategy is given.
Section~\ref{sec:mode_splitting} is devoted to the layer reduction strategy, which allows
to take advantage of the vertical structure of the fast modes.
In Section~\ref{sec:results},
we test the methods based on a simplified regional mid-latitude ocean model. In particular, we show
that the methods deliver high order accuracy for large time step configurations,
and investigate the effect of the artificial diffusion. Moreover, we perform decade long
simulations with several configurations of the methods. Here, single trajectories can not be
compared anymore due to the underlying chaotic structure of the model. However, we verify
that key statistical quantities, such as mean flow and variance of the sea-surface height
are accurately replicated in each simulation, while significant cost reductions are
achieved over an explicit time discretization scheme.

\section{Continuous equations}
\label{sec:continuous}

The governing equations used in this work are the multilayer rotational shallow water equations,
which serve as proxy to the primitive equations in the MPAS-O
model~\cite{ringler2013multi}. For the sake of readability, we first explain the
single-layer model, and then the multilayer extension.

\subsection{Single-layer rotating shallow water equations} The model equations
for this work are defined on a spherical surface, with a variable bottom topography, and
with multiple layers considered. We will start with the simple case of the single-layer
equations.
We denote by $\mathbb{S}^2$ the two-sphere with outward oriented surface-orthogonal unit vector
\(\hat{k}\), and by \(\Omega\) an open sub-manifold with boundary \(\partial\Omega\) 
and outer normal \(\hat{n}\). The time variable is denoted by \(t \in \R\).
The single-layer rotating shallow water equations can now be expressed
in terms of the fluid thickness $h\colon \R\times\Omega \to \R$ and the velocity
$u \colon \R\times\Omega \to \R^3$ in the vector-invariant form as
\begin{equation}
\label{eq:sw}
\begin{curlyeq}
\partial_t h &= - \div{(h u)} &&\text{in}\;\;\Omega, \\
\partial_t u &= - \grad{\left(K[u] + g(h+b)\right)} - q[h,u] \, \hat{k}\cross (h u) +
               \Fcont[h,u] &&\text{in}\;\;\Omega,
\end{curlyeq}
\end{equation}
together with the constraint that the velocity should be tangential to the surface \(u
\vdot \hat{k} = 0\) in \(\Omega\), the no normal flow
boundary condition \(u \vdot \hat{n} = 0\) on \(\partial\Omega\), and appropriate initial
conditions on $h$ and $u$. Here,
$K[u] = \abs{u}^2/2 = (u \vdot u)/2$ is the kinetic energy, $\hat{k} \cross u$ the
perpendicular velocity, $q[h,u] = (\hat{k}\vdot\curl{u} + f)/h$
is the potential vorticity with $f$ the Coriolis parameter. The bathymetry $b<0$ 
encodes the bottom topography. The differential operators are defined in the canonical
way on $\mathbb{S}^2$. The term $\Fcont(h,u)$ contains additional forcing, arising
either from wind or bottom drag or possible diffusion terms, which will be
detailed later. For now, we only assume that
\(\hat{k} \vdot \Fcont(h,u) = 0\), to ensure the consistency of the momentum equation
with the constraint on the velocity.

The rotating shallow water equations~\eqref{eq:sw} can also be given in a more
abstract form, using a Hamiltonian framework. This also
provides an abstract way to guarantee energy conservation
(in the case \(\Fcont \equiv 0\)). Consider the total energy over the domain as
given by the Hamiltonian
\begin{equation}\label{eq:H_sl}
\Hcont[h,u] = \int_\Omega\left( h K[u] + \frac{g}{2} (h+b)^2 \right).
\end{equation}
Furthermore, introduce a skew-symmetric operator $\Jcont$ given formally by
\begin{equation*}\label{eq:J_sl}
\Jcont[h,u] =
\begin{pmatrix}
0 & -\div\\
-\grad & -q[h,u]\,\hat{k}\cross 
\end{pmatrix}\;.
\end{equation*}
In the following, we abbreviate the solution variables by $V = (h,u)$. Furthermore, we
endow the solution space \(\mathcal{X} = L^2(\Omega)\times L^2(\Omega, \R^3)\) by its canonical
Hilbert space structure. We denote the inner product by
\[
\inner{V,W}_{\mathcal{X}} = \int_\Omega (v^h w^h + v^u\vdot w^u) \de x,
\]
where \(V = (v^h, v^u)\) and \(W = (w^h, w^u)\) are elements of \(\mathcal{X}\).
The shallow water equations can then be formed using the
functional derivative of $\Hcont$ given as
\begin{equation}
\label{eq:fdH}
\delta \Hcont[V]
= \fdv{\Hcont}{V}[V]
=
\begin{pmatrix}
K[u] + g(h+b)\\
h u
\end{pmatrix}\;,
\end{equation}
which is identical to the Hilbert space gradient of the energy
functional.
In detail, for a perturbation \(W = (w^h, w^u)\) the directional derivative of
 the Hamiltonian \(\Hcont\) (if it is well-defined) fulfills
\[
\Hcont'[V;W]
= \frac{\de}{\de\tau} \Hcont[V + \tau W]\Big\rvert_{\tau=0}
= \int_\Omega \left((K[u] + g(h+b)) w^h + h u \vdot w^u\right)
= \inner*{\fdv{\Hcont}{V}[V], W}_{\mathcal{X}}.
\]
The first identity shows that~\eqref{eq:fdH} indeed gives the functional derivative from
the calculus of variations, the second identity gives the interpretation as a gradient
with respect to the space \(\mathcal{X}\).
Note that we abbreviate the functional derivative by
$\delta \Hcont [V]$, since the argument of differentiation is clear.
In the following, we will also need the Jacobian of the functional derivative of the
Hamiltonian (the Hessian), denoted by $\delta^2 \Hcont[V]$ where
\(\delta^2 \Hcont [V]\,W = \delta \Hcont'[V;W]\).

Then, we interpret the boundary conditions as incorporated into the solution
space, and  obtain~\eqref{eq:sw} in abstract form as
\begin{equation}\label{eq:sw_ham}
\partial_t V = \Jcont[V] \, \delta\Hcont[V] +
\begin{pmatrix}
0 \\
\Fcont[h,u]
\end{pmatrix}
\;.
\end{equation}
In this work we will often appeal to the Hamiltonian framework to formulate the main
ideas in a concise way. We note that the formal continuous description given above serves
also serves as a motivation for the
employed discrete scheme introduced below, which inherits the Hamiltonian structure.
However, all of the developments can also be carried out without this formalism (and
transferred to different discretization schemes under certain assumptions), but in a less
direct way. Conversely, the main ideas apply also to different sets of equations, provided
they can be written in terms of this framework.
 
\subsection{Extension to multiple layers}\label{sec:ml}
In order to model discrete stratification of bodies of water, multiple layers can be
stacked on top of each other with each layer being modeled by its own set of shallow water
equations. Although layered stratification does not fully represent continuous
stratification, it is an efficient and robust model for describing ocean dynamics. In the
multilayer system, each layer's density is set a priori and is considered the average
density of the layer. The density is increasing with water depth in order to produce a stable
configuration. Knowing the density $\rho$ of each layer a priori gives a convenient set of
Lagrangian coordinates, namely the isopycnal contours that separate the layers with
different densities.
This gives what is known as isopycnal coordinates.

\begin{figure}
\centering
\def\svgwidth{0.75\textwidth}
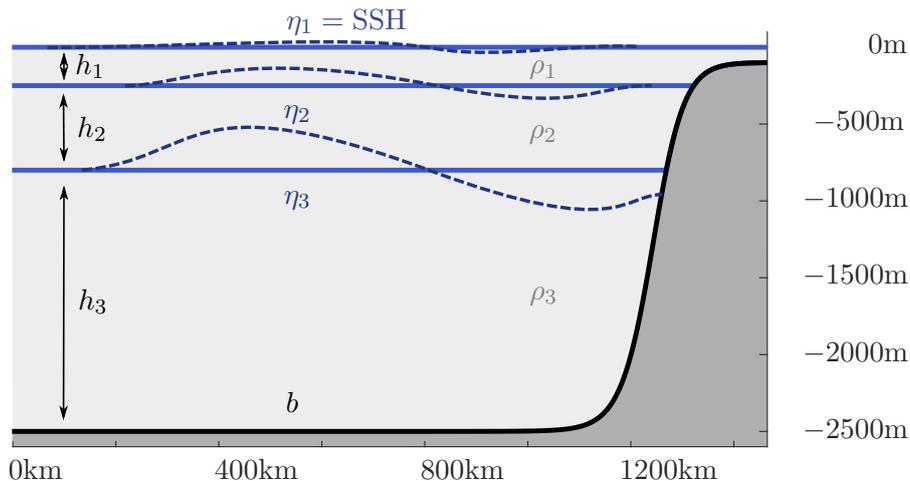
\caption{Visualization of a three layer configuration on the SOMA bathymetry.}
\label{fig:layers}
\end{figure}

The rotating shallow water equations~\eqref{eq:sw} can easily be extended to deal with
a body of water with multiple layers, using isopycnal coordinates. Consider a water basin
that is separated into $L$ layers, and posses $L$ distinct densities $\rho_k$, with
$\rho_1$ being the density for the top layer and $\rho_{k-1} < \rho_{k}$ for
$1 \leq k \leq L$ and the zeroth density defined as \(\rho_0 = 0\) for
convenience (which represents the density of air, much smaller than that of water).
The velocities \(u_k\) and \(h_k\) are now defined for each layer, and the
corresponding layer coordinates \(\eta_k\) are defined by
\[
\eta_{L+1} = b,
\;\;\eta_k[h] = b+\sum_{l=k}^{L} h_l,\;\; k = 1,\dotsc,L.
\]
A visualization is given in Figure~\ref{fig:layers}.
In the following, we denote by \(u\), \(h\) and \(\eta\) the vectors containing all layer
variables, and by \(V = (h, u)\) and \(V_k = (h_k, u_k)\) the combined solution variable.
The solution space is \(\mathcal{X}^L = L^2(\Omega)^L\times L^2(\Omega, \R^3)^L\), which
is endowed with the canonical inner product.
The multilayer rotating shallow water equations for $L$ layers are now given by 
\begin{equation}\label{eq:sw_ml}
\begin{curlyeq}
\partial_t h_k &= - \div{(h_k u_k)} &&\text{in}\;\;\Omega\;, \\
\partial_t u_k &= - \grad{\left(K[u_k] + (g/\rho_k) p_k[h] \right)}
 - q[h_h,u_k]\,\hat{k}\cross(h_k u_k)
 + \Fcont_k(h,u) &&\text{in}\;\;\Omega\;,\\
\end{curlyeq}
\end{equation}
where the main modification with respect to the single-layer case is in the pressure term,
which is defined in each layer as
\[
p_k[h]
= \rho_k \eta_{k+1}[h] + \sum_{l=1}^k \rho_l h_l
= \rho_k \eta_{k}[h] + \sum_{l=1}^{k-1} \rho_l h_l.
\]
The Hamiltonian framework can be extended for the multilayer case (see, e.g.,
\cite{stewart2016energy}). The multilayer Hamiltonian $\Hcont$ can be defined as
\begin{equation}
\label{eq:ham_ml}
\Hcont(h,u)
= \sum_{k=1}^{L} \rho_k\int_\Omega\left( h_k K[u_k]
  + g h_k\left(\eta_{k+1}[h] + h_k/2\right) \right) \de x\;.
\end{equation}
The skew-symmetric operator $\Jcont[h,u]$, consists of \(L\) copies of
\(\Jcont[h_k,u_k]\), represented schematically by a block diagonal matrix
\[
\Jcont[V] = \diag_{k=1,\ldots,L}\frac{1}{\rho_k}\Jcont[h_k,u_k],
\]
with the scaled single-layer operator on the diagonal.
In a weak formulation this corresponds to
\begin{align*}
  (W, \Jcont[V] Y)_{\mathcal{X}^L} &= \sum_{k=1}^L \frac{1}{\rho_k} (W_k ,\Jcont_k [V_k] V_k)_{\mathcal{X}} \\
&= - \sum_{k=1}^L \frac{1}{\rho_k}
  \int_\Omega \left( w^h_k \div y^u_k + w^u_k \vdot \grad y^h_k + q[h_k,u_k] \, w^u_k
  \vdot \hat{k} \cross y_k^u) \right) \de x,
\end{align*}
where \(W = (w^h_k,w^u_k)_{k=1,\dotsc,L}\) and \(Y = (y^h_k,y^u_k)_{k=1,\dotsc,L}\).
The concrete form of~\eqref{eq:sw_ml} can now be derived from~\eqref{eq:sw_ham}, as
before. Note that the multilayer case contains the single-layer case for the special choice
of one layer with arbitrary density.

Alternatively, another version of the multilayer Hamiltonian can be given as
\begin{equation}
\label{eq:ham_ml_easy}
\begin{aligned}
\widetilde{\Hcont}(h,u)
&= \sum_{k=1}^{L} \int_\Omega\left( \rho_k h_k K[u_k]
  + \frac{g}{2} \, \Delta \rho_k \, (\eta_k[h])^2 \right), 
\end{aligned}
\end{equation}
where \(\Delta\rho_k = \rho_k - \rho_{k-1}\) for \(k = 1,\ldots,L\).
Further structure can be exposed by introducing the summation matrix \(T\) with entries
\[
T_{k,j} = \begin{cases} 1 & \text{if } k \leq j, \\
0 & \text{else}.
\end{cases}
\]
It allows to express the layer coordinates as \(\eta_k = b + (T h)_k\), where the
matrix \(T\) operates on the layer variables in an obvious way. More abstractly, we also
write \(\eta = b + Th\). This allows to rewrite the
last term in~\eqref{eq:ham_ml_easy} as
\[
\frac{g}{2} \sum_{k=1}^L \Delta\rho_k (\eta_k[h])^2
= \frac{g}{2} \norm{b + T h}^2_{\Delta\rho},
\]
where \(\norm{\eta}^2_{\Delta\rho} = \sum_{k=1}^L \Delta\rho_k \eta_k^2\) corresponds to a
weighted Euclidean norm. Taking the functional derivative of this term, we obtain 
\(g \, T^\top\diag(\Delta\rho)(T h + b) = g \widetilde{p}_k[h]\), and the corresponding pressure
can be rewritten as:
\begin{align*}
\widetilde{p}_k[h] 
= \rho_1\eta_1[h] + \sum_{l=2}^k (\rho_{l} - \rho_{l-1})\eta_l[h]
= \rho_{k} \eta_{k}[h] + \sum_{l=1}^{k-1} \rho_l h_l
= p_k[h].
\end{align*}
Thus, both Hamiltonians lead to the same pressure and it holds \(\delta\Hcont =
\delta\widetilde\Hcont\). Consequently,
both Hamiltonians are equal up to a constant
value (i.e.\ \(\widetilde{\Hcont} \equiv \Hcont + \textit{const}\)).

Since \(\Delta \rho_1 = \rho_1\) is typically much larger than \(\Delta\rho_k
= \rho_k - \rho_{k-1}\) for \(k > 1\), the pressure differences induced by the free
surface are much larger than the pressure differences stemming from perturbations of the
internal layers. This gives rise to the well-known separation of vertical modes into a
fast barotropic mode, and the remaining slow baroclinic modes; see, e.g.,~\cite{Dukowicz:2006}.
We give an independent exposition that is relevant for the development of the
paper in the next section.

\subsection{Linearization of the model and modes}
\label{sec:lin_and_modes}

We perform a linearized perturbation analysis of~\eqref{eq:sw_ml} for \(\Fcont = 0\),
in order to understand the structure of the fastest modes of certain
linearizations of~\eqref{eq:sw_ml}.
This will be used later to develop 
appropriate linear operators to be used for the exponential time integrators.
For a more in-depth analysis of the fast barotropic
mode arising in ocean models; cf.\ also~\cite{Dukowicz:2006}.
Using the Hamiltonian formalism, the linearized equation for a perturbation
\(V = \reference{V} + W\) can be written as
\begin{equation}\label{eq:sw_ham_lin}
\partial_t W
= \Jcont'[\reference{V}; W]\delta\Hcont[\reference{V}]
+ \Jcont[\reference{V}]\delta^2 \Hcont[\reference{V}]\, W,
\quad W(0) = W_0
\end{equation}
by an application of the product rule, recalling the convention
\(\delta^2 \Hcont[V] W = (\delta\Hcont)'[V;W]\).
The first term contains the
derivatives of \(\Jcont\) with respect to \(V\) and is given for any perturbation \(W\) as
\[
\Jcont'[\reference{V};W] = \diag_{k=1,\dotsc,L}\frac{1}{\rho_k}  \begin{pmatrix}
0 & 0\\
0 & -q'[\reference{V};W] \, \hat{k} \cross 
\end{pmatrix}\;.
\]
This, in turn, contains the derivatives of the
potential vorticity with respect to the solution variables, which are given as
\(q'[V; W] = -(\hat{k}\vdot\curl{u} + f)h^{-2} \, w^h +
(\hat{k}\vdot\curl{w^u})/h\), where \(W = (w^h,w^u)\).

In the following, we linearize around a zero flow, i.e.,
\(\reference{V} = (\reference{h},\reference{u})\) with \(\reference{u} = 0\). Then,
\eqref{eq:sw_ham_lin} simplifies to
\begin{equation}
\label{eq:sw_ham_lin_zero}
\partial_t W
= \Jcont[\reference{V}]\delta^2 \Hcont[\reference{V}]\, W,
\qquad\text{for } \reference{V} = (\reference{h}, 0) \text{},
\end{equation}
since the derivative of \(\Jcont\) contains only entries in the lower right
block (containing the derivatives of the potential vorticity),
which are multiplied by the second entry of \(\delta \Hcont[\reference{V}]\),
which is given by \(\rho_k \reference{u} \reference{h} = 0\).
This system has again Hamiltonian structure, with a fixed
\(\mathcal{J}\)-operator and a quadratic approximation to the energy:
\begin{align*}
\reference{\Jcont} &= \Jcont[\reference{V}]
 = \diag_{k=1,\dotsc,L}\;\frac{1}{\rho_k}\begin{pmatrix}
     0 & -\div\\
-\grad & - (f/\reference{h}_k) \, \hat{k} \cross
   \end{pmatrix}, \\
\quad
 \reference{\Hcont}(W) &:= \frac{1}{2}\inner{W, \delta^2 \Hcont[\reference{V}]\, W}_{\mathcal{X}^L} 
  = \frac{1}{2}\sum_{k=1}^L \int_\Omega \left(\rho_k \reference{h}_k \abs{w_k^u}^2
                       + g \Delta\rho_k (T w^h)^2_k \right),
\end{align*}
where \(W = (w^h,w^u)\).
Thus, \eqref{eq:sw_ham_lin_zero} reads as
\begin{equation}\label{eq:sw_lin_ml}
\begin{curlyeq}
\partial_t w^h_k &= - \div{(\reference{h}_k w^u_k)} &&\text{in}\;\;\Omega\;, \\
\partial_t w^u_k &= - (g/\rho_k)\grad{ \left(T^\top\diag(\Delta\rho) T w^h\right)_k}
 - f\,\hat{k} \cross w^u_k &&\text{in}\;\;\Omega\;.\\
\end{curlyeq}
\end{equation}
Under the simplifying assumption that the Coriolis term and the bathymetry \(b\) are flat,
i.e.\ \(f \equiv \textit{const}\), \(b \equiv \textit{const}\), the eigenmodes
of~\eqref{eq:sw_lin_ml} can be easily computed for the stable reference configuration
\(\reference{h} = h^0\). This reference configuration sets the heights
\(h_k^0\) in such a way that \(\eta_k[h^0] = \max(b, \eta_k^0)\), where 
\(\eta_k^0 \in \R_{<0}\) is a decreasing sequence of constant negative values with
\(\eta_1^0 = 0\), which is the total sea-surface height (SSH) at rest;
see Figure~\ref{fig:layers}. Note that this configuration is just dependent upon
the choice of the total layer volumes \(\int_\Omega h_k\).
Clearly, if the bathymetry is constant, the same holds for the reference heights
\(h^0_k\). Now the eigenmodes \(W^\lambda = (h^{\lambda}, u^{\lambda})\) associated to
the imaginary eigenvalue \(\lambda = \pm \i \abs{\lambda}\) of \(\reference{\mathcal{A}} =
\reference{\Jcont}\delta^2\reference{\Hcont}\) can be grouped into two sets: First, there are stationary
(geostrophic) modes, which are obtained by setting \(\lambda = 0\) in the eigenvalue
equation. Secondly, there are instationary modes, which can be shown (by taking the
\(\div\) and \(\hat{k}\vdot\curl\) of the momentum equation in~\eqref{eq:sw_lin_ml} and
algebraic manipulations) to solve
\[
-g \diag(h^0/\rho) R \Lap h^{\lambda} = \left(\abs{\lambda}^2 - f^2\right) h^{\lambda},
\quad\text{and}\quad
\left(f \, \hat{k} \cross \vdot\; + \lambda\right) u^{\lambda} = - g\diag(1/\rho) R \grad h^{\lambda},
\]
where \(R = T^\top \diag({\Delta\rho}) T\) is the matrix arising from the
layer coupling through the pressure term.
By using the fact that the layer operator and
the spatial Laplacian commute, one can further decouple the above eigenvalue problem to
obtain
\begin{align*}
&\abs{\lambda}^2 - f^2 = \mu^{\textrm{vert}}\mu^{\textrm{horiz}} > 0,
\qquad h^{\lambda} = h^{\mu,\textrm{vert}}\, h^{\mu,\textrm{horiz}} \\
\text{with}\quad
&g\diag(h^0/\rho)R\, h^{\mu,\textrm{vert}} = \mu^{\textrm{vert}} h^{\mu,\textrm{vert}}
\quad\text{and}\quad
-\Lap h^{\mu,\textrm{horiz}} = \mu^{\textrm{horiz}} h^{\mu,\textrm{horiz}},
\end{align*}
where \(h^{\mu,\textrm{horiz}} \in L^2(\Omega)\) with
\(\int_\Omega h^{\mu,\textrm{horiz}} = 0\) and \(h^{\mu,\textrm{vert}} \in \R^L\).
Due to the fact that variations of the density \(\Delta\rho\) are much smaller than
density itself, the layer matrix
\begin{equation}
\label{eq:Rhat}
\diag(1/\rho) R = \diag(1/\rho) T^\top \diag(\Delta\rho) T =
\begin{pmatrix}
1 & 1 & \cdots  & 1 \\
\rho_1/\rho_2 & 1 & \cdots  & 1 \\
\vdots & \vdots & \ddots & \vdots \\
\rho_1/\rho_L & \rho_2/\rho_L & \cdots & 1
\end{pmatrix}\;
\end{equation}
can be well-approximated by a rank-one matrix, e.g. the matrix with all  entries equal to
one. Therefore, the largest
mode of the vertical eigenvalue problem and the associated \(u\)-mode fulfill
approximately
\[
h^{\max,\textrm{vert}}_k \approx \frac{h^0_k}{-b},
\quad
u^{\max,\textrm{vert}}_k = (1/\rho_k)(Rh^{\max,\textrm{vert}})_k \approx 1.
\]
This leads to the well-known (fast) barotropic free-surface mode with wave-speed
\(\sqrt{\mu^{\max,\textrm{vert}}} \approx \sqrt{-b g}\), corresponding to uniformly contracting and expanding
layers, and approximately constant velocities in the vertical. The remaining modes are
associated to (relatively slow) baroclinic modes, which approximately correspond to internal layer
perturbations leaving the free surface constant.

Note that the vertical eigenvalue problem can be rewritten in terms of the generalized
eigenvalue problem for the
vertical \(\eta\)-mode \(\eta^{\mu,\textrm{vert}} = T h^{\mu,\textrm{vert}}\) as
\[
g\diag(\Delta \rho) \, \eta^{\mu,\textrm{vert}}
= \mu^{\textrm{vert}} \,D^{\top}\diag(1/h^0)D \, \eta^{\mu,\textrm{vert}},
\]
where \(D = T^{-1}\) is the discrete difference matrix with entries
\[
T^{-1}_{k,j} =
D_{k,j} = \begin{cases}
1 & \text{if } k = j, \\
-1 & \text{if } k = j-1, \\
0 & \text{else}.
\end{cases}
\]
This formulation relates the vertical layer-modes to discrete approximations of solutions
to the generalized eigenvalue problem from, e.g., \cite{Chelton_etal:1998,Dukowicz:2006}.


\section{Discretization by the TRiSK scheme}
\label{sec:TRiSK}

The rotating shallow water equations~\eqref{eq:sw} and the multilayer
version~\eqref{eq:sw_ml} will be discretized by a mimetic
scheme. This ensures that properties of the continuous equation, such as energy conservation,
are preserved on the discrete level.
In this paper, we will employ the
mimetic TRiSK scheme~\cite{TRiSK:2009,ringler2010unified} (see also~\cite{thuburn2012framework}).
In the following, we briefly introduce a high level notation for employed
differential operators that we will use to describe and analyze the time integration
schemes. For a detailed exposition, we refer to the literature above.

Since the employed scheme is only developed in the literature for unbounded domains, we
will in the following assume that \(\Omega = \mathbb{S}^2\). This also simplifies the
notation. Comments on the adaptation to a bounded domain can be found in Appendix~\ref{app:TRiSK}.

\subsection{Discrete quantities and notation}

The spatial discretization is defined on staggered C-grid that is comprised of spherical
(centroidal) Voronoi tessellations, serving as the primal grid, and a Delaunay triangulation
serving as the dual grid. The discrete quantities are defined at
different locations on the grid, such as the edges, cell centers, and cell vertices. The
edges of the grid will be denoted by $e\in \grid{E}$ (associated to the point of
intersection of primal and dual grid edge \(x_e\)), primal cell grids will be denoted by
$i \in \grid{I}$ (associated to cell centers \(x_i\)),
and the primal cell vertices will be denoted with $v\in \grid{V}$ (associated to the
circum-center \(x_v\) of a dual grid triangle, which is required to lie inside the triangle).
Discrete quantities are denoted by bold vectors, and lie in corresponding cell, vertex, or
edge space \(X_{\grid{I}} = \R^{N_{\grid{I}}}\), \(X_{\grid{V}} = \R^{N_{\grid{V}}}\), and \(X_{\grid{E}} = \R^{N_{\grid{E}}}\),
respectively. Discrete quantities are denoted in the following by bold symbols. The scheme
is built upon the fundamental interpretation of these quantities as piece-wise constant on
the primal or dual cell, and the edge degrees of freedom are associated to a flow across a
interior primal edge (from one primal cell to another, in edge normal direction) or across a dual
edge. Thus for a continuous vector field \(y\), we have \(\disc{y}_e \approx
n_e \vdot y(x_e)\), where \(n_e\) is the geodesic normal to the primal edge \(e\).
Corresponding inner products on these spaces are given by
\[
\inner{\disc{y},\disc{\phi}}_{\grid{I}} = \sum_{i\in \grid{I}} A_i \disc{y}_i\disc{\phi}_i,
\quad
\inner{\disc{y},\disc{\phi}}_{\grid{V}} = \sum_{v\in \grid{V}} A_v \disc{y}_v\disc{\phi}_v,
\quad\text{and }
\inner{\disc{y},\disc{\phi}}_{\grid{E}} = \sum_{e\in \grid{E}} A_e \disc{y}_e\disc{\phi}_e,
\]
respectively, where \(A_i\) denotes the area of a primal cell, \(A_v\) the area of a
dual cell, and \(A_e = l_e d_e\) the area of the square with side lengths given by the
lengths of primal and dual edges (\(l_e\) and \(d_e\)). Note that the sum of the ``edge
areas'' \(A_e\) corresponds to two times the volume of the domain, which is a peculiarity of this
scheme, and corresponds to the fact that the velocities encode only one direction of the
flow (edge normal).

The differential operators are built upon the fundamental relations that for any discrete
variable \(\disc{y} \in X_{\grid{E}}\), and test functions \(\disc{\phi} \in X_{\grid{I}}\) and
\(\disc{\psi} \in X_{\grid{V}}\) we have
\[
\inner{\disc{\phi}, \Ddiv \disc{y}}_{\grid{I}}
= \sum_{i \in \grid{I}} \sum_{e \in \grid{EoI}(i)} \disc{\phi}_i n_{e,i} l_e \disc{y}_e,\quad
\inner{\disc{\psi}, \Dcurl \disc{y}}_{\grid{V}}
= \sum_{v \in \grid{V}} \sum_{e \in \grid{EoV}(v)} \disc{\psi}_v t_{e,v} d_e \disc{y}_e,\quad
\]
where \(\grid{EoI}\) and \(\grid{EoV}\) denote the edges adjacent to each primal or dual cell,
respectively, and \(n_{e,i}\), \(t_{e,v} \in \{\,+1,-1\,\}\) encodes the sign convention
used for the direction of the edge normal velocity \(\disc{y}_e\).
For further details, we refer to~\cite{ringler2010unified}.
We note that in~\cite{thuburn2012framework}, the schemes are built upon the integrated
quantities \(Y_e = l_e \disc{y}_e\) and \(\tilde{Y}_e = d_e \disc{y}_e\), whereas we
follow the convention used in~\cite{ringler2010unified}.

Moreover, a discrete gradient is defined for a cell-wise quantity on the primal grid
(across a primal edge): For each \(\disc{\phi} \in X_{\grid{I}}\) and test
function \(\disc{y} \in X_{\grid{E}}\) we set
\[
\inner{\Dgrad \disc{\phi}, \disc{y}}_{\grid{E}} = -\inner{\disc{\phi}, \Ddiv \disc{y}}_{\grid{I}},
\]
which mirrors the continuous integration by parts formula.
Similarly, a perpendicular gradient \(\Dgradp\) can be defined on \(X_{\grid{V}}\) (across a dual
edge). Additionally, we define the interpolation operators
\[
\inp{\disc{\phi}}_{\grid{E}} \in X_{\grid{E}}, \quad 
\inp{\disc{\psi}}_{\grid{E}} \in X_{\grid{E}},
\]
for \(\disc{\phi} \in X_{\grid{I}}\) and function \(\disc{y} \in X_{\grid{E}}\) that average the values of
the two adjacent primal and dual cells to the corresponding edge.
Interpolation operators from edges to cells are defined by transposition as
\[
\inner{\inp{\disc{y}}_{\grid{I}}, \disc{\phi}}_{\grid{I}} = \inner{\disc{y}, \inp{\disc{\phi}}_{\grid{E}}}_{\grid{E}},
\quad
\inner{\inp{\disc{y}}_{\grid{V}}, \disc{\psi}}_{\grid{V}} = \inner{\disc{y}, \inp{\disc{\psi}}_{\grid{E}}}_{\grid{E}}.
\]
We refer to~\cite{ringler2010unified} for the concrete expressions.
Finally, a reconstruction of tangential velocities is needed (for the implementation of
the perpendicular velocity \(u^\perp\)). This is realized by the reconstruction operator
\[
\Dkcross\colon X_{\grid{E}} \to X_{\grid{E}},
\quad\text{with}\quad
(\Dkcross \disc{y})_e \approx t_e \vdot y(x_e)
\]
for any \(\disc{y} \in X_{\grid{E}}\) representing a continuous vector field \(y\), with \(t_e =
\hat{k} \cross n_e\) the tangent to the primal edge~\(e\).
We refer to~\cite{TRiSK:2009} for a derivation and the concrete expressions,
which in particular ensure the skew-symmetry of the reconstruction operator
\(\Dkcross\) on the edge space \(X_{\grid{E}}\).
For convenience, the specific form of all required operators is also summarized
in Appendix~\ref{app:TRiSK}.

\subsection{Discrete multilayer equations}
\label{sec:TRiSK_multilayer}
We describe the scheme for the general multilayer case, which contains the single-layer
rotating shallow water equations as a special case.
The prognostic variables of the equations are the fluid heights \(\disc{h}_k \in X_{\grid{I}}\) and
the velocities \(\disc{u}_k \in X_{\grid{E}}\), where the degree of freedom for the edge encodes
the (point-wise) velocity in primal cell normal direction.
Diagnostic quantities are the kinetic energy and the potential vorticity, defined by:
\begin{align*}
  \disc{K}[\disc{u}] &= (1/2)\inp{\disc{u} \ast \disc{u}}_{\grid{I}} = \inp{\disc{u}^2/2}_{\grid{I}},
 &&\text{(kinetic energy in } X_{\grid{I}}\text{)}\\
  \disc{q}[\disc{h},\disc{u}] &= \left(\Dcurl \disc{u} + \disc{f}\right) / \inp{\disc{h}}_{\grid{V}}
 &&\text{(potential vorticity in } X_{\grid{V}}\text{)}.
\end{align*}
Here, \(\ast\) denotes the point- or entry-wise product (Hadamard product), and \(/\) the
point-wise division and \(\disc{f} \in X_{\grid{V}}\) is an interpolant of the Coriolis parameter.

The discrete equations are now given as:
\begin{equation}
\label{eq:sw_disc}
\begin{curlyeq}
\partial_t \disc{h}_k &= - \Ddiv{\left(\inp{\disc{h}_k}_{\grid{E}} \ast \disc{u}_k\right)} &&\text{in}\;\; X_{\grid{I}}, \\
\partial_t \disc{u}_k &= \disc{Q}[\disc{h}_k,\disc{u}_k] \left(\inp{\disc{h}_k}_{\grid{E}} \ast \disc{u}_k\right) - \Dgrad{\left(\disc{K}[\disc{u}_k]
    + (g/\rho_k)\,\disc{p}_k[\disc{h}]\right) } + \Fdisc_k[\disc{h},\disc{u}] &&\text{in}\;\;X_{\grid{E}}.
\end{curlyeq}
\end{equation}
Here, the pressure is computed as in the continuous case as
\[
\disc{p}_k[\disc{h}] = \rho_k \disc{\eta}_k[\disc{h}] + \sum_{l=1}^{k-1} \rho_l \disc{h}_l,
\quad
\disc{\eta}_k[\disc{h}] = \disc{b} + \sum_{l=k}^L \disc{h}_k,
\]
where \(\disc{b} \in X_{\grid{I}}\) is an interpolant of the bathymetry. 
The operator \(\disc{Q}[\cdot,\cdot]\colon X_{\grid{E}} \to X_{\grid{E}}\) is defined as
\[
\disc{Q}[\disc{h}_k,\disc{u}_k] \disc{y}
= \frac{1}{2}\left(
\inp{\disc{q}[\disc{h}_k,\disc{u}_k]}_{\grid{E}} \ast \left(\Dkcross \disc{y}\right)
+ \Dkcross \big(\inp{\disc{q}[\disc{h}_k,\disc{u}_k]}_{\grid{E}} \ast \disc{y}\big)
\right),
\]
where \(\disc{y} \in X_{\grid{E}}\) is a discrete flux. The construction of this
operator ensures energy conservation; see~\cite{ringler2010unified}. In terms of the
Hamiltonian framework, this follows from the fact that the operator is skew-symmetric:
for any \(\disc{w},\disc{v} \in X_{\grid{E}}\) we have
\(\inner{\disc{v},\, \disc{Q}[\disc{V}_k] \disc{w}}_{\grid{E}} = - \inner{\disc{Q}[\disc{V}_k]
  \disc{v}, \disc{w}}_{\grid{E}}\), using the skew-symmetry of \(\Dkcross\).

Energy conservation follows directly by introducing a discrete Hamiltonian framework
for~\eqref{eq:sw_disc}. We define the combined solution variable as \(\disc{V} =
(\disc{h},\disc{u}) \in X^L = X_{\grid{I}}^L\times X_{\grid{E}}^L\) analogous to the continuous
case. It is endowed with the discrete inner product
\[
\inner{\disc{V}, \disc{W}}_{X^L} = \sum_{k=1}^L (\disc{h}_k, \disc{w}^h_k)_{\grid{I}} + (\disc{u}_k, \disc{w}^u_k)_{\grid{E}},
\]
where \(\disc{W} = (\disc{w}^h_k, \disc{w}^u_k)_{k=1,\dotsc,L}\).
The discrete Hamiltonian has the form
\begin{equation}
\label{eq:ham_ml_disc}
\Hdisc[\disc{V}]
= \frac{1}{2} \sum_{k=1}^L \big( \rho_k \inner{\disc{h}_k, \disc{K}[\disc{u}_k]}_{\grid{I}}
  + g \Delta\rho_k\, \inner{\disc{\eta}_k[\disc{h}], \disc{\eta}_k[\disc{h}]}_{\grid{I}} \big).
\end{equation}
Mirroring the continuous case, the functional derivative of the Hamiltonian fulfills the identity
\[
\inner{\disc{W}, \delta\Hdisc[\disc{V}]}_{X^L}
= \Hdisc'[\disc{V};\disc{W}]
= \sum_{k=1}^L \left(\inner{\disc{w}^h_k, \rho_k\disc{K}[\disc{u}_k] + g \,\disc{p}_k[\disc{h}]}_{\grid{I}}
 + \rho_k\inner{\disc{w}^u_k, \inp{\disc{h}_k}_{\grid{E}} \ast \disc{u}_k}_{\grid{E}}\right),
\]
using that \(\inner{\disc{h}_k, \disc{K}'[\disc{u}_k; \disc{w}^u_k]}_{\grid{I}} =
\inner{\disc{w}^u_k, \inp{\disc{h}_k}_{\grid{E}} \ast \disc{u}_k}_{\grid{E}}\)
and the definition of \(\disc{\eta}\).
Thus, we can write
\begin{equation}
\label{eq:var_hamil}
\delta\Hdisc[\disc{V}] = \begin{pmatrix}
\rho_k \disc{K}[\disc{u}_k] + g \,\disc{p}_k[\disc{h}] \\
\rho_k \inp{\disc{h}_k}_{\grid{E}} \ast \disc{u}_k
\end{pmatrix}_{k=1,2,\dotsc,L}\;.
\end{equation}
From the concrete form of the equations as given above, one can infer the discrete analogue
of the operator $\Jcont$, which is given by
\begin{equation}\label{eq:J_disc}
\Jdisc[\disc{V}]=
\diag_{k=1,2,\dotsc,L}
\frac{1}{\rho_k}
\begin{pmatrix}
0 & - \Ddiv \\
-\Dgrad & \disc{Q}[\disc{h}_k,\disc{u}_k]\\
\end{pmatrix},
\end{equation}
Using the discrete identities for \(\Dgrad\) and \(\Ddiv\), the definition of
\(\disc{Q}\), and the skew-symmetry of \(\Dkcross\), the skew symmetry of \(\Jdisc\) can
be verified by considering a discrete weak formulation.
Together, this shows that~\eqref{eq:sw_disc} can be described by
\begin{equation}
\label{eq:sw_ham_disc}
\partial_t \disc{V}
= \Jdisc[\disc{V}]\delta{\Hdisc(\disc{V})}
+ \begin{pmatrix}
  0 \\
  \Fdisc[\disc{V}]
\end{pmatrix},
\end{equation}
which directly yields energy conservation in the case \(\Fdisc[\disc{V}] = 0\).
Additional source and dissipation terms can be added to the momentum equation in the term
\(\disc{G}\). We detail some particular choices in Appendix~\ref{app:model_forcing}.

\section{Exponential time integration}
\label{sec:ETD}
Exponential integrators or exponential time differencing methods (ETD) are a special class
of time integration methods; see~\cite{hochbruck2010exponential} and the references
therein. We briefly summarize the relevant content for this manuscript, in the context of
the discrete system introduced in~\eqref{eq:sw_ham_disc}.
We will focus only on 
the case without forcing or dissipation, \(\Fdisc \equiv 0\). Additional forcing terms can
be easily added to the following derivation, but are omitted from the derivation, since
they are usually much less stiff than the core ocean dynamics, and will be added back at
the end.

\subsection{Exponential integrators}
\label{sec:exponential_int}
Exponential integrators are based on a splitting of the forcing term into a linear
part, and a remainder. Denote by \(\disc{V}_n \approx \disc{V}(t_n)\) the current solution
at time \(t_n\), \(n = 0,1,2,\dotsc\), and write
\begin{equation}
\label{eq:force_splitting}
\begin{aligned}
\partial_t \disc{V} = \disc{F}[\disc{V}] = \disc{A}_n \disc{V} + \disc{r}_n[\disc{V}]
\end{aligned}
\end{equation}
with the nonlinear remainder defined by
\begin{align*}
\disc{r}_n[\disc{V}]
&= \disc{F}[\disc{V}] - \disc{A}_n \disc{V}.
\end{align*}
Such a splitting is natural for many  problems, where $\disc{F}$ is given as the
sum of a stiff linear, and a nonlinear term, e.g., semilinear parabolic
problems~\cite{hochbruck2005explicit}. However, for the present case a suitable choice of
\(\disc{A}_n\) is less obvious.
Another point of view is to perform an affine linear expansion of \(\disc{F}\) around \(\disc{V}_n\), which
leads to
\begin{equation}
\label{eq:force_splitting_taylor}
\begin{aligned}
\partial_t \disc{V} = \disc{F}[\disc{V}] = \disc{F}[\disc{V}_n] + \disc{A}_n (\disc{V} - \disc{V}_n) + \disc{R}_n[\disc{V}]
\end{aligned}
\end{equation}
with the nonlinear residual defined by
\begin{align*}
\disc{R}_n[\disc{V}]
&= \disc{F}[\disc{V}] - \disc{F}[\disc{V}_n] - \disc{A}_n(\disc{V} - \disc{V}_n)
= \disc{r}_n[\disc{V}] - \disc{r}_n[\disc{V}_n].
\end{align*}
Clearly, both forms only differ in the constant term \(\disc{r}_n[\disc{V}_n]\)
and are thus very similar.
However, the second form immediately suggests to choose \(\disc{A}_n = \disc{F}'[\disc{V}_n]\),
the Jacobian of \(\disc{F}\), which corresponds to a Taylor expansion in~\eqref{eq:force_splitting_taylor}.
This leads to the development of Rosenbrock type methods.

The idea behind ETD methods, more specifically exponential RK methods, is to treat the
(affine) linear and nonlinear part in different ways: the linear term involving \(\disc{A}_n\) will
be treated exactly, using matrix exponentials, and only the remainder will be approximated
by internal stages of the (exponential) RK method.
In particular, an affine linear problem (i.e.\ when \(\disc{R}_n \equiv 0\)) will be
solved exactly under reasonable assumptions on the methods. By now, there is a well-developed
theory of order conditions for such methods, and several classes of appropriate methods
are known; see, e.g., the overview in~\cite{hochbruck2010exponential}.
If \(\disc{A}_n\) can be chosen in a way that the
residual is significantly less stiff than the linear part, then the CFL conditions that
limit the time step size of explicit methods are less restrictive for exponential RK methods.
However, while the Jacobian \(\disc{F}'[\disc{V}_n]\) always constitutes a mathematically optimal choice in
terms of stiffness reduction and accuracy, it is not necessarily the best choice in terms of
practical performance.

In the specific setting with \(\Fdisc \equiv 0\), due to the product structure of
\(\disc{F}[\disc{V}] = \Jdisc[\disc{V}] \delta \Hdisc[\disc{V}]\) we have
\begin{align}
\disc{F}'[\disc{V}_n; \disc{W}] &= \Jdisc'[\disc{V}_n;\disc{W}] \delta \Hdisc[\disc{V}_n] +
\Jdisc[\disc{V}_n] \delta^2 \Hdisc[\disc{V}_n] \disc{W};
\label{eq:Jac_F}
\end{align}
by the product rule; cf.\ section~\ref{sec:lin_and_modes}.
The concrete expressions for the Jacobians of \(\disc{J}\) and \(\delta\disc{H}\) on the
discrete level are given in Appendix~\ref{app:lin_op}.
Instead of the Jacobian at the current time step,
we will consider choices of linear operator that correspond to
Jacobians that are evaluated at a
reference configuration \(\reference{\disc{V}} = (\reference{\disc{h}}, \disc{0})\).
This leads to
\begin{equation}
\label{eq:Abar}
\disc{F}'[\reference{\disc{V}}]
= \Jdisc[\reference{\disc{V}}] \delta^2\Hdisc[\reference{\disc{V}}],
\end{equation}
since the first term in~\eqref{eq:Jac_F} is zero (cf; section~\ref{sec:lin_and_modes}).
Note that the reference point can be chosen differently in each time step, in order to
take updated height variables into account.
This leads to a choice of \(\disc{A}_n = \reference{\disc{A}}_n =
\disc{F}'[\reference{\disc{V}}_n]\), which leads to a structurally simpler and
computationally more efficient linear operator, at the cost of an increased approximation
error.
As we will demonstrate, in the context of global ocean models, this still
captures enough of the fast dynamics to enable stable and accurate simulations with large
time steps.
We note that~\eqref{eq:Abar} has again Hamiltonian structure, which can be
exploited in computations.
Additional approximations of \(\reference{\disc{A}}\), which further decrease the cost of
the practical 
evaluation in the multilayer case, but keep the underlying structure of the linear
operator intact, will be discussed in section~\ref{sec:mode_splitting}.
Note that, if we were to employ Rosenbrock methods, the Jacobian of any
additional nonlinear terms occurring in \(\Fdisc[\disc{V}]\) would need to be included
in \(\disc{A}_n\). Since we use 
approximate Jacobians, any additional forces that are not stiff can be neglected in \(\disc{A}_n\).

Finally, we note that, on the continuous level, the
splitting~\eqref{eq:force_splitting_taylor} with the linear operator~\eqref{eq:Abar} introduced above
corresponds (up to constant terms) to the splitting of the original equations of the form
\begin{align*}
\partial_t h_k  + \div( \reference{h}_k u_k )
 &\;=\; -\div( (h_k - \reference{h}_k) u_k ), \\
\partial_t u_k  +  \grad (1/\rho_k) p_k[h] + f \, \hat{k} \cross u_k
 &\;=\; - u_k \vdot \grad u_k + \Fcont_k(h,u).
\end{align*}
Here, we have used the vector identity \(u \vdot \grad u = \grad{} ( \abs{u}^2/2 )
+ (\hat{k}\vdot\curl u)\,(\hat{k}\cross u)\). In the splitting above,
the terms on the left are linear (affine linear in the case of the pressure)
and correspond to a multilayer rotating wave equation, and the remaining terms on the right
are nonlinear advection terms. Roughly speaking, the former will always be solved exactly
in theory and treated with matrix exponentials in practice, whereas the latter will be
approximated by the internal stages of an exponential Runge-Kutta method.
Therefore, a method based on~\eqref{eq:Abar} can be expected to have no time step
restrictions associated to the wave phenomena, whereas it would likely still be
subject to CFL conditions associated to the advective processes and other physics
contained in \(\Fcont\).

\subsubsection{Approximation of the residual}
To obtain an exponential integrator, the variation of constants formula is
applied to the continuous equation~\eqref{eq:force_splitting}
to obtain the solution at time \(t_{n+1} = t_n + \dt\) as
\begin{align}
\disc{V}(t_{n+1})
&= 
\exp(\dt\, \disc{A}_n)\disc{V}_n
 + \int_{0}^{\dt}\exp((\dt - \tau)\disc{A}_n)\,\disc{r}_n[\disc{V}(t_n+\tau)] \de \tau
\nonumber\\
\label{eq:var_constant}
&= \disc{V}_n + \int_{0}^{\dt}\exp((\dt - \tau)\disc{A}_n)
  \left(\disc{F}[\disc{V}_n] + \disc{R}_n[\disc{V}(t_n+\tau)] \right) \de \tau.
\end{align}
For further details on the derivations in this section we refer to~\cite{hochbruck2010exponential}.
This formula for the exact solution is further approximated by replacing the residual term
(which still depends on the unknown solution) by a polynomial in time given as
\[
\disc{R}_n[\disc{V}(t_n+\tau)]
\approx \sum_{s=2}^{S} \frac{\tau^{s-1}}{({\dt})^{s-1} (s-1)!} \disc{b}_{n,s},
\]
where the coefficients \(\disc{b}_{n,s}\) should approximate the derivatives
\(\de^{s-1}/\de{\tau^{s-1}} \, \disc{R}_n[\disc{V}(t_n+\tau)]\rvert_{\tau = 0}\).
Note that, since \(\disc{R}_n[\disc{V}_n] = 0\), the constant term in the polynomial can be
omitted.
For exponential Runge-Kutta methods these coefficients will be determined
as linear combinations of the residual \(\disc{R}_n\) evaluated at
the internal stages of the method; see~\cite{hochbruck2010exponential}.

Consequently, by inserting the above approximation into the solution
formula~\eqref{eq:var_constant} (see also Proposition~\ref{prop:phi_ode} below)
we obtain one time step of the underlying method as:
\begin{equation}
\label{eq:ETD_final}
\disc{V}(t_{n+1})
\approx \disc{V}_{n+1}
= \disc{V}_{n} + \dt\left(\phi_1(\dt\disc{A}_n)\disc{F}[\disc{V}_n]
+ \sum_{s=2}^{S} \phi_{s}(\dt\disc{A}_n) \disc{b}_{n,s} \right),
\end{equation}
where the \(\phi\)-functions are defined as
\begin{equation}
\label{eq:phi_function}
\phi_{s}(z)
 = \int_0^1 \exp\left((1-\sigma) z\right) \frac{\sigma^{s-1}}{(s-1)!} \de \sigma
 = \sum_{k = 0}^\infty \frac{z^k}{(k+s)!}
\quad\text{for } z \in \C,\; s=1,2,\dotsc.
\end{equation}
In the case \(s=0\), we set \(\phi_0(\cdot) = \exp(\cdot)\).
Note that the above definition of \(\phi_s\) generalizes to matrix arguments either by
replacing \(z\) by a matrix in the above definition, or by applying the matrix functional calculus.
Based on the construction, there is a simple correspondence between \(\phi\)-functions and
inhomogeneous linear equations; cf., e.g., \cite{niesen2012algorithm}:
\begin{proposition}
\label{prop:phi_ode}
Let \(\disc{x} = \sum_{s=0}^S \phi_s(\dt \disc{A})\disc{b}_s\) for
arbitrary \(\disc{b}_s\), \(s=0,1,\dotsc, S\). Then it holds \(\disc{x} = \disc{w}(\dt)\), which is
the terminal value of the solution to the linear differential equation
\begin{equation}\label{eq:phi_ode}
\begin{curlyeq}
\partial_t\disc{w}(\tau) &= \disc{A}\disc{w}(\tau)
 + \sum_{s=1}^S\frac{\tau^{s-1}}{{\dt}^{s-1}(s-1)!}\,\disc{b}_s,\qquad (0 < \tau < \dt)\\
\disc{w}(0) &= \disc{b}_0.
\end{curlyeq}
\end{equation}
\end{proposition}
Certainly, the efficient computation of these matrix functions is important for
the practical success of ETD methods. Since even for sparse \(\disc{A}\) the matrix
\(\phi_s(\dt \disc{A})\) is generally a full matrix, they can not be assembled in practice in the
large-scale context. Therefore, we will employ iterative methods; see
section~\ref{sec:phi_comp}.

\subsubsection{Example methods}
In the following, we will briefly present specific exponential integrators
employed in this work.
In the simplest case, the residual is simply neglected, and the exponential Euler method
is obtained as
\begin{equation}\label{eq:exp_euler}
\disc{V}_{n+1}= \disc{V}_n + \dt\phi_1(\dt\disc{A}_n)\disc{F}[\disc{V}_n].
\end{equation}
In the case that \(\disc{A}_n\) is only an approximation to the Jacobian, this method
is only first-order accurate in \(\dt\), and not attractive for practical
computations.
We remark that this method is second-order for \(\disc{A}_n = \disc{F}'[\disc{V}_n]\),
which highlights the fact that Rosenbrock-ETD methods have different order
conditions.

For the case of approximate Jacobians, e.g., for~\eqref{eq:Abar}, a family of
two-stage, second-order methods (fulfilling the stiff order conditions)
is given by the one-parameter family
\begin{equation}
\label{eq:ETD2}
\begin{aligned}
\disc{v}_{n,2} &= \disc{V}_n + c_2\dt \phi_1(c_2\dt \disc{A}_n)\disc{F}[\disc{V}_n], \\
\disc{V}_{n+1} &= \disc{V}_n + \dt\phi_1(\dt\disc{A}_n)\disc{F}[\disc{V}_n]
  + (\dt/c_2) \phi_2(\dt\disc{A}_n)\disc{R}_n[\disc{v}_{n,2}],
\end{aligned}
\end{equation}
for the parameter \(c_2 \in (0,1]\); see~\cite{hochbruck2005explicit}. In the case
\(c_2 = 1\), we obtain the exponential 
version of Heun's method, while \(c_2 = 2/3\) corresponds to Ralston's method.
More details and the description of a three stage third-order method that will be used in the
computational experiments are given in Appendix~\ref{app:exp_rk}.
A general discussion of higher order methods can be found in, e.g.,
\cite{hochbruck2005explicit,hochbruck2010exponential}.

\subsection{Approximation of the matrix functions}
\label{sec:phi_comp}

As previously mentioned, the most challenging aspect of ETD methods is efficiently
evaluating the matrix functions $\phi_s$. This challenge made ETD methods computational
infeasible for many years after their discovery due to a lack of efficient matrix function
evaluation methods \cite{moler2003nineteen}. However, in recent years more efficient ways
to approximate $\phi_s$ have been found such as Krylov subspace projections
\cite{saad1992analysis}, which can be combined with sub-stepping algorithms
\cite{niesen2012algorithm}, or instead found with Leja-point interpolation
\cite{bergamaschi2004relpm}, or Chebyshev polynomial approximations \cite{Suhov2014}.
Due to the optimality of the matrix polynomials produced by Krylov methods, it is likely that
these methods will provide an advantage over the other approximation methods (requiring
only matrix vector products), therefore
they will be the focus from this point forward. Methods based on rational approximation
are advantageous from a theoretical standpoint and also promising from a practical
standpoint~\cite{haut2015high}. However, in the context of the spatial scheme employed in
this work, an efficient parallel way of solving the large sparse linear systems remains
challenging.

\subsubsection{Polynomial Krylov methods}
\label{sec:krylov}
Krylov subspace methods, or Krylov methods, provide an efficient way to approximate matrix
functions. This is done by projecting the matrix into a Krylov subspace and then
evaluating the function in a much smaller space than the original. Another benefit of
Krylov methods is the fact that the matrix itself is never explicitly required throughout
the method, only it's action upon single vectors.

The Krylov subspace of dimension $M$ for a matrix
$\disc{A}\in\R^{N\times N}$ and a vector $\disc{b}\in\R^{N}$ is defined as
\begin{equation}
K_M(\disc{A},\disc{b}) = \operatorname{span} \{\,\disc{b}, \disc{A}\disc{b},\dotsc,\disc{A}^{M-1}\disc{b}\,\}
= \{\,p(\disc{A})\disc{b}\;|\; p\in \mathcal{P}_{M-1}\}\;
\end{equation}
Essentially, Krylov methods find the optimal polynomial to approximate a
matrix function applied to a matrix applied to a single vector. In the case of the linear
systems arising in linearly implicit methods, the Krylov method approximates a rational
function to form $\disc{x} = (\Id + \dt\disc{A})^{-1}\disc{b}$;
for ETD the expressions $\phi_s(\dt \disc{A})\disc{b}$ are
approximated. For this purpose, an orthonormal basis of \(K_M(\disc{A},\disc{b})\) is
constructed, which is typically done by the Arnoldi process.
The approximation of matrix functions by Krylov methods is well documented in the
literature.
However, we will employ an inner product induced by another matrix, which is usually not
discussed. Thus, we briefly summarize the necessary extensions for this case.

In this work, we will mostly employ linear operators of the form~\eqref{eq:Abar}, which
have the product structure
\begin{equation}
\label{eq:structure_operator}
\disc{A} = \Jdisc\,\delta^2\Hdisc.
\end{equation}
Additionally the matrix \(\Jdisc\) and \(\delta^2 \Hdisc\) have symmetry
properties with respect to the inner product of the space \(X^L = X^L_{\grid{I}}\times X^L_{\grid{E}}\),
namely
\[
\inner{\disc{W},\, \Jdisc \disc{V}}_{X^L} = - \inner{\Jdisc\disc{W},\, \disc{V}}_{X^L},
\quad
\inner{\disc{W},\, \delta^2\Hdisc\, \disc{V}}_{X^L} = \inner{\delta^2\Hdisc\,\disc{W},\, \disc{V}}_{X^L}
\]
for all \(\disc{V},\disc{W} \in X^L\).
To express this in terms of linear algebra, we introduce the symmetric
and diagonal mass matrix
\(\disc{M}_{X^L}\) of the solution space \(X^L\), containing \(L\) copies of the
cell and edge areas on the diagonal.
In terms of linear algebra, we can now reformulate the symmetry properties above as
\begin{equation}
\label{eq:symmetry}
\disc{M}_{X^L} \Jdisc = - \Jdisc^\top \disc{M}_{X^L},
\quad \disc{M}_{X^L} \delta^2\Hdisc = {\delta^2\Hdisc}^\top \disc{M}_{X^L},
\end{equation}
where \(\cdot^\top\) denotes the transpose.
This yields the skew-symmetry of \(\disc{A}\) with respect to the inner product induced by
the symmetric matrix induced by the second variation of the Hamiltonian.
\begin{proposition}
\label{prop:skew}
The operator \(\disc{A}\) given in~\eqref{eq:structure_operator} with~\eqref{eq:symmetry}
fulfills
\[
\disc{M}_H \disc{A} = - \disc{A}^\top \disc{M}_H
\quad\text{where } \disc{M}_H = \disc{M}_{X^L}\delta^2\Hdisc.
\]
\end{proposition}
In order to take advantage of this symmetry, we will describe the following
orthogonalization procedures for a general operator \(\disc{A}\) with respect to the inner
product and norm
\[
(\disc{x}, \disc{y})_{\arnoldi{M}} = \disc{x}^\top \arnoldi{M} \disc{y},
\quad \norm{\disc{x}}_{\arnoldi{M}} = \sqrt{\disc{x}^\top \arnoldi{M} \disc{x}}
\]
induced by another symmetric matrix \(\arnoldi{M}\), which can be chosen as either
\(\disc{M}_X\) (corresponding to the \(L^2\) space \(X^L\)) or \(\disc{M}_H\)
(corresponding to a norm induced by a quadratic approximation of the Hamiltonian). 

The orthonormal basis vectors $\disc{v}_m$, \(m = 1,\ldots,M\) can be found through the iterative Arnoldi process:
\begin{align*}
\widetilde{\disc{v}}_{m+1}& = \disc{A} \disc{v}_{m} - \sum^m_{j=1}\disc{h}_{j,m}\disc{v}_j,
& \disc{h}_{j,m}&=(\disc{v}_j,\disc{A}\disc{v}_m)_{\arnoldi{M}}, \quad j = 1,2,\dotsc,m\;, \\
\disc{v}_{m+1} &= \frac{1}{\disc{h}_{m+1,m}}\widetilde{\disc{v}}_{m+1},
& \disc{h}_{m+1,m} &= \norm{\widetilde{\disc{v}}_{m+1}}_{\arnoldi{M}},
\end{align*}
where $\disc{v}_1=\disc{b}$.
The Arnoldi process can be collectively given by the Arnoldi decomposition
\begin{equation*}
\disc{A}\arnoldi{V}_M
= \arnoldi{V}_M \arnoldi{H}_M + \disc{h}_{M+1,M}\disc{v}_{M+1}\disc{e}^\top_M
= \arnoldi{V}_{M+1}\widetilde{\arnoldi{H}}_M\;,
\end{equation*}
\begin{equation*}
\widetilde{\arnoldi{H}}_M =
\begin{pmatrix*}[c]
\arnoldi{H}_M\\
\disc{h}_{M+1,M}\disc{e}^\top_M
\end{pmatrix*}
\in \mathbb{R}^{(M+1)\times M}\;,
\end{equation*}
where $\arnoldi{V}_M\in \R^{N \times M}$ contains the $M$ orthogonal basis vectors, i.e., 
\(\arnoldi{V}_M^\top \arnoldi{M}\arnoldi{V}_M = \Id_M\),
$\disc{e}_M$ is the $M$-th canonical basis vector with entries $\disc{e}_{M,i}=\delta_{M,i}$, and
$\arnoldi{H}_M\in \mathbb{R}^{M\times M}$ is the Hessenberg matrix given by
\(\arnoldi{H}_M = \arnoldi{V}_M^\top \arnoldi{M} \disc{A} \arnoldi{V}_M\).
Finally, the Krylov approximation of a matrix function $\phi_s(\dt\disc{A})$
(see, e.g., \cite[Section~4.2]{hochbruck2010exponential}) is given by
\begin{equation}
\label{eq:Krylov_approx}
\phi_s(\dt\disc{A})\disc{b}
\approx \arnoldi{V}_M \phi_s(\arnoldi{V}_M^\top \arnoldi{M} \dt \disc{A}\arnoldi{V}_M)\arnoldi{V}_M^\top\arnoldi{M}\disc{b}
= \norm{\disc{b}}_{\arnoldi{M}}\arnoldi{V}_M \phi_s(\dt\arnoldi{H}_M)\disc{e}_1\;,
\end{equation}
where $\disc{e}_1$ is the first canonical basis vector.
Here, $\phi_s(\dt\arnoldi{H}_M)\disc{e}_1$ can be computed using a dense Pad{\'e} approximation
or an exponential of an augmented matrix (see, e.g., \cite{Sidje1998}).

For an operator $\disc{A}$ that is skew-symmetric with respect to \(\arnoldi{M}\), there exists a more
efficient method known as the skew-Lanczos process (see, e.g., \cite{Faber:1984,Greif:2009}). In the situation of
Proposition~\ref{prop:skew}, the Hessenberg matrix produced by the Arnoldi process is
skew-symmetric and tri-diagonal, and the recurrence relation simplifies to the skew-Lanczos process given for
\(m = 1,\ldots, M-1\) by:
\begin{align*}
\widetilde{\disc{v}}_{m+1} &= \disc{A} \disc{v}_{m} - \disc{h}_{m-1,m}\disc{v}_{m-1},
& \disc{h}_{m-1,m} &= - \disc{h}_{m,m-1}, \\
\disc{v}_{m+1} &= \frac{1}{\disc{h}_{m+1,m}}\widetilde{\disc{v}}_{m+1},
& \disc{h}_{m+1,m} &= \norm{\widetilde{\disc{v}}_{m+1}}_{\arnoldi{M}},
\end{align*}
where \(\widetilde{\disc{v}}_1 = \disc{b}\) and \(\disc{v}_{-1}\) and \(\disc{h}_{0,1}\) are
defined as zero, for convenience.
In the case of a tri-diagonal Hessenberg matrix, a
diagonalization can be performed in time \(\mathcal{O}(M^2)\), which also makes a direct
evaluation using the eigen-decomposition of \(\arnoldi{H}_M\) practically efficient.

Thus, the skew-Lanczos process avoids most of the reorthogonalization steps, which reduces
the computational cost of the Arnoldi-method from \(\mathcal{O}(M^2N)\) to \(\mathcal{O}(MN)\).
Therefore, it is preferable to use an appropriate inner product for
computations, if possible.
If such symmetry cannot be found (for instance in the case of a full
Jacobian), an alternative is to use the incomplete orthogonalization method
IOM~\cite{gaudreault2016efficient,vo2017approximating}, which performs orthogonalization
only with respect to the last \(p\) Arnoldi vectors, while maintaining an exponential
asymptotic convergence rate towards the exact solution; see~\cite{vo2017approximating}.

Concerning the convergence behavior of the methods, we note that,
according to the theoretical estimates, an exponential convergence rate of the Krylov
approximation towards the matrix \(\phi\)-function holds; see the overview in
\cite[Section~4.2]{hochbruck2010exponential}.
For instance, in the skew-symmetric case~\cite[Theorem~4]{Hochbruck:1997}, after a
minimum of \(M \geq \dt \abs{\disc{A}}\) iterations, where \(\abs{\disc{A}}\) is the
spectral radius of \(\disc{A}\), the error decreases at an exponential rate.
We note that this error estimate couples the effort for an accurate approximation of the
matrix exponential of \(\disc{A}\)
to a proportional factor of the time step size; cf.\ also section~\ref{sec:impl_krylov}.

\subsection{Artificial numerical dissipation}
\label{sec:stabilization}

Both the modeling concerns and considerations of numerical efficiency favor a
skew-symmetric choice of the linear operator as given in~\eqref{eq:Abar}. In fact, the
dissipation terms contained in \(\Fdisc\) correspond to numerical closure terms and are
usually relatively slow processes (with the possible exception of vertical diffusion in
the case of a very fine vertical discretization, which we do not consider here). Moreover,
horizontal diffusion can even be set up to be perfectly energy conserving, which leads to
the development of the anticipated potential vorticity method; see, e.g.,
\cite{Chen:2011}.
Therefore, a skew-symmetric operator, which conserves a linearized energy, appears the most
reasonable choice and also provides algorithmic benefits.
However, if the methods are employed together with very large step-sizes -- which is the
desired configuration -- the temporal discretization error can lead to a build-up of
spurious energy in high
scales. If no dissipation term is present in the linear operator, this can lead to an
eventual breakdown of the method due to nonlinear interaction over very long simulation
horizons (of several months).

To remedy this, additional diffusion or high-frequency filtering techniques can be employed. In
the following, we describe a simple technique which can be easily analyzed and
ties into the ETD-Krylov approach described above. For the rest of this section, we assume
that \(\disc{A}_n = \disc{A} = \Jdisc\delta^2\Hdisc\), which is skew-symmetric with respect to the
\(\delta^2\Hdisc\) inner product. We
note that this implies that \(\disc{A}\) has purely imaginary spectrum, and the
associated eigen-vectors are \(\delta^2\Hdisc\) orthogonal.
To dampen the high frequencies, we replace any occurrence of a matrix function \(\phi_s(c z)\) appearing in
the scheme by a modified function \(\phi_{s,\gamma}(c, z)\) defined by
\begin{equation}
\label{eq:smth_phi}
  \phi_{s,\gamma}(c,z)
  = \phi_{s}(c\,(z - (-z^2/\gamma^2)^{p})),
\end{equation}
for some fixed \(p \geq 1\) (e.g., \(p = 2\)) and time scale selective parameter
\(\gamma > 0\). Here, $0\leq c \leq 1$ represents the constant appearing in the internal stage of
the ETD-RK method, or \(c = 1\) for the final stage.
This change is motivated by the following result, which is simple to derive and given here
without proof.
\begin{proposition}
Let \(\disc{V}_{\gamma,n}\) be computed by an ETD-RK method, where
each occurrence of \(\phi_s(c \dt \disc{A})\) is replaced by
\(\phi_{s,\gamma}(c,\dt \disc{A})\).
Then \(\disc{V}_{n,\gamma} \approx \disc{V}_\gamma(t_n)\) approximates the
solution of the modified problem
\begin{equation}
\label{eq:dissipation_equation}
\partial_t \disc{V}_\gamma = \disc{F}[\disc{V}_\gamma]
- (\dt^{2p-1}/\gamma^{2p})\,(-\disc{A}^{2})^p\disc{V}_\gamma,
\quad
\disc{V}_\gamma(0) = \disc{V}_0,
\end{equation}
at the same order as the underlying ETD-RK method.
\end{proposition}
Thus, if \(s\) is the convergence order of the underlying ETD-RK scheme,
the modified scheme is of order \(\min\{\,2p-1,s\,\}\). In particular, the order of convergence
is maintained for \(p \geq (s+1)/2\).
We further comment on the structure of the
perturbation term.
Due to the skew-symmetry of \(\disc{A}\), it follows that
\(-\disc{A}^2 = \Jdisc^\top \delta^2\Hdisc \Jdisc \delta^2\Hdisc\)
is a \(\delta^2\Hdisc\)-symmetric positive operator,
and thus the appearance of its \(p\)-th power in~\eqref{eq:dissipation_equation}
dissipates the quadratic energy induced by \(\delta^2\Hdisc\). 
\begin{remark}
Additionally, in the concrete case of \(\disc{A}\) derived as~\eqref{eq:Abar}
from~\eqref{eq:sw_ham_disc} it can be further verified
that \(-\disc{A}^2\) behaves similar to a second-order differential operator (a weighted
negative Laplacian).
In this case, for \(p=2\) the artificial dissipation is given by an additional
biharmonic diffusion with coefficient proportional to \(\dt^{3}/\gamma^{4}\).
\end{remark}

Concerning the numerical implementation, a direct application of the Krylov method
to~\eqref{eq:smth_phi} would lead to a significant increase in computation times
if the Krylov space is constructed for
\(\dt \disc{A}_\gamma = \dt\disc{A} - (\dt/\gamma)^{2p}(-\disc{A}^2)^{p}\), since this requires
additional multiplications by \(\disc{A}\). Instead, we build the Krylov space as before
for \(\disc{A}\), and apply the modified \(\phi\)-function, i.e., 
\begin{equation*}
  \phi_{s,\gamma}(c,\dt \disc{A}) \disc{b}
  \approx \norm{\disc{b}} \arnoldi{V}_M\phi_{s,\gamma}(c,\dt \arnoldi{H}_M) \disc{e}_1
  = \norm{\disc{b}} \arnoldi{V}_M \phi_{s}(c\,\dt (\arnoldi{H}_M - \dt^{2p-1}/\gamma^{2p} (-\arnoldi{H}^2_M)^{p}) \disc{e}_1\;,
\end{equation*}
where $\arnoldi{H}_M$ is the Hessenberg matrix from section~\ref{sec:phi_comp}. In this
way, the additional cost for a Krylov approximation with \(M\) vectors is limited to
computing powers of \(\arnoldi{H}_M\), which is usually negligible.

\subsection{Conservation of mass}
\label{sec:mass_conservation}

The property of a scheme to be exactly mass conserving is a basic and important
requirement for global ocean models. We note that the multilayer TRiSK-scheme is layerwise
mass conserving in continuous time. Most commonly employed  time integration methods such
as explicit Runge-Kutta or implicit methods preserve this property, which makes it
desirable also in the context of exponential integrators. More generally, this corresponds
to the preservation of linear invariants present in the semidiscrete problem. Fortunately,
under a simple requirement on the linear operator, which are fulfilled for the choices
made above, many exponential integrators share this property as well.

We begin by summarizing the mass conserving properties in the multilayer model.
\begin{proposition}
The evolution of the model~\eqref{eq:sw_ml} is layer-wise volume conserving, i.e.,
\[
(\disc{1}, \disc{h}_k(t))_{I} = \textrm{const}
\quad\text{for all } k=1,2,\dotsc,L,
\]
where \(\disc{1}\in X_{\grid{I}}\) denotes the constant one cell-vector.
\end{proposition}
\begin{proof}
The property follows by testing the \(k\)-th component of the mass equation with
\(\disc{1}\) to obtain
\[
\frac{\de}{\de t} (\disc{1}, \disc{h}_k(t))_{I} = (\disc{1}, \partial_t\disc{h}_k(t))_{I}
= (\disc{1}, \Ddiv(\inp{\disc{h}_k}_{\grid{E}}\ast\disc{u}_k))_{\grid{I}}
= -(\Dgrad\disc{1}, \inp{\disc{h}_k}_{\grid{E}}\ast\disc{u}_k)_{\grid{E}} = 0,
\]
using the discrete adjoint relation between \(\Ddiv\) and \(\Dgrad\).
\end{proof}
Consequently, also the total mass \(\sum_k \int_\Omega \rho_k h_k\), given in the discrete
equation by
\begin{equation}
\label{eq:total_mass}
\disc{m}[\disc{h}] = \sum_{k=1}^L \rho_k (\disc{1}, \disc{h}_k(t))_{I},
\end{equation}
is conserved. We note that these conservation properties can be expressed more generally as
linear invariants,
\begin{equation}
\label{eq:linear_invariant}
\frac{\de}{\de t} \inner{\disc{l},\disc{V}(t)}_X
= \inner{\disc{l},\partial_t \disc{V}(t)}_X
= \inner{\disc{l},\disc{F}[\disc{V}(t)]}_X = 0,
\end{equation}
where \((\disc{l}, \cdot)_X\) is a linear functional, represented by testing
with the vector \(\disc{l}\). For instance, in the case of~\eqref{eq:total_mass}, we
choose \(\disc{l} = (\disc{l}^h,\disc{l}^u) = (\disc{\rho},\disc{0})\), where \(\disc{\rho}_k =
\rho_k \disc{1}\).

Under the appropriate assumption on the linear operator \(\disc{A}\), and the underlying
ETD method, linear invariants remain preserved in the time discrete system.
\begin{theorem}
\label{thm:linear_invariant}
Assume that \((\disc{l}, \disc{F}(\disc{V}))_X = 0\) for all
\(\disc{V} \in X\) (which implies the linear
invariant~\eqref{eq:linear_invariant}). Assume further that
\[
(\disc{l}, \disc{A}_n \disc{V})_X = 0
\quad\text{for all } \disc{V} \in X.
\]
Then, the ETD methods presented in this section preserve the same linear invariant.
\end{theorem}
\begin{proof}
We use the explicit formula for the final stage~\eqref{eq:ETD_final} to obtain
\[
\inner{\disc{l}, \disc{V}_{n+1} - \disc{V}_n}_X
= \sum_{s=1}^S \inner{\disc{l}, \phi_s(\disc{A}_n) \disc{b}_s}_X,
\]
where \(\disc{b}_1 = \disc{F}(\disc{V}_n))\) and \(\disc{b}_s\) for \(s>1\) is a linear
combination of the residuals \(\disc{R}_n(\disc{v}^i_n)\) evaluated at the internal stages
\(\disc{v}_n^i\). Now, we directly obtain
\(\inner{\disc{l},\disc{b}_s}_X = 0\) using the properties of \(\disc{F}\) and \(\disc{A}_n\).
Then, testing the linear equation~\eqref{eq:phi_ode} which corresponds to the expressions
involving the \(\phi\)-functions by \(\disc{l}\) reveals that also
\(\inner{\disc{l}, \phi_s(\disc{A}_n) \disc{b}_s}_X = 0\) for all \(s\), which yields
\[
\inner{\disc{l}, \disc{V}_{n+1}}_X = \inner{\disc{l},\disc{V}_n}_X,
\]
which is the desired property.
\end{proof}
Thereby, we directly obtain mass conservation for the specific ETD schemes.
\begin{corollary}
For the choices \(\disc{A}_n = \disc{F}'[\disc{V}_n]\) and
\(\disc{A}_n = \reference{\disc{A}} = \disc{F}'[\reference{\disc{V}}]\), see~\eqref{eq:Jac_F}
and~\eqref{eq:Abar}, we obtain
\[
\inner{\disc{1}, \disc{h}_{k,n+1}}_{\grid{I}} = \inner{\disc{1}, \disc{h}_{k,n}}_{\grid{I}}
\quad k=1,2,\dotsc,L,
\quad\text{and}\quad
\disc{m}[\disc{h}_{n+1}] = \disc{m}[\disc{h}_{n}],
\]
where \((\disc{h}_n,\disc{u}_n) = \disc{V}_n\) are the time steps~\eqref{eq:ETD_final}.
\end{corollary}

\section{Energetically consistent layer reduction}
\label{sec:mode_splitting}

Motivated by the mode analysis of section~\ref{sec:lin_and_modes}, we propose a
layer reduction technique to project the linear operator used in the ETD method to a
subspace corresponding to the fastest modes. Here, we take special care to perform this
projection in an energetically consistent way and to account for invariants responsible
for mass conservation. In this section, we focus attention on the linear operators of the
structure~\eqref{eq:Abar} and let \(\Jdisc = \Jdisc[\reference{\disc{V}}]\) and
\(\delta^2 \Hdisc = \delta^2 \Hdisc[\reference{\disc{V}}]\)
for some reference configuration \(\reference{\disc{V}} =
(\reference{\disc{h}},\disc{0}) \in X\).

\subsection{A general class of linear operators}

At first, we describe a general method for layer reduction, which is based on a projection
preserving the Hamiltonian structure of the linearized equation.
For an ETD method, it is desirable to only approximate the fast modes of the system. Based
on the eigenvalue analysis of section~\ref{sec:lin_and_modes}, we can exploit the fact
that the vertical modes are decreasing in magnitude with finer vertical resolution
(in contrast to the horizontal modes, which are increasing with finer horizontal resolution).

We introduce a reduced-layer space
\[
X^{\laylowL} =  X_{\grid{I}}^{\laylowL}\times X_{\grid{E}}^{\laylowL}
\quad\text{for } 1 \leq \laylowL \ll L,
\]
which is supposed to parametrize the fastest vertical modes of the discrete linear
operator \(\disc{A} = \Jdisc\, \delta^2\Hdisc\) as in~\eqref{eq:Abar}. 
We define a corresponding ansatz by the linear
mapping \(\disc{\Psi} \colon X^{\laylowL} \to X^L\) with the structure
\begin{equation*}
\disc{\Psi} =
\begin{pmatrix}
\disc{\Psi}_h & 0\\
0 & \disc{\Psi}_u\\
\end{pmatrix},
\end{equation*}
Here, the matrix \(\disc{\Psi}_h\) (and similarly, \(\disc{\Psi}_u\)) is defined by
\[
\left[\disc{\Psi}_h \laylow{\disc{h}}\right]_i = \sum_{j=1}^{\laylowL} \disc{\Psi}_h^{j,i} \laylow{\disc{h}}_{k,i},
\quad\text{for all } i\in \grid{I},
\]
and any \(\laylow{\disc{h}} \in X_{\grid{I}}^{\laylowL}\).
Here, the vector \(\disc{\Psi}^{j,i}_h \in \R^L\) should roughly correspond to the \(j\)-th
fastest vertical height mode, which can be different in each cell \(i \in \grid{I}\).
\begin{remark}
Consider the simplified case of a constant Coriolis term and bathymetry, 
i.e.\ \(f \equiv \textit{const}\), \(b \equiv \textit{const}\) with constant reference
heights \(h^0 \equiv \textit{const}\). Let \(\mu_j > \mu_{j+1} > 0\), \(j = 1,2,\dotsc,L\) be the eigenvalues of the matrix
\(A^{\textrm{vert}} = g\diag(h^0/\rho) R\) with \(R\) defined as in~\eqref{eq:Rhat},
corresponding to the eigenvectors
\(w^{j,h,vert} \in \R^L\). These modes correspond to the height variable, whereas the
corresponding modes for the velocity are given as
\(w^{j,u,\textrm{vert}} = \diag(1/\rho) R w^{j,h,\textrm{vert}}\). In this case, we can choose
\(\disc{\Psi}^{j,i}_h = w^{j,h,\textrm{vert}}\) and
\(\disc{\Psi}^{j,e}_u = w^{j,u,\textrm{vert}}\) for all cells and edges. In this particular
case, the following arguments would greatly simplify. However, since \(b\) and \(f\) are
generally not constant, the horizontal modes can only be defined in an
approximate sense, separately in each cell and edge.
\end{remark}

While the ansatz functions \(\disc{\Psi}\) can be chosen freely, in principle, it is
important to obtain appropriate dynamics for the reduced linear system. Here, since we are dealing with a
linear Hamiltonian system (the rotating multilayer wave equation), we will take care to
preserve this structure when deriving the reduced dynamics.
The special case of using only the single fastest barotropic mode, which leads to simple
concrete formulas, will be discussed further in section~\ref{eq:barotropic_ETD}.
We start by defining the reduced linearized Hamiltonian, which arises from
the canonical ansatz
\[
\reference{\laylow{\Hdisc}}(\laylow{\disc{V}}) = \reference{\Hdisc}(\disc{\Psi}\laylow{\disc{V}})
= \frac{1}{2}\inner*{\disc{\Psi}\laylow{\disc{V}}, \delta^2 \Hdisc \, \disc{\Psi}\laylow{\disc{V}}}_{X^{L}}
\]
for \(\laylow{\disc{V}} \in X^{\laylowL}\).
Now, we use additionally the fact that the mass matrix \(\disc{M}_{X^L}\) of the space \(X^L\)
commutes with \(\disc{\Psi}^\top\) in the sense that
\(\disc{M}_{X^L} \disc{\Psi} = \disc{\Psi} \disc{M}_{X^{\laylowL}}\). In fact, for
\(\disc{\Psi}_h\) (and similarly for \(\disc{\Psi}_u\) it simply holds
for any \(\laylow{\disc{h}} \in X_{\grid{I}}^{\laylowL}\) that
\[
\left[\disc{M}_{X_{\grid{I}}^L} \disc{\Psi}_h \laylow{\disc{h}}\right]_i
 = A_i \left[\disc{\Psi}_h \laylow{\disc{h}}\right]_i
= \left[\disc{\Psi}_h \disc{M}_{X_{\grid{I}}^{\laylowL}} \laylow{\disc{h}}\right]_i
\quad\text{for all } i\in \grid{I},
\]
where \(A_i\) it the area of the \(i\)-th cell. This corresponds to the fact that
\(\disc{\Psi}\) operates only on the layer
indices for each cell and edge stack, and that both mass matrices consist of multiple
identical copies of the (diagonal) single-layer mass-matrix.
Thus, it follows that
\[
\reference{\laylow{\Hdisc}}(\laylow{\disc{V}})
= \frac{1}{2} (\disc{M}_{X_{\grid{I}}^{L}} \disc{\Psi} \laylow{\disc{V}})^\top (\delta^2 \Hdisc \, \disc{\Psi}\laylow{\disc{V}})
= \frac{1}{2}\inner*{\laylow{\disc{V}}, \disc{\Psi}^\top \delta^2 \Hdisc \,
  \disc{\Psi}\laylow{\disc{V}}}_{X^{\laylowL}}
= \frac{1}{2} \inner*{\laylow{\disc{V}}, \delta^2 \laylow{\Hdisc} \laylow{\disc{V}}}_{X^{\laylowL}},
\]
with the corresponding reduced (linearized) Hamiltonian matrix defined as
\[
\delta^2 \laylow{\Hdisc}
= \disc{\Psi}^\top \delta^2 \Hdisc\, \disc{\Psi}.
\]
We note that the above projection can be computed separately for each cell- and
edge-stack, since for fixed layer coordinates \(\delta^2 \Hdisc\) is a
diagonal matrix in the horizontal dimensions. We also note that
\(\delta^2 \laylow{\Hdisc}\) shares the same property.

In order to derive a projection operator, we first introduce the transpose matrix of
\(\disc{\Psi}\), denoted by \(\disc{\Psi}^\top\colon X^L \to X^{\laylowL}\).
Note that for the particular setting considered here, it is
identical to the adjoint of \(\disc{\Psi}\), i.e., it fulfills
\[
\inner{\disc{\Psi}^\top \disc{V}, \laylow{\disc{V}}}_{X^{\laylowL}}
= \inner{\disc{V}, \disc{\Psi} \laylow{\disc{V}}}_{X^L},
\]
for any \(\disc{V}\in X^L\) and \(\laylow{\disc{V}} \in X^{\laylowL}\).
Here, we have used again the fact that the mass matrix commutes with \(\disc{\Psi}\).
Although \(\disc{\Psi}^\top\) maps from the full to the reduced space, it is not
appropriate to map solution variables from the full to the reduced space.
Instead, in order to provide an energetically consistent projection of the linearized equation to
the reduced space, we introduce \(\disc{\Psi}^\dagger\), the (generalized) Moore--Penrose
pseudoinverse
\[
\disc{\Psi}^\dagger\colon X^L \to X^{\laylowL},
\quad
\disc{\Psi}^\dagger = (\delta^2 \laylow{\Hdisc})^{-1} \disc{\Psi}^\top \delta^2 \Hdisc.
\]
Hence, $\disc{\Psi}^\dagger$ gives the reduced layer coordinates for any discrete solution
variable. Again, we note that the inverse \((\delta^2 \laylow{\Hdisc})^{-1}\) can be
computed in a cell- and edge-stack wise fashion, reducing the solution to a large
number of \(\laylowL\times\laylowL\) sized systems.
The definition of \(\disc{\Psi}^\dagger\) is motivated by the following derivation:
\begin{proposition}
\label{prop:psi_dagger}
The restriction matrix $\disc{\Psi}^\dagger$ gives the solution to the following minimization problem:
For any \(\disc{V} \in X^L\) we have \(\disc{\Psi}^\dagger \disc{V} = \laylow{\disc{V}}\) where
\begin{align}\label{eq:Gdagger_min}
\laylow{\disc{V}}
 &= \argmin_{\laylow{\disc{V}} \in X^{\laylowL}}\norm*{\disc{V} - \disc{\Psi}\laylow{\disc{V}}}_{\delta^2 \Hdisc}
  = \argmin_{\laylow{\disc{V}} \in X^{\laylowL}} \;\frac{1}{2} \inner{\disc{V} - \disc{\Psi}\laylow{\disc{V}}, \delta^2\Hdisc (\disc{V} - \disc{\Psi} \laylow{\disc{V}})}_{X^L}\;.
\end{align}
\end{proposition}
\begin{proof}
We define the energy norm \(\norm{\disc{V}}_{\delta^2 \Hdisc}
= \sqrt{\inner{\disc{V}, \delta^2\Hdisc\, \disc{V}}_{X}}\) in the canonical way. The derivation is then
standard, by writing out the optimality condition of~\eqref{eq:Gdagger_min}, given by
\[
\disc{\Psi}^\top \disc{M}_{X} \delta^2\Hdisc \, \disc{\Psi} \laylow{\disc{V}}
 = \disc{\Psi}^\top \disc{M}_{X} \delta^2\Hdisc\, \disc{V}.
\]
Now, we use again the fact that the mass matrix \(\disc{M}_{X^L}\) of the space \(X^L\)
commutes with \(\disc{\Psi}^\top\) in the sense that
\(\disc{\Psi}^\top \disc{M}_{X^L} = \disc{M}_{X^{\laylowL}} \disc{\Psi}^\top\).
\end{proof}

Additionally, we introduce the orthogonal projection
\[
\disc{P} = \disc{\Psi}\disc{\Psi}^\dagger \colon X^L \to X^L\;.
\]
By Proposition~\ref{prop:psi_dagger}, the projection is given as \(\disc{P} \disc{V} =
\disc{\Psi} \laylow{\disc{V}}\) , where \(\laylow{\disc{V}}\) is the minimizer
of~\eqref{eq:Gdagger_min}, and thus
minimizes the projection error in the canonical linearized energy norm. 
Based on this choice of the projection, a projected linear operator
\(\disc{A}_P\) can be defined for \(\disc{A} = \reference{\disc{A}} = \Jdisc\,\delta^2\Hdisc\) as
\begin{equation}
\label{eq:projected_A}
\begin{aligned}
\disc{A}_P &= \disc{P} \disc{A} \disc{P} = \disc{\Psi} \laylow{\disc{A}} \disc{\Psi}^\dagger, \\
\text{where}\quad \laylow{\disc{A}} &=  \disc{\Psi}^\dagger \disc{A} \disc{\Psi}.
\end{aligned}
\end{equation}
Here, \(\laylow{\disc{A}} \colon X^{\laylowL} \to X^{\laylowL}\) is the corresponding
reduced layer operator.
Using the properties of \(\disc{\Psi}^\dagger\), we now verify that the operators
\(\laylow{\disc{A}}\) and \(\disc{A}_P\) again have Hamiltonian structure, together with an
  appropriate definition of the reduced \(\Jcont\)-operator.
\begin{proposition}
Define the skew-adjoint operators
\begin{equation}
\begin{aligned}
\Jdisc_P&\colon X^{L} \to X^{L} & \Jdisc_P&= \disc{P} \Jdisc \disc{P}^\top \\
\laylow{\Jdisc} &\colon X^{\laylowL} \to X^{\laylowL} & \laylow{\Jdisc} &= \disc{\Psi}^\dagger \Jdisc (\disc{\Psi}^\dagger)^\top.
\end{aligned}
\end{equation}
Then, it holds that 
\[
\laylow{\disc{A}} = \laylow{\Jdisc}\, \delta^2\laylow{\Hdisc}
\quad\text{and}\quad
\disc{A}_P = \Jdisc_P \, \delta^2\Hdisc_P.
\]
where \(\delta^2\Hdisc_P = \disc{P}^\top \delta^2 \Hdisc \disc{P}\).
\end{proposition}
\begin{proof}
By definition of \(\disc{\Psi}^\dagger\) we have
\((\disc{\Psi}^\dagger)^\top \, \delta^2\laylow{\Hdisc} = \delta^2 \Hdisc \, \disc{\Psi}\) and thus
\[ 
\disc{\Psi}^\dagger \disc{A} \disc{\Psi} = \disc{\Psi}^\dagger \Jdisc \,\delta^2\Hdisc \disc{\Psi}
 = \disc{\Psi}^\dagger \Jdisc(\disc{\Psi}^\dagger)^\top \, \delta^2\laylow{\Hdisc}
 = \laylow{\Jdisc} \, \delta^2\laylow{\Hdisc}
\]
Concerning the second case, we compute
\[
\disc{A}_P = \disc{\Psi}\disc{\Psi}^\dagger \disc{A} \disc{\Psi}\disc{\Psi}^\dagger
= \disc{\Psi}\disc{\Psi}^\dagger\Jdisc (\disc{\Psi}^\dagger)^\top\disc{\Psi}^\top \delta^2\Hdisc \disc{\Psi}\disc{\Psi}^\dagger,
\]
using the previous result. Now we observe that
\((\disc{\Psi}^\dagger)^\top\disc{\Psi}^\top = \disc{P}^\top
= \disc{P}^\top\disc{P}^\top\), owing to the fact that \(\disc{P}\) is a projection.
\end{proof}
Thereby, the structural properties important for stability of the solutions are preserved.
Thus, we can easily employ the derived linear operator in the context of an ETD method, by
simply replacing the operator \(\disc{A}\) by \(\disc{A}_P\). By the
construction the adaptation is straightforward.
However, for the practical use of the method it is important that the
computation of the matrix \(\phi\)-functions can be reduced
to a smaller-size problem, based on the following observation.
\begin{proposition}
\label{eq:prop_exp_Ahat}
Let \(\disc{A}_P = \disc{P}\disc{A}\disc{P}\) and \(\laylow{\disc{A}} = \disc{\Psi}^\dagger
\disc{A} \disc{\Psi}\) be defined as above. Then, it holds for any \(s \geq 0\) that
\[
\phi_s(\disc{A}_P) = \frac{1}{s!}(\Id - \disc{P}) + \disc{\Psi} \,\phi_s(\laylow{\disc{A}})\, \disc{\Psi}^\dagger.
\]
\end{proposition}
\begin{proof}
Due to \eqref{eq:phi_function}, we have
\[
\phi_{s}(\disc{A}_P) = \sum_{k = 0}^\infty \frac{\disc{A}_P^k}{(k+s)!}
= \frac{1}{s!}(\Id - \disc{P}) + \disc{\Psi} \sum_{k = 0}^\infty \frac{(\disc{\Psi}^\dagger\disc{A} \disc{\Psi})^k}{(k+s)!} \disc{\Psi}^\dagger,
\]
using the identities \(\disc{P} = \disc{\Psi} \disc{\Psi}^\dagger\), \(\disc{A}_P^0 = \Id\) and
\(\disc{A}_P^k = \disc{\Psi} (\disc{\Psi}^\dagger \disc{A} \disc{\Psi})^k \disc{\Psi}^\dagger\) for \(k \geq 1\).
\end{proof}
Thereby, the computation of \(\phi_s(\dt\disc{A}_P)\disc{b}_s\) in an ETD method can be reduced to
the computation of \(\phi_s(\dt\laylow{\disc{A}}) \laylow{\disc{b}}_s\), the restriction
\(\laylow{\disc{b}}_s = \disc{\Psi}^\dagger \disc{b}_s\), and another application of the prolongation
\(\disc{\Psi}^\dagger\). In the context of an ETD-Krylov method, this significantly reduces
the required computational work.

\subsection{Implementation of the resulting ETD methods}
Finally, for illustrative purposes, we explicitly write the exponential Euler method using the
projected operator $\disc{A}_P$.
On a high level, we simply replace the operator \(\disc{A}\) by \(\disc{A}_P\)
in~\eqref{eq:exp_euler}.
Using Proposition~\ref{eq:prop_exp_Ahat} it can then be rewritten as
\begin{align*}
  \disc{V}_{n+1}
  &= \disc{V}_n + \dt\,\disc{\Psi}\,\phi_1(\dt\disc{A}_P) \disc{F}[\disc{V}_{n}]\\
  &= \disc{V}_n + \dt(\Id - \disc{P}) \disc{F}[\disc{V}_{n}]
   + \dt\,\disc{\Psi}\,\phi_1(\dt\laylow{\disc{A}})\,\disc{\Psi}^\dagger \disc{F}[\disc{V}_{n}]\;.
\end{align*}
Thus, the method performs an explicit Euler step with the forcing term projected to the
orthogonal complement of the span of \(\disc{\Psi}\), whereas the part of the forcing term
in the space spanned by \(\disc{\Psi}\) is treated with the matrix exponential associated to
the projected matrix $\laylow{\disc{A}}$.
If the matrix $\laylow{\disc{A}}$ is assembled ahead of time, the evaluation of the matrix
exponential of the reduced matrix is more efficient, since the number of
degrees of freedom and the nonzero entries of the projected matrix is much smaller than
the number of degrees of freedom of the original one.
Similarly, the higher order ETD methods can be rewritten in the above
way to a form that is suitable for implementation purposes.

\subsection{Barotropic ETD method}
\label{eq:barotropic_ETD}
Due to the fact that the quotient of the first mode (the fast barotropic mode) and
the second mode (the fastest baroclinic mode)
is usually much bigger than one in realistic global ocean simulations,
the stiffest parts of the linear operator can be captured by a particularly simple choice
of \(\disc{\Psi}\), which exploits the analytical structure of this mode. We note that state-of-the art
global models exploit this splitting as well. In particular, we refer to the widely
used split-explicit scheme; see~\cite{higdon2005two}.
Here, a suitable method arises from a direct application of an exponential integrator with a
particular choice of \(\disc{\Psi}\), which we refer to as barotropic ETD (B-ETD) method.

For a reference configuration \(\reference{\disc{h}} \in X_{\grid{I}}^L\) define the
corresponding total height and average density as
\[
\widehat{\disc{h}} = \sum_{k=1}^L \reference{\disc{h}}_{k},
\qquad
\widehat{\disc{\rho}} = \sum_{k=1}^L
\frac{\rho_k\reference{\disc{h}}_{k}}{\widehat{\disc{h}}}
\in X_{\grid{I}}.
\]
Based on the approximate form of the fastest vertical mode (see section~\ref{sec:lin_and_modes}), 
we consider the concrete choice with \(\laylowL = 1\) given by
\begin{equation}
\label{eq:barotropic_ansatz}
\disc{\Psi} =
\begin{pmatrix}
\disc{y} & 0\\
0 & \disc{1}\\
\end{pmatrix},\;
\quad\text{where }
\disc{1} \equiv 1,\;
\disc{y}_{k} = \reference{\disc{h}}_{k} / \widehat{\disc{h}}.
\end{equation}
In the following, we compute the concrete form of the layer-reduced Hamiltonian and the
operator \(\disc{\Psi}^\dagger\), containing the test-functions.
Due to the fact that the average density is given as \(\widehat{\disc{\rho}}
= \rho^\top \disc{y}\), the concrete form of the reduced Hamiltonian is readily derived as
\begin{equation*}
\delta^2\widehat{\Hdisc} = \begin{pmatrix}
     g \diag(\widehat{\disc{r}}) & 0\\
     0 & \diag\left(\inp{\widehat{\disc{\rho}}\ast\widehat{\disc{h}}}_{\grid{E}}\right)
    \end{pmatrix},
\quad
\text{where } \widehat{\disc{r}} = \disc{y}^\top R \disc{y} \in X_{\grid{I}}.
\end{equation*}
with the matrix \(R = T^\top\diag(\Delta\rho) T = (\rho_{\min\{k,l\}})_{k,l}\) introduced
in section~\ref{sec:lin_and_modes}.
We note that \(\delta^2\widehat{\Hdisc}\) simply corresponds to a quadratic approximation
of a single-layer Hamiltonian, albeit with variable densities. In fact, \(\laylow{\disc{h}}\)
is the total column height of the reference configuration, and both \(\laylow{\disc{\rho}}\) and
\(\laylow{\disc{r}}\) are average values of the density over each stack.
A simple computation now yields the concrete form of
\(\disc{\Psi}^\dagger\) as
\begin{equation*}
  \disc{\Psi}^\dagger = \begin{pmatrix}
     (1 / \widehat{\disc{r}}) \ast \disc{y}^\top R & 0\\
     0 & \inp{\rho \ast \reference{\disc{h}}}_{\grid{E}}^\top \,/\, \inp{\widehat{\disc{h}}\ast\widehat{\disc{\rho}}}_{\grid{E}}
    \end{pmatrix}\;.
\end{equation*}
At first glance, the concrete form of \(\disc{\Psi}^\dagger\) is not very instructive,
even though it can be easily computed in practice. However, in the special case of constant densities,
it simplifies further.
\begin{remark}
In the case where \(\rho_k = \reference{\rho}\) for
\(k = 1,2,\dotsc,L\), we obtain that
\begin{equation*}
  \disc{\Psi}^\dagger = \begin{pmatrix}
     \disc{1}^\top & 0\\
     0 & \inp{\reference{\disc{h}}}_{\grid{E}}^\top \,/\, \inp{\widehat{\disc{h}}}_{\grid{E}}
    \end{pmatrix}\;.
\end{equation*}
Thus, the roles of the test-functions in \(\disc{\Psi}^\dagger\) and ansatz-functions in
\(\disc{\Psi}\) are simply interchanged with respect to the continuity and momentum equation.
We note that this closely resembles the averaging operators employed in the split-explicit
scheme; cf.~\cite{higdon2005two}.
\end{remark}

\subsection{Total mass conservation}

One drawback of the outlined approach is that the form of the test
functions in \(\disc{\Psi}^\dagger\) can not be controlled directly, rather they arise from the choice
of \(\disc{\Psi}\) in an indirect way. This is a problem for instance for exact
mass-conservation as considered in section~\ref{sec:mass_conservation}. We briefly
describe a simple remedy for this in the context of the barotropic method.

In a first step, we replace the second variation of the Hamiltonian by the modified
version
\[
\delta^2\widetilde{\Hdisc}
= \begin{pmatrix}
  g \diag\left(\rho \rho^\top / \widehat{\disc{\rho}}\right) & 0\\
  0 & \diag\left(\rho\ast\inp{\disc{h}}_{\grid{E}}\right)
\end{pmatrix},
\]
where the matrix \(R\) is replaced by the cell-wise defined rank-one matrix
\(\rho \rho^\top / \widehat{\disc{\rho}}_i\in \R^{L\times L}\) for every \(i\in \grid{I}\).
The reduced operator is then derived in the same way as before, based on this modified
Hamiltonian. Using again the ansatz \(\disc{\Psi}\) from~\eqref{eq:barotropic_ansatz}, we obtain now
\begin{equation*}
\delta^2\widehat{\Hdisc} = \begin{pmatrix}
     g \diag(\widehat{\disc{\rho}}) & 0\\
     0 & \diag\left(\inp{\widehat{\disc{\rho}}\widehat{\disc{h}}}_{\grid{E}}\right)
    \end{pmatrix}
\text{ and } 
  \disc{\Psi}^\dagger = \begin{pmatrix}
     \rho^\top / \widehat{\disc{\rho}} \, & 0\\
     0 & \inp{\rho \ast \reference{\disc{h}}}_{\grid{E}}^\top \,/\, \inp{\widehat{\disc{h}}\ast\widehat{\disc{\rho}}}_{\grid{E}}
    \end{pmatrix}\;.
\end{equation*}
The resulting linear operator from this choice leads to global mass conservation.
\begin{proposition}
For a reference configuration, define as before
\(\widetilde{\disc{A}}_P = \disc{P} \,\Jdisc\, \delta^2
\widetilde{\Hdisc}\, \disc{P} =
\disc{\Psi}\,\widehat{\Jdisc}\,\delta^2\widehat{\Hdisc}\,\disc{\Psi}^\dagger\).
Then, a corresponding ETD-method with \(\disc{A}_n = \widetilde{\disc{A}}_P\) preserves the
total mass;
\[
\disc{M}[\disc{h}_{n+1}] = \disc{M}[\disc{h}_n],
\quad \text{where } \disc{M}[\disc{h}] = \sum_{k=1}^L \rho_k(\disc{1},\disc{h}_k)_{\grid{I}},
\]
and \((\disc{h}_n,\disc{u}_n) = \disc{V}_n\) are the time steps~\eqref{eq:ETD_final}.
\end{proposition}
\begin{proof}
Defining the vector \(\disc{l} =
(\disc{\rho}, \disc{0})\), mass can be computed as \(\disc{M}(\disc{V})
= \inner{\disc{l}, \disc{V}}_X\), and with Theorem~\ref{thm:linear_invariant}, we have to verify that
\[
(\disc{l}, \widetilde{\disc{A}}_P \disc{V})_X
= (\disc{\rho}, \disc{\Psi}_h \disc{\Psi}_h^\dagger \, \disc{f}_h[\disc{V}])_{X_{\grid{I}}^L}
= 0
\quad\text{for all } \disc{V} \in X,
\]
where \(\disc{f}_{h}[\disc{V}] \in X_{\grid{I}}^L\) is the first component (corresponding to height
variables) of \(\disc{f}[\disc{V}] = \Jdisc\, \delta^2 \widetilde{\Hdisc}\, \disc{P} \disc{V}\).
We note that \(\disc{\Psi}_h \disc{\Psi}^\dagger_h = \disc{y}\rho^\top /
\widehat{\disc{\rho}}\), from the concrete form of \(\disc{\Psi}^\dagger\).
Thus, it follows that 
\(\rho^\top \disc{\Psi}_h \disc{\Psi}^\dagger_h = \rho^\top\disc{y}\rho^\top /
\widehat{\disc{\rho}} = \rho^\top\), since \(\widehat{\disc{\rho}} = \rho^\top \disc{y}\).
Therefore, it indeed follows
\((\disc{\rho}, \disc{\Psi}_h \disc{\Psi}^\dagger_h \, \disc{f}_h[\disc{V}])_{X_{\grid{I}}^L}
 = (\disc{\rho}, \disc{f}_h[\disc{V}])_{X_{\grid{I}}^L} = 0\), due to the properties of \(\Jdisc\).
\end{proof}

\begin{remark}
Similarly, replacing the matrix \(R\) by the constant rank-one matrix
\((\disc{R}_i)_{k,l} = \widehat{\disc{\rho}}_i\) in every cell \(i\in \grid{I}\) results in an ETD
method which exactly conserves the total layer volume, defined as \(\sum_k (\disc{1},\disc{h}_k)_{\grid{I}}\).
\end{remark}

We note that this method incurs an additional approximation error. However, since the
approximation of \(R\) with a rank-one matrix is well justified, and the reduced system
can only capture the single fast mode contained in this space, we still expect good
properties from this linear operator in the context of an ETD-scheme.

\section{Numerical results}
\label{sec:results}
In this section, we numerically demonstrate the stability and performance of the ETD methods
described in this work. We first given an overview of the simulation setup, which is
based on a simplified version of the SOMA testcase~\cite{wolfram2015diagnosing}. 
The computational domain is given by a circular basin on the surface
of the sphere of radius \(6371.22\) km, centered at
\((35^\circ\text{N},0^\circ\text{W})\) longitude-latitude. The basin is
\(2500\) km in diameter, has a depth ranging from \(2.5\) km at the center to
\(100\) m on the coastal shelf. The concrete form of the bathymetry can be found in
\cite[Appendix~A]{wolfram2015diagnosing}; see also Figure~\ref{fig:layers}.
We consider a single-layer and a three-layer configuration.

\subsection{Algorithmic details}
\label{sec:alg_detail}

In the following, we detail the numerical setup employed in the computational experiments.

\subsubsection{Spatial mesh}
A quasi-uniform mesh is constructed from a centroidal Voronoi tessellation~\cite{Ringler2008},
where the distance of the cell centers \(d_e\) is ca.\ \(16\) km resolution. In order to
obtain an initial condition for the initial layer configuration, we interpolate the
initial heights as in Figure~\ref{fig:layers} to the cell centers. Then, all cell
variables in each layer that correspond to zero heights are marked as dry. Additionally,
all edges adjacent to a dry cell are marked as boundary edges. Subsequently, the degrees
of freedom corresponding to those cells and edges are fixed to zero and thus eliminated
from the computation; cf.\ also Appendix~\ref{app:TRiSK}.

In order to compare different time stepping methods at different
CFL-numbers, we introduce the reference time step and the Courant number. In this context,
we define it for simplicity as
\[
\dtC = 1/\abs{\disc{A}_0},
\quad \Cour(\dt) = \dt \abs{\disc{A}_0} = \dt/\dtC
\]
where \(\abs{\cdot} = \sigma_{\max}(\cdot)\) denotes the largest magnitude eigenvalue, and
\(\disc{A}_0 = \disc{F}'(\disc{V}^0)\) is the linearized operator at the stable reference
configuration. We note that \(\dtC\) is determined (up to a constant factor) by the largest
quotient of the local mesh-width and the local free-surface wave-speed
\(\sqrt{g \, b(x)}\).
Concretely, we obtain \(\dtC \approx 37.9\) [s] on the given domain, mesh, and bathymetry.

\subsubsection{Considered time stepping methods}
In the tests, the explicit fourth-order Runge-Kutta
method (RK4) serves as a base-line, since it is explicit (thus easy to efficiently
implement in a parallel environment), sufficiently high order accurate (in combination with
the second-order TRiSK scheme), and includes an
imaginary interval in its stability region. Specifically, stability of RK4 (for the
linearized equation \(\partial_t \disc{V} = \disc{A}_0 \disc{V}\)) is given for Courant numbers
\(\Cour(\dt) \leq \sqrt{8}\).
Thus, the maximal RK4 stepsize is given as \(\dt_{\text{RK4}} = \sqrt{8}\dtC \approx
107.2\)~[s] for the \(16\)~km grid, which is used as a reference time step for performance
considerations.
We remark that, in practice, a stable simulation is only obtained for slightly smaller Courant numbers, since the definition employed above ignores
the nonlinearity in the forcing term.
Concerning the choice of RK4 over lower order methods, we note that optimal order one and
two stage RK schemes are unconditionally unstable for imaginary
eigenvalues, and that RK4 delivers a better ratio of the number of internal stages to the
maximal CFL-compliant time step than RK3.

The ETD methods described in this work can be separated into two classes.
The first class of methods is constructed by choosing the linear operator as
\(\disc{A}_n = \reference{\disc{A}}\), as in~\eqref{eq:Abar}, linearized either at the reference
configuration \(\reference{\disc{V}} = (-\disc{b}, 0)\) or updated in each time step
with the current height \(\reference{\disc{V}}_n = (\disc{h}_n, 0)\).
Since \(\reference{\disc{A}}_n\) corresponds to a first-order
wave operator, this class of methods will be called ETD\(S\)wave, where
\(S\) refers to the number of internal stages.
The second class of methods are based on section~\ref{sec:mode_splitting}.
Within this class, we will focus on the methods where the linear
operator is projected onto the barotropic mode~\ref{eq:barotropic_ETD}.
For this reason we will refer to these methods as B-ETD\(S\)wave.

\subsubsection{Implementation of the Krylov methods}
\label{sec:impl_krylov}
For both classes of ETD methods, the Krylov subspace method from section~\ref{sec:krylov}
is used to evaluate the \(\phi_s\) functions. Because both classes of methods possess the
properties in Proposition~\ref{prop:skew}, the more efficient skew-Lanczos process, described in
section~\ref{sec:phi_comp}, is chosen over the Arnoldi process or IOM.

  The cost of evaluating the inner products is not significant, since the corresponding mass
  matrix is diagonal in the single-layer case and involves only a vertical translation between
  height and layer coordinates using the layer matrix \(T\) in the multilayer case.
  For each stage of the considered methods, matrix functions of
  \(\disc{A}_n\) need to be computed for additional right-hand sides. Specifically, in the
  second stage the right-hand side is \(\disc{b}_1 = \disc{F}[\disc{V}_n]\), 
  and in each subsequent stage the right-hand side \(\disc{b}_s =
  \disc{R}_n[\disc{v}_{n,s-1}]\) is introduced, where \(\disc{v}_{n,s-1}\) is the previous
  stage; cf.\ Appendix~\ref{app:exp_rk}. Additionally, possibly different matrix functions
  of \(\disc{A}_n\) need to be applied to \(\disc{b}_1,\ldots, \disc{b}_{s-1}\). Here, the previous Krylov
  spaces can be reused; only the matrix function of the Hessenberg matrix needs to be
  recomputed. In cases where additional Krylov vectors are required, the Arnoldi process
  can be continued. Thus, for each additional stage, effectively one
  additional skew-Lanczos process needs to be computed, and the matrix
  exponentials of the Hessenberg matrix and linear combinations
  in~\eqref{eq:Krylov_approx} need to be updated at most \(s-1\) times.

Additionally, we comment on the number of Krylov iterations per evaluation of a matrix
\(\phi\)-function.
The theoretical estimates (see, e.g., \cite[Section~4.2]{hochbruck2010exponential})
suggest that the required number of Krylov vectors 
effectively depends linearly on the Courant number, before an exponential rate of
convergence sets in.
In practice, we also employ the adaptive a~posteriori error criterion suggested in~\cite{Botchev2016}, based on~\cite{CELLEDONI1997365,druskin1998using}.
However, in the numerical experiments, we found that the convergence
behavior suggested by theory was sharp:
e.g., an error tolerance of \(10^{-6}\) was usually met after \(M \geq a_1 \Cour(\dt) + a_2\)
iterations, where appropriate constants \(a_1 \approx 1.25\), \(a_2 \approx 15\) were
determined empirically.
Moreover, using less than \(\Cour(\dt)\) Krylov iterations usually lead to completely
inaccurate solutions and even unstable simulations. This can be contrasted with an
approximation of the matrix \(\phi\)-function based on RK4 time stepping
using~\eqref{eq:phi_ode}, which requires at a minimum a number of
\(4/\sqrt{8} = \sqrt{2} \approx 1.41\)
matrix multiplications per unit time step (with \(\Cour(\dt) = 1\)), just to obtain basic
stability, and then converges at fourth order.

\subsection{Discussion of results}
In the first two test cases, obtained using the
single-layer configuration, the order of convergence and
energy conservation of the ETD\(S\)wave methods are investigated.
The third and final test case uses a three-layer configuration and a spin-up initial
condition (over a ten year horizon), and investigates the performance and accuracy of the
methods over a ten year simulation time, including additional forcing and biharmonic
smoothing terms.

\subsubsection{Single-layer scenario} \label{sec:SL_scen}
The first test scenario is used to
verify the accuracy and the energy conservation properties of the
ETD\(S\)wave methods. For simplicity, this scenario is implemented using the single-layer
configuration.
We consider an unforced problem either without or with a minimal amount of biharmonic
smoothing added to the problem.
We consider two initial conditions corresponding to fast and slow modes of the
single-layer equation, respectively.
The initial condition for the height \(h_0 = -b + \eta_0\) is a Gaussian
perturbation of the stable reference height with
\(\eta_0 = \bar{\eta} \, \exp\left( - (x-x_{\text{center}})^2 / (2\sigma^2)\right)\),
where the radius is \(\sigma = 200\) km, the total perturbation height is \(\bar{\eta} =
2\) m and \(x_{\text{center}}\) is the location at the center of the domain.

In the first case, this is combined with a zero initial velocity
\(u_0 = 0\). This then leads to a free-surface gravity wave emanating from the center of the
domain. Over the simulation horizon of six hours the wave spreads out from the center of
the domain, is reflected at the coastal boundaries, and roughly ends up back at the center of the domain.
In the second case, the initial height is chosen in the same way. However, now the initial
velocity is given as \(u_0 = (g/f(x_{\text{center}})) \, \hat{k} \cross \grad \eta_0\). This
choice ensures that \(\div u_0 = 0\) and the pressure gradient \(g\grad\eta_0\) balances the
Coriolis force \(f \hat{k} \cross u_0\), which is referred to as
\emph{geostrophic balance}. The
dynamics of this solution evolve on a slower time scale; a snapshot of the solution after
ten days is given in Figure~\ref{fig:single_layer}.
In the following, we will refer to the former as the \emph{gravity wave}, and to
the latter as the \emph{geostrophic} testcase.
\begin{figure}[!ht]
  \centering
  \includegraphics[width=0.32\linewidth,trim={.5cm 0cm, 1cm, 1cm},clip]{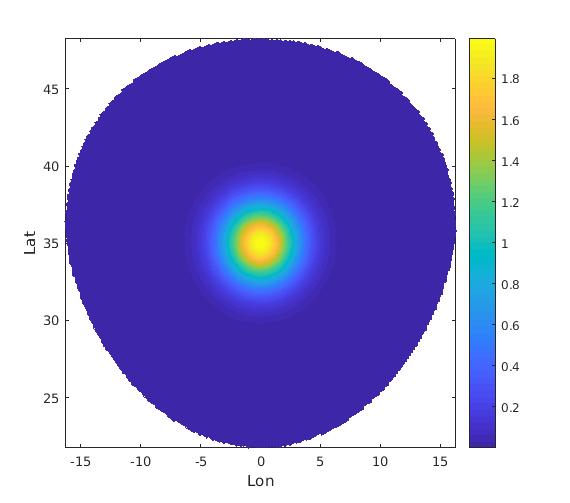}
  \includegraphics[width=0.32\linewidth,trim={.5cm 0cm, 1cm, 1cm},clip]{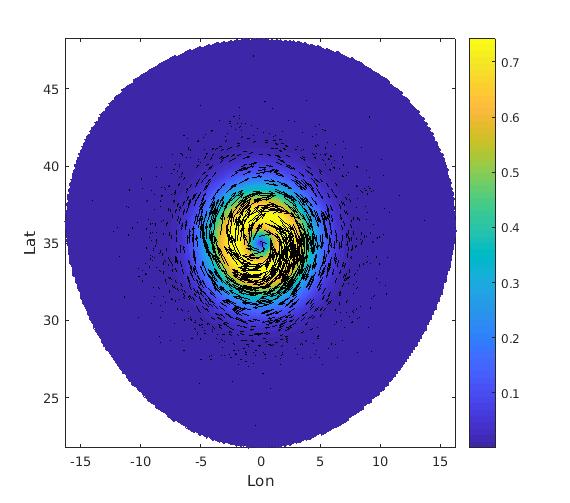}
  \includegraphics[width=0.32\linewidth,trim={.5cm 0cm, 1cm, 1cm},clip]{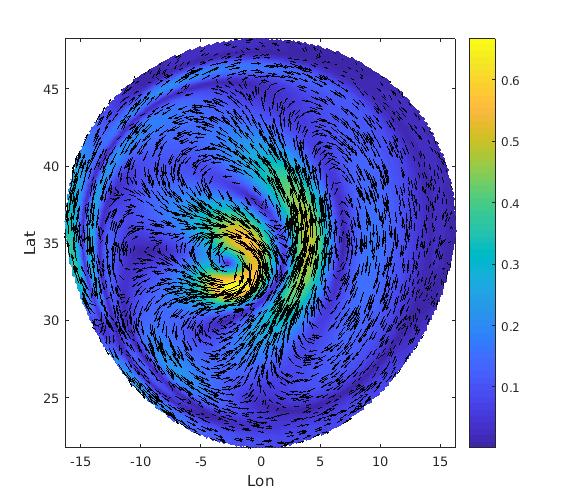}
  \caption{Snapshots of the initial sea surface height \(\eta_0 = h_0 + b\) [m],
      the initial velocity \(u_0\) [m/s] and the
    velocity after ten days of simulation time for the geostrophic testcase.
    Latitude and Longitude are in degrees.
}
  \label{fig:single_layer}
\end{figure}

\paragraph{Convergence test}
The errors of RK4 and various ETD\(S\)wave methods are computed using a reference solution
computed with RK4 using a time step size of \(\dt = (1/4)\dtC\).
The time step sizes for the tested methods are chosen as
\(\dt = 2^j \dtC\), \(j = 0,1,\dotsc,7\), and the two values \(\dt = 0.9 \dt_{\text{RK4}}\)
and \(\dt = 1.1\dt_{\text{RK4}}\) are added to verify the stability region of RK4.
Additionally, we consider the first-order ETDwave method, the
second-order ETD2wave method with \(c_2=1\) and \(c_2=2/3\) and the 
third-order ETD3wave method (detailed in Appendix~\ref{app:exp_rk}), where the
coefficients are chosen to be \((c_2,c_3)=(1/2,3/4)\).
We also consider the impact of using updated heights for the ETD2wave methods.
A small amount of biharmonic smoothing is added to the problem
(see Appendix~\ref{app:model_forcing}) with horizontal viscosity
\(\disc{\nu}_{\text{h}} = 2 \times 10^{9}\). 
Note that, for all experiments, we use a number of Krylov iterations set to
\(M = 1.3\,C(\dt) + 15\) (rounded to the nearest integer), which ensures an accurate
evaluation of the matrix \(\phi\)-functions; see section~\ref{sec:impl_krylov}.

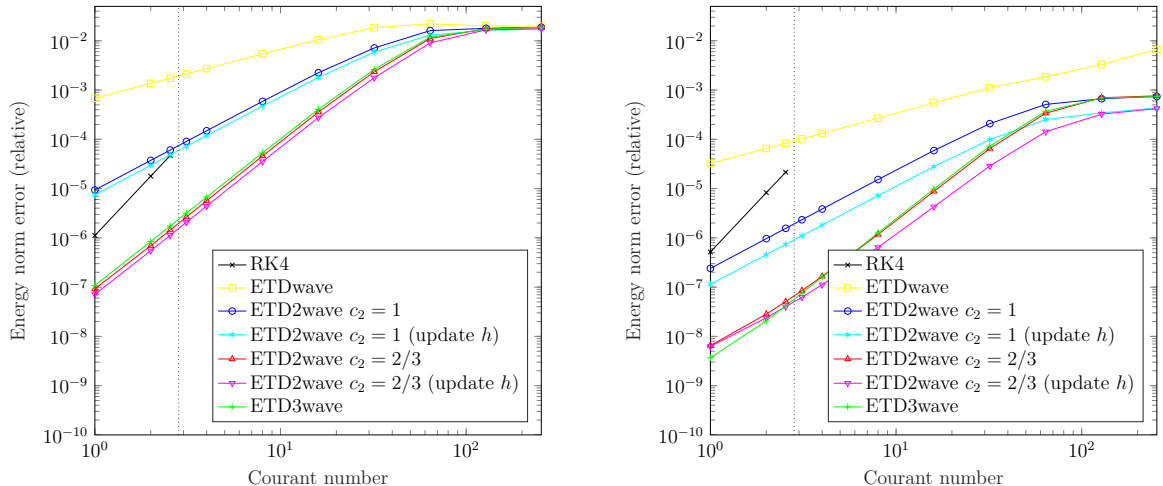
\begin{figure}[!ht]
  \centering
  \scalebox{.62}{
    \large
%
%
\definecolor{mycolor1}{rgb}{1.00000,1.00000,0.00000}%
\definecolor{mycolor2}{rgb}{0.00000,1.00000,1.00000}%
\definecolor{mycolor3}{rgb}{1.00000,0.00000,1.00000}%
\begin{tikzpicture}

\begin{axis}[%
width=3.73in,
height=3.608in,
at={(0.654in,0.487in)},
scale only axis,
unbounded coords=jump,
xmode=log,
xmin=1,
xmax=254,
xminorticks=true,
xlabel style={font=\color{white!15!black}},
xlabel={Courant number},
ymode=log,
ymin=1e-10,
ymax=0.05,
yminorticks=true,
ylabel style={font=\color{white!15!black}},
ylabel={Energy norm error (relative)},
axis background/.style={fill=white},
legend style={at={(0.97,0.03)}, anchor=south east, legend cell align=left, align=left, draw=white!15!black}
]
\addplot [color=black, mark=x, mark options={solid, black}]
  table[row sep=crcr]{%
1	1.1130668912532e-06\\
2	1.78710232585091e-05\\
2.54558441227157	4.68571041536303e-05\\
3.11126983722081	nan\\
4	nan\\
8	nan\\
16	nan\\
32	nan\\
64	nan\\
128	nan\\
254	nan\\
};
\addlegendentry{RK4}

\addplot [color=mycolor1, mark=square, mark options={solid, mycolor1}]
  table[row sep=crcr]{%
1	0.000682683322167987\\
2	0.00136549442024085\\
2.54558441227157	0.00173776809933797\\
3.11126983722081	0.00212340676589503\\
4	0.00272814395555221\\
8	0.00542164164314498\\
16	0.0105082929268486\\
32	0.0184963699748349\\
64	0.0220768657254146\\
128	0.0201136298483832\\
254	0.0196564503643306\\
};
\addlegendentry{ETDwave}

\addplot [color=blue, mark=o, mark options={solid, blue}]
  table[row sep=crcr]{%
1	9.35937214570486e-06\\
2	3.74158778774015e-05\\
2.54558441227157	6.05760530055885e-05\\
3.11126983722081	9.04199184450456e-05\\
4	0.000149205973285393\\
8	0.00058968943092643\\
16	0.00224652615617312\\
32	0.00715332501536734\\
64	0.0161055257054992\\
128	0.0179266945036315\\
254	0.0187623896545523\\
};
\addlegendentry{ETD2wave $c_2=1$}

\addplot [color=mycolor2, mark=asterisk, mark options={solid, mycolor2}]
  table[row sep=crcr]{%
1	7.46406920603577e-06\\
2	2.9837440721888e-05\\
2.54558441227157	4.83064010925268e-05\\
3.11126983722081	7.2106674764768e-05\\
4	0.000118995024332074\\
8	0.000470760468551887\\
16	0.00180246876278411\\
32	0.00590784831198182\\
64	0.0130825046359281\\
128	0.0163344956305759\\
254	0.0172266499069932\\
};
\addlegendentry{ETD2wave $c_2=1$ (update $h$)}

\addplot [color=red, mark=triangle, mark options={solid, red}]
  table[row sep=crcr]{%
1	9.19666345050595e-08\\
2	6.92679903492092e-07\\
2.54558441227157	1.43687927819158e-06\\
3.11126983722081	2.64198807112277e-06\\
4	5.66288808560738e-06\\
8	4.58684707472242e-05\\
16	0.000356636421959851\\
32	0.00235021550532163\\
64	0.0111022461361773\\
128	0.0175942100053575\\
254	0.0186714121339397\\
};
\addlegendentry{ETD2wave $c_2=2/3$}

\addplot [color=mycolor3, mark=triangle, mark options={solid, rotate=180, mycolor3}]
  table[row sep=crcr]{%
1	7.13949006268937e-08\\
2	5.46056897638229e-07\\
2.54558441227157	1.12874756394348e-06\\
3.11126983722081	2.06765294429802e-06\\
4	4.41076618010512e-06\\
8	3.53742797631863e-05\\
16	0.000274616464334762\\
32	0.00177259104067741\\
64	0.00907000918232822\\
128	0.0166809595703467\\
254	0.0175440375753625\\
};
\addlegendentry{ETD2wave $c_2=2/3$ (update $h$)}

\addplot [color=green, mark=+, mark options={solid, green}]
  table[row sep=crcr]{%
1	1.08012475560765e-07\\
2	8.46304022217722e-07\\
2.54558441227157	1.74134072874745e-06\\
3.11126983722081	3.17493698616283e-06\\
4	6.73309347407024e-06\\
8	5.32329463925399e-05\\
16	0.00040685258239176\\
32	0.00263446827248619\\
64	0.0118397624205946\\
128	0.0172403688226405\\
254	0.0188129904439507\\
};
\addlegendentry{ETD3wave}


\addplot [color=black, dotted, forget plot]
  table[row sep=crcr]{%
2.82842712474619	1e-12\\
2.82842712474619	1000\\
};
\end{axis}

\begin{axis}[%
width=4.844in,
height=4.427in,
at={(0in,0in)},
scale only axis,
xmin=0,
xmax=1,
ymin=0,
ymax=1,
axis line style={draw=none},
ticks=none,
axis x line*=bottom,
axis y line*=left,
legend style={legend cell align=left, align=left, draw=white!15!black}
]
\end{axis}
\end{tikzpicture}%
  }
  \scalebox{.62}{
    \large
%
%
\definecolor{mycolor1}{rgb}{1.00000,1.00000,0.00000}%
\definecolor{mycolor2}{rgb}{0.00000,1.00000,1.00000}%
\definecolor{mycolor3}{rgb}{1.00000,0.00000,1.00000}%
\begin{tikzpicture}

\begin{axis}[%
width=3.73in,
height=3.608in,
at={(0.654in,0.487in)},
scale only axis,
unbounded coords=jump,
xmode=log,
xmin=1,
xmax=254,
xminorticks=true,
xlabel style={font=\color{white!15!black}},
xlabel={Courant number},
ymode=log,
ymin=1e-10,
ymax=0.05,
yminorticks=true,
ylabel style={font=\color{white!15!black}},
ylabel={Energy norm error (relative)},
axis background/.style={fill=white},
legend style={at={(0.97,0.03)}, anchor=south east, legend cell align=left, align=left, draw=white!15!black}
]
\addplot [color=black, mark=x, mark options={solid, black}]
  table[row sep=crcr]{%
1	5.16683307384611e-07\\
2	8.26061339637053e-06\\
2.54558441227157	2.14613105766831e-05\\
3.11126983722081	nan\\
4	nan\\
8	nan\\
16	nan\\
32	nan\\
64	nan\\
128	nan\\
254	nan\\
};
\addlegendentry{RK4}

\addplot [color=mycolor1, mark=square, mark options={solid, mycolor1}]
  table[row sep=crcr]{%
1	3.24197139515344e-05\\
2	6.51514635001052e-05\\
2.54558441227157	8.31449771971315e-05\\
3.11126983722081	0.000101905452928549\\
4	0.000131593676094379\\
8	0.000268530152005977\\
16	0.000556072454562158\\
32	0.00110520281445259\\
64	0.00186207698494019\\
128	0.00330229930120925\\
254	0.00665587932864713\\
};
\addlegendentry{ETDwave}

\addplot [color=blue, mark=o, mark options={solid, blue}]
  table[row sep=crcr]{%
1	2.39629841116489e-07\\
2	9.62666136681801e-07\\
2.54558441227157	1.55982440752891e-06\\
3.11126983722081	2.32968277460833e-06\\
4	3.84775406104637e-06\\
8	1.52771564619764e-05\\
16	5.92417221442025e-05\\
32	0.00020818609824126\\
64	0.00051149135085352\\
128	0.000666899599423297\\
254	0.000741313115451426\\
};
\addlegendentry{ETD2wave $c_2=1$}

\addplot [color=mycolor2, mark=asterisk, mark options={solid, mycolor2}]
  table[row sep=crcr]{%
1	1.13915073454918e-07\\
2	4.55014549707169e-07\\
2.54558441227157	7.36679086853659e-07\\
3.11126983722081	1.09975848191236e-06\\
4	1.81555736699134e-06\\
8	7.20160722514074e-06\\
16	2.79250400512614e-05\\
32	9.88167922135174e-05\\
64	0.000251820090572374\\
128	0.000344112633432161\\
254	0.000434006306658002\\
};
\addlegendentry{ETD2wave $c_2=1$ (update $h$)}

\addplot [color=red, mark=triangle, mark options={solid, red}]
  table[row sep=crcr]{%
1	6.52086822755978e-09\\
2	2.8234603978186e-08\\
2.54558441227157	5.07518701300394e-08\\
3.11126983722081	8.47498488717458e-08\\
4	1.65569046972156e-07\\
8	1.16974700764299e-06\\
16	8.79165607950458e-06\\
32	6.49112921521106e-05\\
64	0.000338611311958805\\
128	0.000698604800997335\\
254	0.000764186805068507\\
};
\addlegendentry{ETD2wave $c_2=2/3$}

\addplot [color=mycolor3, mark=triangle, mark options={solid, rotate=180, mycolor3}]
  table[row sep=crcr]{%
1	6.2513402530729e-09\\
2	2.38430138921002e-08\\
2.54558441227157	3.99102793090503e-08\\
3.11126983722081	6.23093658852251e-08\\
4	1.11429004109383e-07\\
8	6.39229407593306e-07\\
16	4.27400891966051e-06\\
32	2.87255441126362e-05\\
64	0.000142688797847457\\
128	0.000323959244425019\\
254	0.000421105152006572\\
};
\addlegendentry{ETD2wave $c_2=2/3$ (update $h$)}

\addplot [color=green, mark=+, mark options={solid, green}]
  table[row sep=crcr]{%
1	3.67272158091893e-09\\
2	2.0603623223565e-08\\
2.54558441227157	4.16721251768563e-08\\
3.11126983722081	7.54872406188523e-08\\
4	1.59517191652178e-07\\
8	1.26102686529212e-06\\
16	9.80452051034822e-06\\
32	7.26848707280316e-05\\
64	0.000370030778358157\\
128	0.000688493106832089\\
254	0.00073730834423367\\
};
\addlegendentry{ETD3wave}

\addplot [color=black, dotted, forget plot]
  table[row sep=crcr]{%
2.82842712474619	1e-12\\
2.82842712474619	1000\\
};
\end{axis}

\begin{axis}[%
width=4.844in,
height=4.427in,
at={(0in,0in)},
scale only axis,
xmin=0,
xmax=1,
ymin=0,
ymax=1,
axis line style={draw=none},
ticks=none,
axis x line*=bottom,
axis y line*=left,
legend style={legend cell align=left, align=left, draw=white!15!black}
]
\end{axis}
\end{tikzpicture}%
  }
  \caption{The relative error in the solution (given in terms of the SSH
    \(\eta = h + b\) and the velocities \(u\)) at the final time in a (linearized) energy norm
    for various methods and Courant numbers. The vertical dotted line denotes the
    maximal stable time step for RK4 at Courant number $\sqrt{8}$.
    Left: for the gravity wave initial condition \((h_0,0)\)
    after six hours. Right: for the geostrophic initial condition \((h_0,u_0)\) after ten days.}
  \label{conv_test}
\end{figure}
To compare the methods, in Figure~\ref{conv_test} we show the relative solution error at
the final time in the discrete linearized energy norm as a function of the
Courant number.
The discrete linearized energy norm is induced by the mass matrix \(\disc{M}_H =
\disc{M}_{X^L}\delta^2\reference{\Hdisc}\) (cf.\ sections~\ref{sec:lin_and_modes}
and~\ref{sec:krylov}), and locally combines the weighted errors in heights and weights in
a way that more closely matches their contribution to the total energy (up to a linearization error).
We note that RK4 is unstable for time steps larger than $\dt_{\text{RK4}}$ (as predicted by
theory) but the ETD methods remain stable for all time steps considered.
Within the regions of stability, the methods exhibit the expected convergence order.
Only the ETD2wave method with \(c_2 = 2/3\) is noteworthy, since it
appears to have almost third-order convergence for a large regime of time step sizes.
Moreover, we note that all ETD methods deliver solutions that are accurate up to the
second significant digit for all time step sizes.
Surprisingly, some of the second- and third-order ETD methods are more accurate than RK4 at
the same time step size, despite being of lower order than RK4.

We can also notice that the ETD methods are more accurate for the geostrophic
testcase, whereas the accuracy of RK4 is similar in both cases. Moreover, differences can
be seen among the ETD methods:
In the gravity wave test-case the error of all methods is similar at a Courant number of
ca.\ \(100\). In the range \(1 \leq \dt \leq 50\) a clear benefit in accuracy can be seen for
the three stage method and ETD2wave with \(c_2 = 2/3\) over the second- and first-order
methods. Updating the reference height appears to only provide a marginal benefit in the
first test-case, which
can be attributed to the fact that the perturbation of the height compared to the stable
reference height changes appreciably over each time step, due to the fast free surface wave.
In contrast, for the geostrophic test-case, we observe an improvement due
to incorporating the current reference height \(\disc{h}_n\) into \(\disc{A}\). We note that this
improvement is also present for large Courant numbers.

\paragraph{Performance test}
From a practical point of view, the most interesting question is the performance of the methods.
Thus, we also plot the errors as a function of the wall clock time, measured in simulated
years per day (SYPD); see Figure~\ref{conv_test_SYPD}.
Here, the purely explicit RK4 method has an advantage over the ETD
methods at moderate time step sizes, since no \(\phi\)-functions need to be evaluated.
In fact, in the gravity wave testcase, we observe that RK4 outperforms the proposed ETD
methods in terms of accuracy in the range of SYPD that it can achieve. In the geostrophic
testcase, the ETD2wave methods with \(c_2 = 2/3\) outperform RK4, using fewer but larger, more
expensive time steps. In all cases, the maximal SYPD that can be achieved with
RK4 is bounded by the maximal time step, whereas the ETD methods achieve higher SYPD in the
large time step regime. However, we also notice that performance gains decrease at
higher Courant numbers. In the gravity wave testcase, the simulation horizon is very short,
so that SYPD actually decreases around Courant number 100. In the geostrophic case, where
the horizon is longer, SYPD continues to increase with larger time steps, but at Courant
number 200, the gain of doing fewer nonlinear forcing term evaluations per time interval
is increasingly balanced by the increased cost of computing the \(\phi\)-functions, since
the Krylov method requires a number of iterations that is roughly proportional to the
length of the time step.
Note that this also highlights that the performance gains of the ETD methods exploit the
fact that the cost of a matrix-vector product with \(\disc{A}_n\) is cheaper than an
evaluation of \(\disc{F}\). For Rosenbrock-ETD methods, this is more difficult to achieve, since
the Jacobian of every term in \(\disc{F}\) also needs to be included in the linearization.
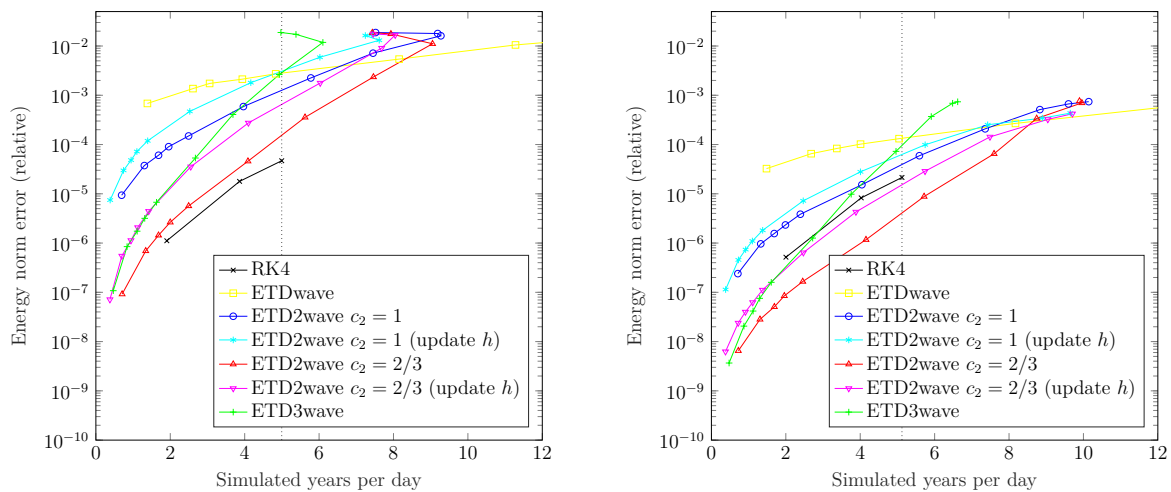
\begin{figure}[!ht]
  \centering
  \scalebox{.62}{
    \large
%
%
\definecolor{mycolor1}{rgb}{1.00000,1.00000,0.00000}%
\definecolor{mycolor2}{rgb}{0.00000,1.00000,1.00000}%
\definecolor{mycolor3}{rgb}{1.00000,0.00000,1.00000}%
\begin{tikzpicture}

\begin{axis}[%
width=3.73in,
height=3.608in,
at={(0.654in,0.487in)},
scale only axis,
unbounded coords=jump,
xmin=0,
xmax=12,
xlabel style={font=\color{white!15!black}},
xlabel={Simulated years per day},
ymode=log,
ymin=1e-10,
ymax=0.05,
yminorticks=true,
ylabel style={font=\color{white!15!black}},
ylabel={Energy norm error (relative)},
axis background/.style={fill=white},
legend style={at={(0.97,0.03)}, anchor=south east, legend cell align=left, align=left, draw=white!15!black}
]
\addplot [color=black, mark=x, mark options={solid, black}]
  table[row sep=crcr]{%
1.90717641701486	1.1130668912532e-06\\
3.86286036989887	1.78710232585091e-05\\
4.99534823228167	4.68571041536303e-05\\
0	nan\\
0	nan\\
0	nan\\
0	nan\\
0	nan\\
0	nan\\
0	nan\\
0	nan\\
};
\addlegendentry{RK4}

\addplot [color=mycolor1, mark=square, mark options={solid, mycolor1}]
  table[row sep=crcr]{%
1.38496341082798	0.000682683322167987\\
2.60988934873877	0.00136549442024085\\
3.05766706144027	0.00173776809933797\\
3.93712413844548	0.00212340676589503\\
4.83460445492282	0.00272814395555221\\
8.15107696623293	0.00542164164314498\\
11.2797008917849	0.0105082929268486\\
15.1924448618787	0.0184963699748349\\
16.6778829933855	0.0220768657254146\\
17.4160547772335	0.0201136298483832\\
14.8456750326891	0.0196564503643306\\
};
\addlegendentry{ETDwave}

\addplot [color=blue, mark=o, mark options={solid, blue}]
  table[row sep=crcr]{%
0.692907943912879	9.35937214570486e-06\\
1.30289086444394	3.74158778774015e-05\\
1.68455107555527	6.05760530055885e-05\\
1.95731215195125	9.04199184450456e-05\\
2.49226552221649	0.000149205973285393\\
3.96541960241734	0.00058968943092643\\
5.77823044030809	0.00224652615617312\\
7.45045371961048	0.00715332501536734\\
9.27422557656363	0.0161055257054992\\
9.18783350320262	0.0179266945036315\\
7.51924211915587	0.0187623896545523\\
};
\addlegendentry{ETD2wave $c_2=1$}

\addplot [color=mycolor2, mark=asterisk, mark options={solid, mycolor2}]
  table[row sep=crcr]{%
0.383920321040066	7.46406920603577e-06\\
0.743776906046719	2.9837440721888e-05\\
0.94672964844093	4.83064010925268e-05\\
1.10246042176273	7.2106674764768e-05\\
1.39277588574887	0.000118995024332074\\
2.52453501347824	0.000470760468551887\\
4.16446043449867	0.00180246876278411\\
6.02213079537248	0.00590784831198182\\
7.61683359224896	0.0130825046359281\\
7.24309905898847	0.0163344956305759\\
7.41803447953253	0.0172266499069932\\
};
\addlegendentry{ETD2wave $c_2=1$ (update $h$)}

\addplot [color=red, mark=triangle, mark options={solid, red}]
  table[row sep=crcr]{%
0.705443450924016	9.19666345050595e-08\\
1.34007403643518	6.92679903492092e-07\\
1.68402381567374	1.43687927819158e-06\\
1.9985709717909	2.64198807112277e-06\\
2.49445146515123	5.66288808560738e-06\\
4.08811931725095	4.58684707472242e-05\\
5.6221788740952	0.000356636421959851\\
7.4603543558486	0.00235021550532163\\
9.0496797630462	0.0111022461361773\\
7.93101564406421	0.0175942100053575\\
7.43185089498193	0.0186714121339397\\
};
\addlegendentry{ETD2wave $c_2=2/3$}

\addplot [color=mycolor3, mark=triangle, mark options={solid, rotate=180, mycolor3}]
  table[row sep=crcr]{%
0.377697045290808	7.13949006268937e-08\\
0.697908221691101	5.46056897638229e-07\\
0.943165762320435	1.12874756394348e-06\\
1.12405169437349	2.06765294429802e-06\\
1.41883022771622	4.41076618010512e-06\\
2.55039585606203	3.53742797631863e-05\\
4.10010129281262	0.000274616464334762\\
6.03082712567873	0.00177259104067741\\
7.68134024024512	0.00907000918232822\\
8.04620698263997	0.0166809595703467\\
7.45455505228508	0.0175440375753625\\
};
\addlegendentry{ETD2wave $c_2=2/3$ (update $h$)}

\addplot [color=green, mark=+, mark options={solid, green}]
  table[row sep=crcr]{%
0.467436482081467	1.08012475560765e-07\\
0.838433147917815	8.46304022217722e-07\\
1.10946718692792	1.74134072874745e-06\\
1.31313325571988	3.17493698616283e-06\\
1.63083783886262	6.73309347407024e-06\\
2.67543120423265	5.32329463925399e-05\\
3.68494310674314	0.00040685258239176\\
4.93029646153376	0.00263446827248619\\
6.10220212202045	0.0118397624205946\\
5.37898653189928	0.0172403688226405\\
4.9733267248566	0.0188129904439507\\
};
\addlegendentry{ETD3wave}


\addplot [color=black, dotted, forget plot]
  table[row sep=crcr]{%
4.99534823228167	1e-12\\
4.99534823228167	1000\\
};
\end{axis}

\begin{axis}[%
width=4.844in,
height=4.427in,
at={(0in,0in)},
scale only axis,
xmin=0,
xmax=1,
ymin=0,
ymax=1,
axis line style={draw=none},
ticks=none,
axis x line*=bottom,
axis y line*=left,
legend style={legend cell align=left, align=left, draw=white!15!black}
]
\end{axis}
\end{tikzpicture}%
  }
  \scalebox{.62}{
    \large
%
%
\definecolor{mycolor1}{rgb}{1.00000,1.00000,0.00000}%
\definecolor{mycolor2}{rgb}{0.00000,1.00000,1.00000}%
\definecolor{mycolor3}{rgb}{1.00000,0.00000,1.00000}%
\begin{tikzpicture}

\begin{axis}[%
width=3.73in,
height=3.608in,
at={(0.654in,0.487in)},
scale only axis,
unbounded coords=jump,
xmin=0,
xmax=12,
xlabel style={font=\color{white!15!black}},
xlabel={Simulated years per day},
ymode=log,
ymin=1e-10,
ymax=0.05,
yminorticks=true,
ylabel style={font=\color{white!15!black}},
ylabel={Energy norm error (relative)},
axis background/.style={fill=white},
legend style={at={(0.97,0.03)}, anchor=south east, legend cell align=left, align=left, draw=white!15!black}
]
\addplot [color=black, mark=x, mark options={solid, black}]
  table[row sep=crcr]{%
2.00346798552942	5.16683307384611e-07\\
4.02654496946949	8.26061339637053e-06\\
5.12198389050807	2.14613105766831e-05\\
0	nan\\
0	nan\\
0	nan\\
0	nan\\
0	nan\\
0	nan\\
0	nan\\
0	nan\\
};
\addlegendentry{RK4}

\addplot [color=mycolor1, mark=square, mark options={solid, mycolor1}]
  table[row sep=crcr]{%
1.48087342302122	3.24197139515344e-05\\
2.6815151259791	6.51514635001052e-05\\
3.37110015115205	8.31449771971315e-05\\
4.00852233449796	0.000101905452928549\\
5.04441201400524	0.000131593676094379\\
8.1726141872032	0.000268530152005977\\
12.0032577166531	0.000556072454562158\\
15.5098110280128	0.00110520281445259\\
17.3869196067863	0.00186207698494019\\
19.6517912818089	0.00330229930120925\\
20.6636887684274	0.00665587932864713\\
};
\addlegendentry{ETDwave}

\addplot [color=blue, mark=o, mark options={solid, blue}]
  table[row sep=crcr]{%
0.710967473094431	2.39629841116489e-07\\
1.32613589062929	9.62666136681801e-07\\
1.68863242792489	1.55982440752891e-06\\
1.98933535269789	2.32968277460833e-06\\
2.39063814957596	3.84775406104637e-06\\
4.04443758918098	1.52771564619764e-05\\
5.59047729028274	5.92417221442025e-05\\
7.35839596290006	0.00020818609824126\\
8.82797981663899	0.00051149135085352\\
9.59806511005769	0.000666899599423297\\
10.1420945598069	0.000741313115451426\\
};
\addlegendentry{ETD2wave $c_2=1$}

\addplot [color=mycolor2, mark=asterisk, mark options={solid, mycolor2}]
  table[row sep=crcr]{%
0.381957519881165	1.13915073454918e-07\\
0.726109715764083	4.55014549707169e-07\\
0.920312400756401	7.36679086853659e-07\\
1.1020432028545	1.09975848191236e-06\\
1.37798695493013	1.81555736699134e-06\\
2.46803102652344	7.20160722514074e-06\\
4.00577192306933	2.79250400512614e-05\\
5.74521223372703	9.88167922135174e-05\\
7.42549280979873	0.000251820090572374\\
8.89479047896517	0.000344112633432161\\
9.64809392418631	0.000434006306658002\\
};
\addlegendentry{ETD2wave $c_2=1$ (update $h$)}

\addplot [color=red, mark=triangle, mark options={solid, red}]
  table[row sep=crcr]{%
0.723862833981183	6.52086822755978e-09\\
1.30567759276112	2.8234603978186e-08\\
1.6918273419317	5.07518701300394e-08\\
1.95984718700065	8.47498488717458e-08\\
2.45478585569938	1.65569046972156e-07\\
4.15504594343838	1.16974700764299e-06\\
5.71664266507282	8.79165607950458e-06\\
7.59771752723946	6.49112921521106e-05\\
8.74337049687092	0.000338611311958805\\
9.95846140180755	0.000698604800997335\\
9.89418925833845	0.000764186805068507\\
};
\addlegendentry{ETD2wave $c_2=2/3$}

\addplot [color=mycolor3, mark=triangle, mark options={solid, rotate=180, mycolor3}]
  table[row sep=crcr]{%
0.378377945433783	6.2513402530729e-09\\
0.719694559559801	2.38430138921002e-08\\
0.91426018669977	3.99102793090503e-08\\
1.10894574790535	6.23093658852251e-08\\
1.37372008022878	1.11429004109383e-07\\
2.47228882487997	6.39229407593306e-07\\
3.88410595516147	4.27400891966051e-06\\
5.74223118870382	2.87255441126362e-05\\
7.48740988226097	0.000142688797847457\\
9.04305751798364	0.000323959244425019\\
9.70362391296063	0.000421105152006572\\
};
\addlegendentry{ETD2wave $c_2=2/3$ (update $h$)}

\addplot [color=green, mark=+, mark options={solid, green}]
  table[row sep=crcr]{%
0.475306085186422	3.67272158091893e-09\\
0.873580011219546	2.0603623223565e-08\\
1.12005300380265	4.16721251768563e-08\\
1.29785680363817	7.54872406188523e-08\\
1.60889429372496	1.59517191652178e-07\\
2.72168634692892	1.26102686529212e-06\\
3.75746432552211	9.80452051034822e-06\\
4.96189169605049	7.26848707280316e-05\\
5.90960317838513	0.000370030778358157\\
6.48346540358658	0.000688493106832089\\
6.6198201362176	0.00073730834423367\\
};
\addlegendentry{ETD3wave}

\addplot [color=black, dotted, forget plot]
  table[row sep=crcr]{%
5.12198389050807	1e-12\\
5.12198389050807	1000\\
};
\end{axis}

\begin{axis}[%
width=4.844in,
height=4.427in,
at={(0in,0in)},
scale only axis,
xmin=0,
xmax=1,
ymin=0,
ymax=1,
axis line style={draw=none},
ticks=none,
axis x line*=bottom,
axis y line*=left,
legend style={legend cell align=left, align=left, draw=white!15!black}
]
\end{axis}
\end{tikzpicture}%
  }
  \caption{
    The relative error in the solution (given in terms of the SSH
    \(\eta = h + b\) and the velocities \(u\)) at the final time in the (linearized) energy norm
    for various methods as a function of the wall clock time, measured in simulated years
    per day (SYPD). The vertical dotted line denotes the maximal SYPD that could be
    obtained with RK4. Left: for gravity wave initial condition \((h_0,0)\)
    after six hours. Right: for the geostrophic initial condition \((h_0,u_0)\) after ten days.
  }
  \label{conv_test_SYPD}
\end{figure}

\paragraph{Artificial dissipation test}
The next test focuses on energy conservation in the
ETD methods and the effect of the artificial numerical dissipation described in
section~\ref{sec:stabilization}. In particular, we investigate its effect on the total
energy. Concretely, we fix the parameter \(p=2\), and consider different values of the
spectral cut-off parameter \(\gamma\). Again, we use the initial condition from
Figure~\ref{fig:single_layer}, which is in geostrophic balance.
First, the evolution of the energy from a simulation using RK4 close
to the maximal time step \(\dt_{\text{RK4}} = \sqrt{8}\dtC\),
is compared to the energy obtained using ETD with
various \(\gamma\) values. The time step size for ETD2wave is chosen as
\(\dt = 10\dt_{\text{RK4}}\) (using the reference heights, \(c_1 = 1\), and \(M=45\)).
The values \(\gamma \in \{5,10,20,30\}\) are employed for ETD2wave, and also the
unmodified case without artificial dissipation is considered.
Secondly, we repeat the same test, but add a biharmonic smoothing to the model (see
Appendix~\ref{app:model_forcing}) with horizontal viscosity \(\disc{\nu}_{\text{h}} = 2
\times 10^{10}\). This is motivated by the fact that the same term will be included in the decade
long simulations in section~\ref{sec:ML_scen}. There, it provides a necessary turbulence
closure, which prevents an unphysical build-up of vorticity in the finest grid cells. We
note that the concrete viscosity value for this grid resolution is taken
from~\cite{wolfram2015diagnosing}.

\begin{figure}[!ht]
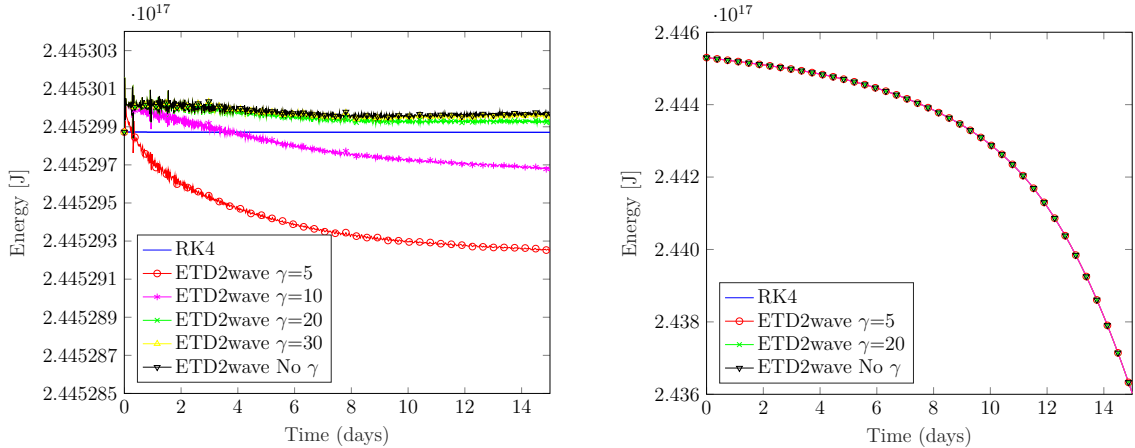

  \centering
  \scalebox{.63}{
    \large
    \input{figures/Energy_tests_nobi_cor.tikz}
  }
  \scalebox{.63}{
    \large
    \input{figures/Energy_tests_bi_cor.tikz}
  }
  \caption{Energy~\eqref{eq:ham_ml} for RK4 at Courant number
    \(\Cour(\dt) = (3/4)\sqrt{8} \approx 2.1\)
    and ETD at \(\Cour(\dt) = 10\sqrt{8} \approx 28.3\) for varying \(\gamma\) values
    over a time horizon of fifteen days. Left, without biharmonic smoothing. Right, with
    biharmonic smoothing.} 
  \label{fig:nobi_test}
\end{figure}
In Figure~\ref{fig:nobi_test} we plot the evolution of the energy for all methods.
Note that, in the case of additional biharmonic smoothing, the curves visually coincide.
This show that the energy dissipating effect
of the biharmonic viscosity is stronger than either the
time discretization error or the artificial numerical diffusion. 
Concerning the case without biharmonic smoothing,
we observe that the ETD methods are affected by a larger
time discretization error than RK4.
This is not surprising, due to the much larger time step employed by these methods.
Concerning the influence of \(\gamma\), we observe that
for the largest value of \(\gamma = 30\) the energy is barely affected, which is explained
by the fact that \(\gamma\) is bigger than the Courant number \(\Cour(\dt) \approx 28.3\).
For smaller \(\gamma\), there is an
increasing effect on the energy, which tends to be dissipative on average.
However, we note that only for the smallest value of
\(\gamma\), the effect of the artificial dissipation is noticeably larger than the
time discretization error.

\subsubsection{Multilayer scenario}
\label{sec:ML_scen}
The second scenario is used to investigate the long term stability and accuracy of the
methods over simulation horizons of decades. The scenario tries to represent a
realistic simulation in the context of climate studies and, in addition to the bathymetry,
shares the same forcing and smoothing terms as the SOMA test case in
\cite[Appendix~A]{wolfram2015diagnosing}.
The wind stress \(\tau_\lambda\) is in the easterly direction in
the center of the domain in the westerly at the top and bottom of the domain. This induces a
double-gyre mean circulation pattern. To extract energy from the system, a quadratic bottom drag
with coefficient \(c_{\text{drag}} = 10^{-3}\) is added.
Also, a vertical Laplacian is implemented
such that the bottom drag term can be interpreted as a Robin-like
bottom boundary condition.
The concrete form of these terms is given in Appendix~\ref{app:model_forcing}.
The horizontal and vertical viscosities are set to
\(\disc{\nu}_{\text{h}} = 2 \times 10^{10}\) and \(\disc{\nu}_{\text{v}} = 10^{-4}\),
respectively (which are the values given in~\cite{wolfram2015diagnosing} for the 16 km grid).
The three layer configuration for this scenario has initial layer interfaces
located at \(\eta^0_{1} = 0\), \(\eta^0_2 = - 25/3\), and \(\eta_3^0 = - 50/3\) [km].
This evenly distributed layer
configuration is chosen to avoid the possible out-cropping of layers (which
refers to the vanishing of one of the layer heights on some part of the domain), which could
lead to a breakdown of the simulation.
The layer densities are set to \(\rho = (1025,1027,1028)\)~[kg~m\(^{-3}\)].

The initial condition is obtained from a ten year spin-up simulation initiated at the
resting state and using RK4 with a time step of
\(\dt = (3/4)\dt_{\text{RK4}}\). The SSH and top layer velocity of the
resulting spun-up initial condition are shown in Figure~\ref{IC_soma}.
\begin{figure}[!ht]
  \centering
  \includegraphics[width=0.48\linewidth,trim={1cm 0cm, 1cm, 1cm},clip]{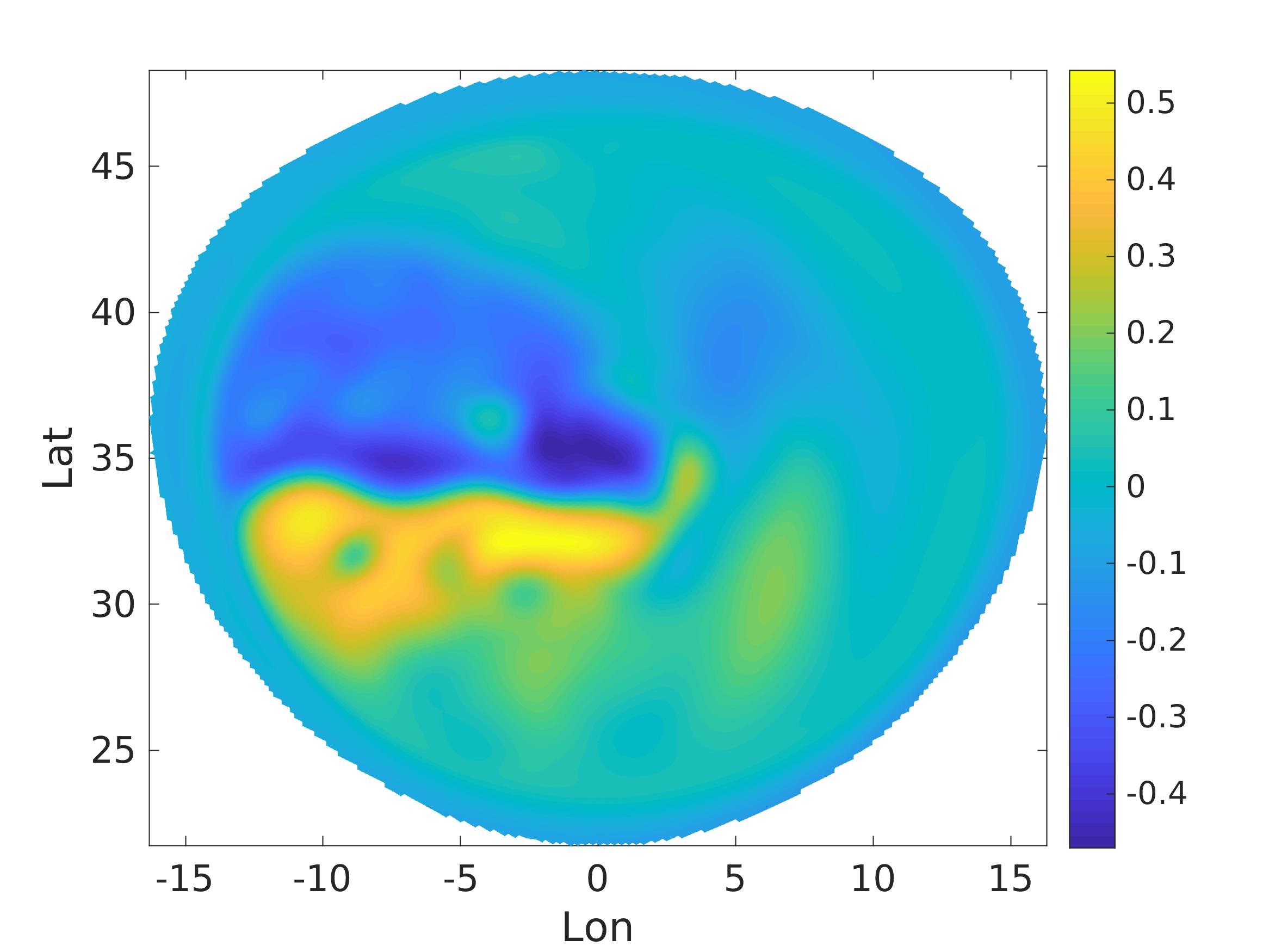}
  \includegraphics[width=0.48\linewidth,trim={1cm 0cm, 1cm, 1cm},clip]{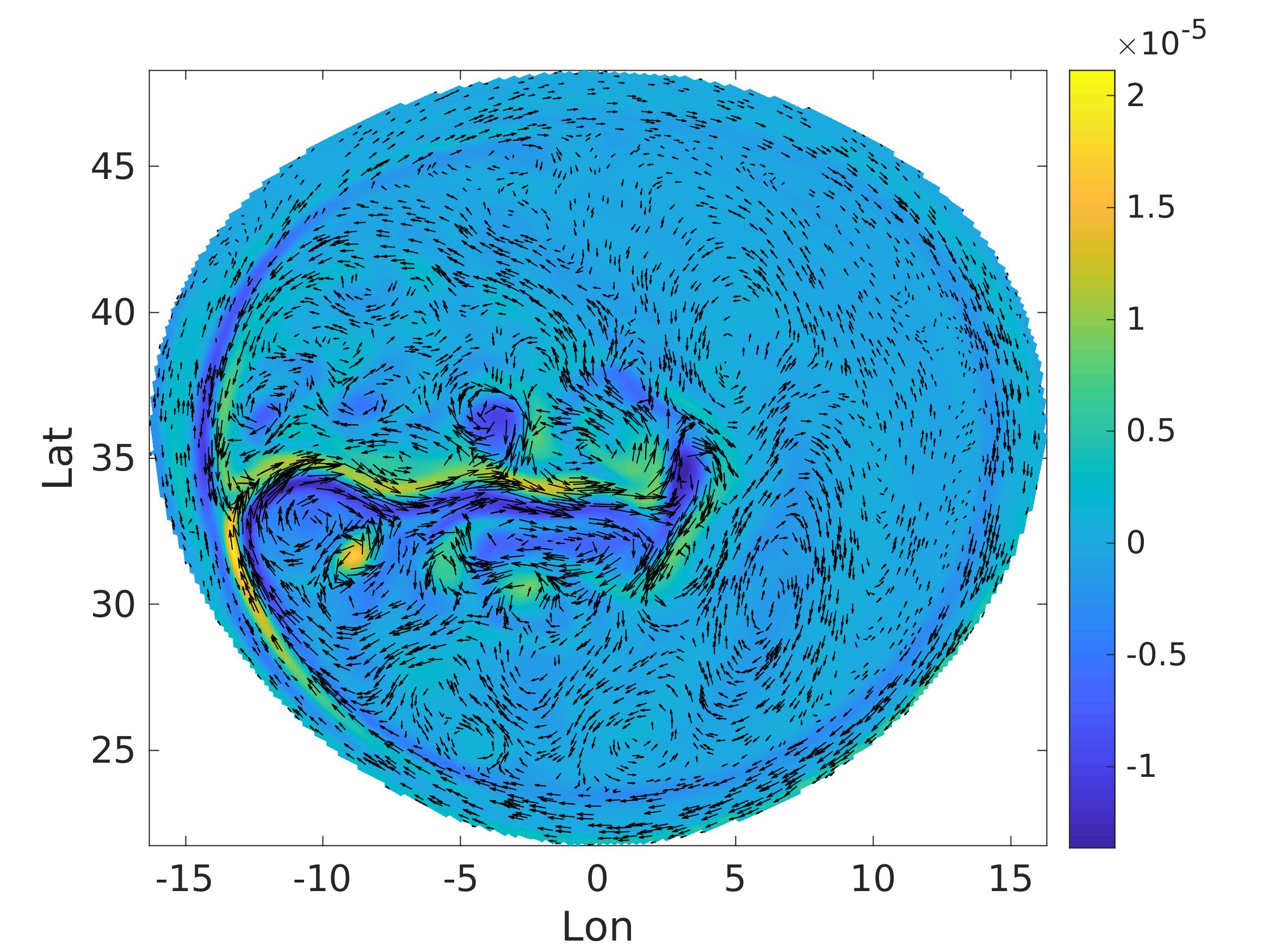}
  \caption{Spin-up initial condition: Left, SSH. Right, velocity vectors [m/s] and
    relative vorticity \(\hat{k}\cdot\curl{u}\) [1/s] (density plot) after ten years in the top
    layer. Latitude and Longitude are in degrees.}
  \label{IC_soma}
\end{figure}
This process ensures that the system is in dynamic equilibrium, which means that the
long-term statistics, such as the mean flow or the root mean square (RMS) of the sea
surface height (SSH), have coherent structure. This
is important since, over time horizons of years and longer, it is expected that the trajectories
computed with different methods will drift apart. Thus, a comparison of instantaneous
values of the solution becomes meaningless, and only the behavior of the
long-term statistics can be used to assess the quality of the different time discretization
methods. A second motivation for evaluating
solution statistics is that climate-ocean models are concerned with long time scale changes,
not instantaneous phenomena. Therefore, a method's ability to accurately predict these long-term
statistics is important.

\paragraph{Results}
We consider a simulation starting from the spin-up initial condition over the horizon of
ten simulation years.
We employ ETD2wave, B-ETD2wave, and B-ETD3wave (using the reference heights and \(c_1 = 1\)) with
time steps increased above the maximal RK4 time step \(\dt_{\text{RK4}} = \sqrt{8}\dtC \approx
107.2\) [s] for the \(16\) km grid. For ETDwave we were not able to obtain stable
simulations in any configuration.
For ETD2wave the time step is increased \(10\) and \(15\) times, for B-ETD2wave \(5\)
and \(7\) times, and for B-ETD3wave \(10\) and \(12.5\) times over \(\dt_{\text{RK4}}\).
Additionally, to avoid spurious high-frequency oscillations, the artificial dissipation
from section~\ref{sec:stabilization} is employed, using \(\gamma = 20\) and \(p=2\).
We note that the larger time step for each ETD method reflects the largest
time step that was empirically found to be stable over the entire time horizon in combination with a
value of \(\gamma = 20\).

For the purposes of comparison, two additional simulations are performed with RK4
at \(1/4\) and \(3/4\) the maximal time step, respectively.
We note that global mass is conserved up to machine precision over the whole
simulation horizon for all considered methods, as predicted by theory.
Due to the wind forcing and smoothing, we can not expect conservation of energy.
The global energy evolution for all methods is given in Figure~\ref{SOMA_EN_mass}.
\begin{figure}[!ht]
  \centering
  \scalebox{.63}{
    \large
    \input{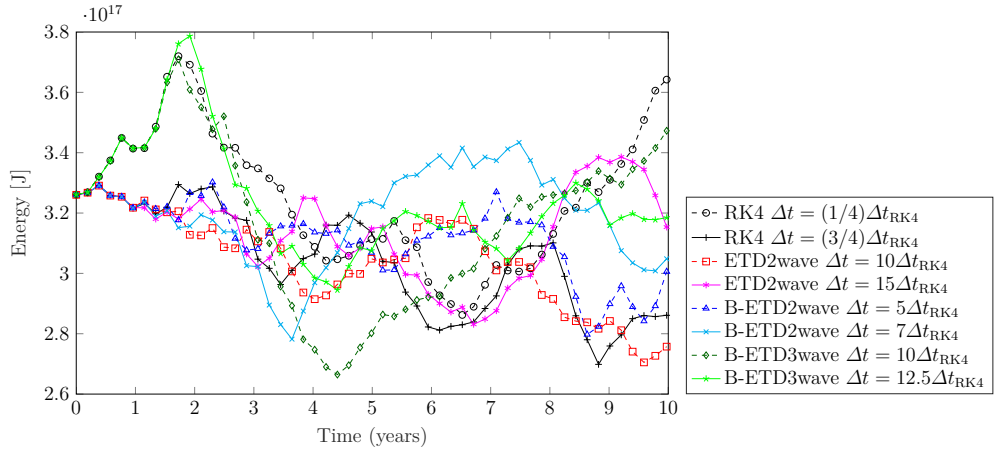}
  }
  \caption{Global energy evolution over ten years in the SOMA simulations for different methods.}
  \label{SOMA_EN_mass}
\end{figure}
From this, it is evident that the solutions
differ significantly after the first years of simulation time.
Therefore, we consider the statistical quantities mean flow and SSH RMS (to be precise, we
compute the RMS of the deviation of the SSH from its temporal mean, which corresponds to the
statistical variance).
The mean and variance are approximated by the statistical mean and the sample
variance, with snapshots taken every two weeks.
The velocity and vorticity of the top layer
mean flow, and the SSH RMS are shown in Figure~\ref{fig:SOMA_RK4}, which are computed
from the RK4 simulation.
\begin{figure}[!ht]
  \centering
  \includegraphics[width=0.48\linewidth,trim={1cm 0cm, 1cm, 1cm},clip]{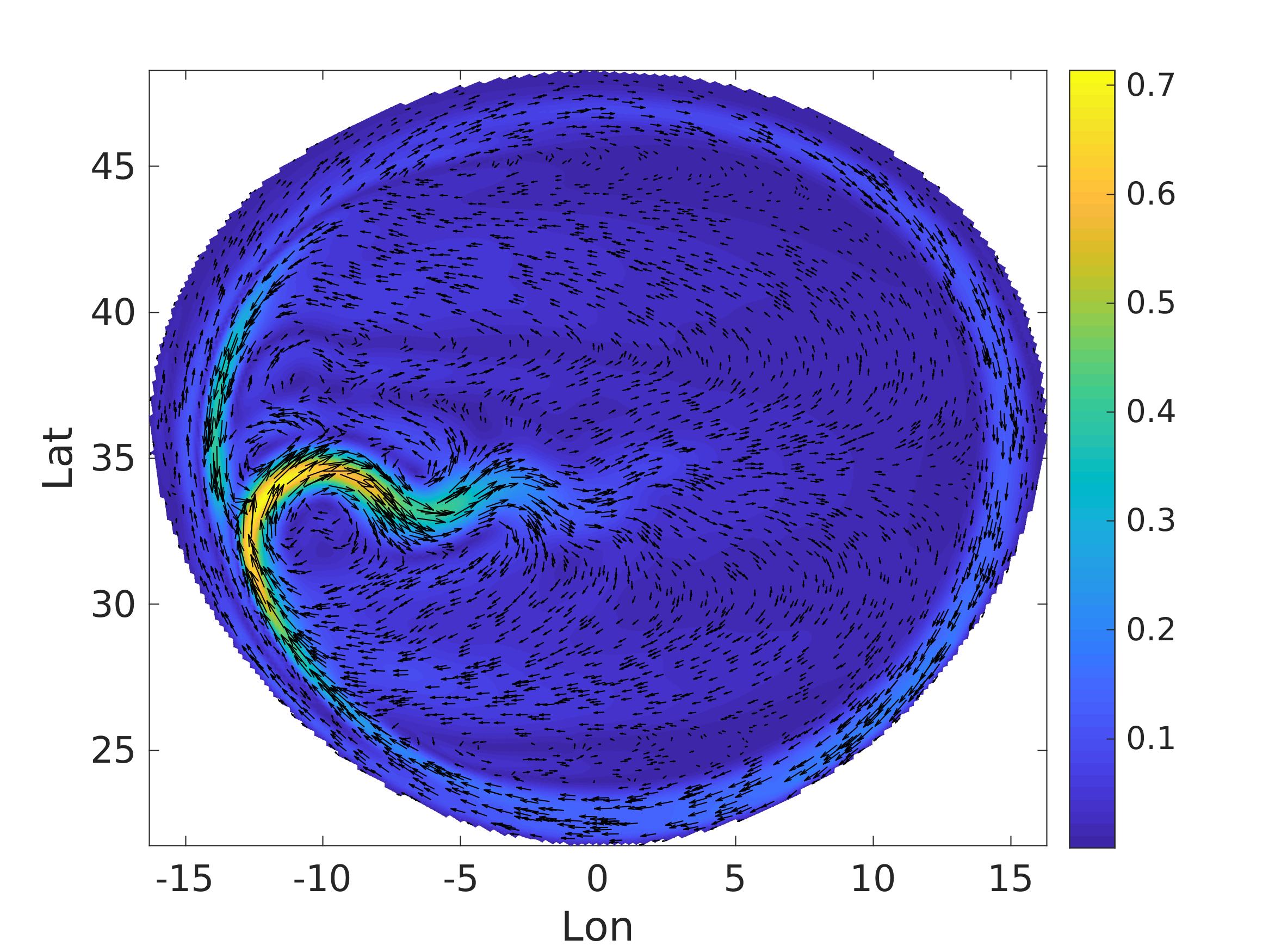}
  \includegraphics[width=0.48\linewidth,trim={1cm 0cm, 1cm, 1cm},clip]{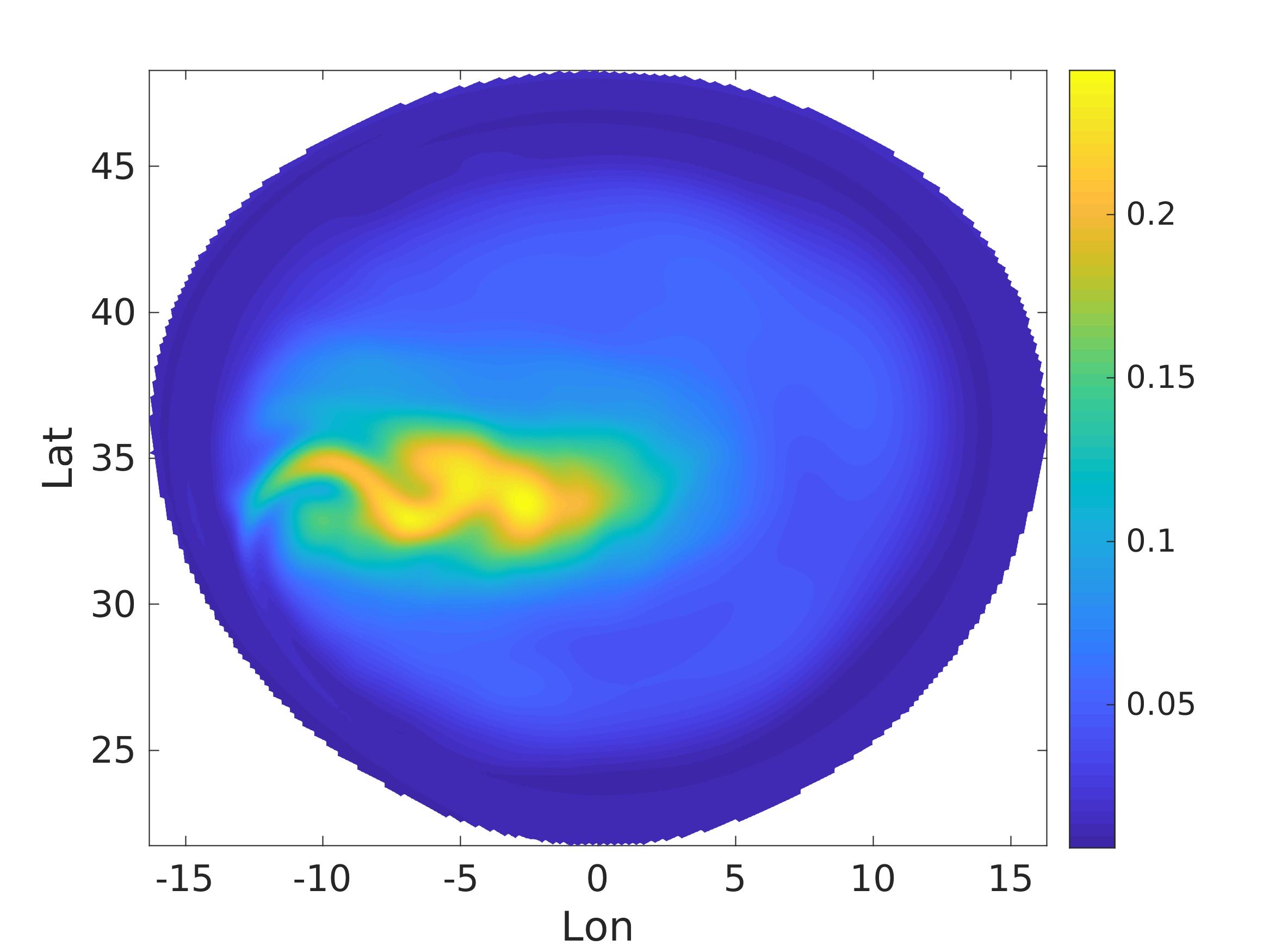}
  \caption{SOMA simulation statistics computed with RK4:
    The mean flow of the velocity, vectors and magnitude [m/s] (left) and the RMS of the
    SSH [m] (right) computed as the sample mean and variance with snapshots of the
    solution taken every two weeks over ten years.
  }
  \label{fig:SOMA_RK4}
\end{figure}
Concretely, we compute the relative reference-thickness weighted \(L^2\) error of the mean
flow and the \(L^\infty\) error of the SSH RMS, which are given in Table~\ref{SOMA_perm}.
Here, we define the reference-thickness weighted \(L^2_{\reference{h}}\) norm for any
velocity profile \(u\) by the square-root of
\(\sum_{k=1,\ldots,L} \int_{\Omega}\reference{h}_k u^2\), which roughly corresponds to the physical
\(L^2\) norm of the underlying three-dimensional velocity field.

Comparing the error in these quantities for each method, using the small time step RK4
solution as a base-line, we find that the results of the methods differ very little.
We observe that all methods (including the RK4 simulation close to the CFL) reproduce the
mean flow up to a similar tolerance of around \(2\%\) to~\(12\%\).
This suggests that the bulk of the error is caused
by replacing the true mean value by a sample average of an effectively random trajectory on a finite
interval and not by the employed time discretization method.
Concerning the maximum error in the SSH RMS, which appears to be a more sensitive criterion,
the methods are reproduce the reference value up to ca.\ \(5\) to \(9\) cm, and
thus range from ca.\ \(20\%\) to \(35\%\) relative accuracy.
However, the accuracy is only slightly affected by
the larger step-sizes, which suggests that all methods reproduce the chosen statistical
quantities similarly accurate in this test.
Moreover, the error does not necessarily increase with larger time steps for
each method, which indicates that much of it could be attributed to statistical
effects caused by the effectively random trajectories and sampling error introduced by
approximating mean and variance by a finite number of samples.
\begin{table}[!ht]
\begin{center}
\begin{tabular}{ l c c c c c c c c}
  \toprule        
  Method & $\dt/\dt_{\text{RK4}}$ & \(\Cour(\dt)\) & \(M\) & SYPD & mean-flow (\(L^2_{\reference{h}}\) rel.)&SSH RMS (\(L^\infty\))  \\
  \midrule
  RK4         & 3/4  & 2.12 & -- & 0.911 & 0.0542 & 0.0541 \\
  ETD2wave    & 10   & 28.3 & 45 & 2.527 & 0.0598 & 0.0510 \\
  ETD2wave    & 15   & 42.4 & 63 & 2.771 & 0.0437 & 0.0603 \\
  B-ETD2wave  & 5    & 14.1 & 28 & 4.051 & 0.0462 & 0.0794 \\
  B-ETD2wave  & 7    & 19.5 & 35 & 4.609 & 0.0370 & 0.0714 \\
  B-ETD3wave  & 10   & 28.3 & 48 & 3.583 & 0.1372 & 0.0864 \\
  B-ETD3wave  & 12.5 & 35.4 & 56 & 4.204 & 0.0206 & 0.0716 \\
  \bottomrule
\end{tabular}
\end{center}
\vspace{-1em}
\caption{Comparison of the ETD methods against the base-line RK4 method: the Courant number,
  the number of Krylov vectors \(M\), the average simulated years per real day (SYPD),
  the relative mean velocity error in the \(L^2_{\reference{h}}\) norm, and the SSH RMS
  error in the \(L^\infty\) norm. Reference values are computed with using RK4
    with \(\dt = (1/4)\dt_{\text{RK4}}\).}
\label{SOMA_perm}
\end{table}
With regards to performance, Table~\ref{SOMA_perm} also shows the simulated years per real day of the various
ETD methods and RK4. In all cases, ETD outperforms RK4 in terms of wall times. Moreover, the
B-ETD2wave method outperforms the ETD2wave method, although it is forced to use a smaller time
step, which is due to the cheaper \(\phi\)-function computations.
Moreover, we note that the barotropic method with three stages is able to take
time steps of similar magnitude as ETD2wave, in contrast to the barotropic second order
version. We attribute this to the fact that B-ETD2wave is based on RK2, which is not
stable on the imaginary axis, but still has to resolve the baroclinic waves contained in
the remainder. On the
other hand, B-ETD3wave is based on RK3, which includes an imaginary interval in its
stability region.
This, combined with the fact that the cost of additional stages
decreases for the barotropic method in configuration with more than three layers, will
likely make it (or even higher order methods) favorable for configurations with more
layers.

\section{Conclusion}
In this paper, we have developed ETD methods that can take big time steps for
the multilayer shallow water equations and deliver sufficiently accurate solutions at a
reduced cost compared to explicit methods. We have based this development specifically on the
spatial TRiSK-scheme, but it applies directly to any scheme based on a
Hamiltonian framework. More generally, it should be applicable to any scheme that
conserves a discrete energy. This also includes classical finite difference/finite
volume schemes on structured quad-meshes, used in ocean modeling.

In the following, we address the further steps that are
needed to use these methods in order to improve current ocean models running on massively parallel
architectures.
Most currently employed ocean models use a
splitting of the dynamics into a fast free-surface equation and a slow remainder, which
usually speeds up simulations by an order of magnitude. This is achieved by solving the
free-surface equation either with explicit or implicit time stepping methods.
ETD methods are an attractive substitute, since they restrict the fast dynamics to be linear, and
allow for different approximation methods, such as polynomial or rational Krylov methods.
Moreover, methods of high order are available.
In this paper we have only considered problems with up to three layers, such that
the potential for computational speed-up exploiting the structure of
the barotropic method was limited. Clearly, if more layers are added, the computational
effort for the barotropic methods is going to decrease relative to RK4,
since the reduction in size of the reduced linear operator is increased, and the
amount of work to compute the \(\phi_s\) functions remains independent of the number of
layers.
In such situations, larger time steps, higher order methods, and methods incorporating also the first
baroclinic mode into the linear operator, which we did not consider in our numerics, may
become increasingly competitive.

In order to reliably consider situations with more layers, it will be essential to make
the model more realistic. For instance, the out-cropping of internal layers (layer heights going to
zero) can no longer be avoided for thin layers. This can be addressed by leaving the
isopycnal reference
frame and considering the primitive equations in an arbitrary Lagrangian vertical
coordinate system together with tracer equations (e.g., for temperature and salinity) and
an equation of state.
In the future, we aim to extend the methods to this case. Certainly, the
development of monolithic ETD methods for the combined set of equations is desirable, but
preliminary versions can be based upon an operator splitting into isopycnal dynamics
and separate tracer plus ALE updates, which allow a more direct use of the developed methods.
Moreover, the ETD methods should include an appropriate treatment of
the tracer equations, and be able to scale to solve a large number of such equations
efficiently. Here, exponential methods can also enable larger time steps, and a
pre-computation of \(\phi\)-functions for vertical transport may enable further efficiency.

Additional challenges arise due to the necessity of implementing these methods in
massively parallel environments. Concerning the ETD
methods proposed here, we first must note that Krylov methods require a ``reduce-all''
communication step in every iteration, which can be inefficient on certain parallel
architectures. Here, other iterative approaches such as an approximation of the matrix
functions using Chebyshev polynomials may be used instead, to avoid global communication.
An additional problem of the iterative
solution of the fast equation, which also plagues split-explicit methods, are the frequent
communications with small message size and only a small number of floating point
operations in between. This makes overlapping domain
decomposition methods a promising alternative, due to the finite speed of propagation of
the free surface waves. In this context, exponential methods can additionally exploit the
linearity of the propagation matrices and the recursive relations between exponentials
associated to time intervals of different length.
In terms of incorporating ETD methods in existing computational ocean
models, one may also consider them as a drop-in replacement for the single-layer barotropic
solver. Here, the speed-up can not come from the layer reduction, but instead must come from
the computational advantages of matrix exponentials of the linear operator over the
existing implicit or explicit solution procedure.

Finally, global ocean models are expected to use increasingly nonuniform meshes of higher
resolution near coastal boundaries, in order to more accurately resolve local features and
interface with coastal/estuary models. As long as the bathymetry is sufficiently deep such
that linear waves are still sufficiently faster than the advection, ETD methods using a
global time step may still be effective. However, as the CFL requirements of the fast
waves become less restrictive and the advective CFL becomes more so near the coast,
smaller explicit time steps may be more beneficial. Here, the development of ETD methods
that can take different time steps in different parts of the domain may provide a natural
way of realizing that, since they degrade smoothly to explicit methods for smaller time
steps.

\section*{Acknowledgments}

The authors gratefully acknowledge funding by the US Department of Energy Office of Science
grant DE-SC0016591.




%
%

\appendix
\addcontentsline{toc}{section}{\textbf{Appendices}}
\renewcommand{\thesection}{\Alph{section}}

\section{TRiSK operators}
\label{app:TRiSK}

In order to implement the TRiSK scheme several discrete operators are required for the
differential operators and for averaging quantities from one grid location to another. In
total there are eight operators required for the scheme that are built using the
connectivity relations defined in \cite[p.~6, table~2]{ringler2013multi}. The operators
consist of the divergence $\Ddiv$, the gradient
$\Dgrad$, the gradient in the perpendicular direction
$\Dgradp$, the scalar curl
$\Dcurl$, the (perpendicular) flux reconstruction
operator $\Dkcross$, and the interpolation operators.

The definition of flux reconstruction operator
$\Dkcross$ given in \cite{thuburn2012framework}, is a necessary condition for geostrophic
balance. The ensures that Coriolis force and pressure gradient balance each other to
maintain divergent free flows under the correct conditions. For a given Delaunay
triangulation (and its CVT) on the sphere, along with the normals $n_{e,i}$ (elements of
the matrix $\disc{N}$), tangents $t_{e,v}$ (elements of the matrix $\disc{T}$), kite areas
(the intersections between triangles) $R_{i,v} A_i$ (elements of the matrix $\vb{R}$), the
weights on edge array $w_{e,e'}$ (elements of the matrix
$\disc{W}$) are then defined as quantities satisfying
the relation 
\begin{equation}\label{weight_def}
-\disc{T}\disc{W} = \disc{R}\disc{N}.
\end{equation}
We note that the presented identities are only valid for a domain which has no boundary,
e.g., the full sphere.

In order to obtain discrete operators on \(\Omega \subset \mathbb{S}^2\), we follow the
procedure used in the MPAS-O software and restrict a spherical mesh to a
subset of the cells, which eliminates the ``dry'' cells of zero layer height.
Note, that this results only in a first-order accurate resolution of
the boundary. We obtain a discretization of the model on the bounded domain
by fixing all velocity variables stored in edges adjacent
to at least one ``dry'' cell to zero, which conveniently incorporates the no flux boundary
conditions. However, since the edge-tangential velocity is reconstructed from the
edge-normal velocity,
and this reconstruction takes into account the zero edge
velocities in boundary and ``dry'' edges,
this implementation effectively introduces a full no-slip condition for the
velocity. Thus, it is essential to employ additional diffusion terms such as (bi-)harmonic
closure to obtain a mathematically meaningful model,
since the shallow water equations are over-specified with no-slip conditions.

\begin{table}
\begin{center}
\begin{tabular}{l r c l}
\toprule
 Divergence: &$(\Ddiv \disc{y})_e$ & $=$ &$(1/A_i)\sum_{e\in \grid{EoI}(i)} n_{e,i} l_e \disc{y}_e$\\
 Gradient: & $(\Dgrad \disc{y})_i$ & $=$ & $(1/d_e)\sum_{i\in \grid{IoE}(e)} - n_{e,i} \disc{y}_i$\\
 Curl: &$(\Dcurl\disc{y})_v$ & $=$ & $({1}/{A_v})\sum_{e\in \grid{EoV}(v)} t_{e,v} d_e \disc{y}_e $\\
 Perpendicular Gradient: &$(\Dgradp \disc{y})_e$ & $=$ & $({1}/{l_e})\sum_{v\in \grid{VoE}(e)}t_{e,v}\disc{y}_v$\\
 Perpendicular Flux: &$(\Dkcross \disc{y})_e$ & $=$ & $({1}/{d_e})\sum_{e'\in \grid{EoE}(e)}w_{e,e'}l_e\disc{y}_{e'}$\\
 Cell to Vertex interpolation: & $\inp{\disc{y}}_{\grid{V},v}$ & $=$ & $({1}/{A_v})\sum_{i\in \grid{CoV}(v)}R_{i,v}A_i\disc{y}_i$\\
 Vertex to Edge interpolation: &$\inp{\disc{y}}_{\grid{E},e}$ & $=$ & $\sum_{v\in \grid{VoE}(e)}\disc{y}_v/2$\\
 Edge to Cell interpolation: &$\inp{\disc{y}}_{\grid{I},i}$ & $=$ & $({1}/{A_i})\sum_{e\in \grid{EoI}(i)}\disc{y}_eA_e/2$\\
 Cell to Edge interpolation: &$\inp{\disc{y}}_{\grid{E},e}$ & $=$ & $\sum_{i\in \grid{IoE}(e)} \disc{y}_i/2$\\
\bottomrule
\end{tabular}
\end{center}
\caption{The discrete operators given concretely in terms of geometrical quantities.
  $d_e$ and $l_e$ denote the distances between the cell centers and cell
  vertices, respectively. $A_v$ and $A_i$ are the triangle and cell areas, respectively.
  $R_{i,v}$ are the kite-areas, the intersection of a primal
  and dual grid cell divided by the cell area. $w_{e,e'}$ are the
  edge-weights from~\eqref{weight_def}. The index
  sets in the summation correspond to geometrical connectivity arrays~\cite{ringler2010unified}.}
\label{op_tab}
\end{table}

\subsection{Choices of the forcing term}
\label{app:model_forcing}

In wind-driven circulation, energy is injected at the ocean surface by a source term in
the momentum equation. Concretely, the forcing term can be implemented as
\[
\disc{f}_{\text{wind}}[\disc{h}] = 
\chi_{k=1} \, \disc{\tau}_{\text{wind}} / \inp{\rho_1\disc{h}_1}_{\grid{E}},
\]
where \(\disc{\tau}_{\text{wind},e} \approx n_e \vdot \tau_{\text{wind}}(x_e)\) is a edgewise
approximation to the continuous wind profile, and the characteristic function
\(\chi_{k=1} \in \{\,0,1\,\}\) ensures that wind forcing is only applied in the top layer.

Energy is typically extracted in the bottom layer, by a drag term that represents
interaction of the flow with the (rough) bottom topography. A classical choice for this
term is
\[
\disc{f}_{\text{drag}}[\disc{h},\disc{u}] = 
- \chi_{k=l_{\text{bot}}} \, c_{\text{drag}} \frac{\inp{\sqrt{2 \disc{K}[\disc{u}_{l_{\text{bot}}}]}}_{\grid{E}}}{\rho_{l_{\text{bot}}}\inp{\disc{h}_{l_{\text{bot}}}}_{\grid{E}}} \ast \disc{u}_{l_{\text{bot}}},
\]
where \(l_{\text{bot}}\) is the bottom layer index. This corresponds to a quadratic drag term
\(- c_{\text{drag}} \abs{u_{l_{\text{bot}}}} u_{l_{\text{bot}}} /h_{l_{\text{bot}}}\) in the
continuous equation.

Due to the massive length scales relevant for global ocean modeling and the relatively
coarse discretization, physical viscosities in the momentum equation are usually
negligible. However, in order to account for the energy dissipated in scales below the
grid resolution (due to turbulence), and to prevent a build-up of vorticity in finest grid
scales, numerical dissipation terms have to be introduced to the discrete equation. Here,
we employ a classical biharmonic viscosity, which is modified to be energetically
consistent. Concretely, we choose
\[
\disc{D}_{\text{bihar}}[\disc{h},\disc{u}]
 = -\frac{1}{\inp{\disc{h}}_{\grid{E}}} \ast \DLap \left(\disc{\nu}_{\text{h}} \ast \inp{\disc{h}}_{\grid{E}} \ast
 \left(\DLap \disc{u}\right)\right),
\]
where \(\DLap = \Dgrad \Ddiv - \sqrt{3}\Dgradp \Dcurl\) is a discrete approximation to an
anisotropic vectorial Laplace-Beltrami operator (see~\cite{thuburn2012framework}).

The appearance of \(\disc{h}\) is motivated by the form of physical viscosities in the shallow
water equation (see, e.g., \cite{Bresch:2009}), and the fact that the concrete form given above
leads to consistent energy dissipation in the discrete equation.
In fact, combining these terms by a choice of 
\[
\Fdisc[\disc{h},\disc{u}]
= \disc{D}_{\text{bihar}}[\disc{h},\disc{u}]
+ \disc{f}_{\text{drag}}[\disc{h},\disc{u}]
+ \disc{f}_{\text{wind}}[\disc{h}],
\]
we obtain for~\eqref{eq:sw_disc} the energy equality
\begin{multline*}
\dv{\Hdisc}{t} (\disc{V}) =
 - \sum_{k=1}^L\rho_k\inner{\disc{\nu}_{\text{h}} \ast \inp{\disc{h}_k}_{\grid{E}} \ast \DLap \disc{u}_k , \DLap \disc{u}_k}_{\grid{E}} \\
 - c_{\text{drag}} \inner{\inp{\sqrt{2 \disc{K}[\disc{u}_{l_{\text{bot}}}]}}_{\grid{E}}
   \ast \disc{u}_{l_{\text{bot}}}, \disc{u}_{l_{\text{bot}}}}_{\grid{E}}
 + \inner{\disc{u}_k, \disc{\tau}_{\text{wind}}}_{\grid{E}},
\end{multline*}
which shows that the smoothing and damping terms are energy dissipating.
The horizontal viscosity \(\disc{\nu}_{\text{h}} \in X_{\grid{I}}\) is usually chosen in a grid
dependent fashion. However, since we only employ quasi-uniform grids, we set it to a
constant in computations.

Additionally, in the multilayer case a vertical smoothing can be introduced in the
momentum equation in the form of a vertical Laplacian. This can be based on a mimetic
discretization of a vertical gradient and divergence. Since the
vertical mesh size is given by the layer thicknesses \(\disc{h}_k\), the vertical
Laplacian will depend non-linearly on the variable \(\disc{h}\). However, for the sake of
brevity, we omit a detailed presentation. We only note that in this case, the drag term
given above can also be interpreted as a Robin-like boundary condition for the vertical
Laplacian.

\section{Linearized operators}
\label{app:lin_op}

For convenience, we give the explicit form of the differential operators defined in
section~\ref{sec:TRiSK_multilayer}. The second variation of the
Hamiltonian~\eqref{eq:sw_ham_disc} can be computed as the linearization
of~\eqref{eq:var_hamil} as
\begin{equation*}
\delta^2\Hdisc[\disc{V}] \disc{W} = \delta\Hdisc'[\disc{V};\disc{W}] =
\begin{pmatrix}
g \left(R \disc{w}^h\right)_k + \rho_k \inp{\disc{u}_k \ast \disc{w}^u_k}_{\grid{I}}  \\
\rho_k \disc{u}_k \ast \inp{\disc{w}^h_k}_{\grid{E}} + \rho_k \inp{\disc{h}_k}_{\grid{E}} \ast \disc{w}^u_k
\end{pmatrix}_{k=1,2,\dotsc,L}\;,
\end{equation*}
for all \(\disc{W} = (\disc{w}^h,\disc{w}^u)\). In the case that \(\disc{V} =
\reference{\disc{V}} = (\reference{\disc{h}}, \disc{0})\), it can be represented by the
block-diagonal matrix
\begin{equation*}
\delta^2\Hdisc[\reference{\disc{V}}] =
\begin{pmatrix}
g R & 0  \\
0 & \diag_{k=1,2,\dotsc,L}(\rho_k\,\inp{\reference{\disc{h}}_k}_{\grid{E}})
\end{pmatrix}\;.
\end{equation*}
The linearization of \(\Jdisc\) from~\eqref{eq:J_disc} is given for all \(\disc{W} =
(\disc{w}^h,\disc{w}^u)\) as
\[
\Jdisc'[\disc{V};\disc{W}]
=
\diag_{k=1,2,\dotsc,L}
\frac{1}{\rho_k}
\begin{pmatrix}
0 & 0 \\
0 & \disc{Q}'[\disc{V}_k;\disc{W}_k]\\
\end{pmatrix}\;,
\]
where
\[
\disc{Q}'[\disc{V}_k;\disc{W}_k]\disc{y} =
\frac{1}{2}\left(
\inp{\disc{q}'[\disc{V}_k;\disc{W}_k]}_{\grid{E}} \ast \left(\Dkcross \disc{y}\right)
+ \Dkcross \big(\inp{\disc{q}'[\disc{V}_k;\disc{W}_k]}_{\grid{E}} \ast \disc{y}\big)
\right)
\]
and \(\disc{q}'[\disc{V}_k;\disc{W}_k]
= (\Dcurl \disc{w}^u_k) / \inp{\disc{h}_k}_{\grid{V}}
- (\Dcurl \disc{u}_k + \disc{f}) \ast \inp{\disc{w}^h_k}_{\grid{V}}  / \inp{\disc{h}_k}^2_{\grid{V}}\).
We remark that due the sparsity-pattern of \(\Dkcross\), which has the most entries of any
of the discrete operators considered (apart from the biharmonic smoothing term), this term is
expensive to evaluate in practice.

\section{Exponential Runge-Kutta schemes}
\label{app:exp_rk}

Exponential integrators can be given in terms of their Butcher tableau, which contains the
intermediate time points \(c_i\), the coefficients for the internal stages \(a_{i,j}\), and
the final coefficients \(b_j\) in the form
\[
\begin{array}{c|c}
c_i & a_{i,j}  \\
\hline
    & b_j
\end{array}\quad,
\]
which represents the method in terms of the remainder as
\begin{align*}
\disc{v}_{n,i} &= \exp(c_i\dt \disc{A}_n)\disc{V}_n + \dt \sum_{j=1,\dotsc,i-1} a_{i,j}(\dt \disc{A}_n)\disc{r}_n(\disc{v}_{n,j}), \\
\disc{V}_{n+1} &= \exp(\dt \disc{A}_n)\disc{V}_n + \dt \sum_{j=1,\dotsc,S}
                 b_j(\dt\disc{A}_n)\disc{r}_n(\disc{v}_{n,j}),
\end{align*}
where \(c_1 = 0\) and \(a_{i,j} = 0\) for \(j \geq i\) for explicit methods, which implies
\(\disc{v}_{n,1} = \disc{V}_n\). Under the simplifying assumptions
\(\sum_{j=1,\dotsc,S} b_j = \phi_1\) and
\(\sum_{j=1,\dotsc,i-1} a_{i,j} = \phi_1(c_i\cdot)\), these methods can be 
equivalently rewritten in terms of the residual as:
\begin{align*}
\disc{v}_{n,i} &= \disc{V}_n + \phi_1(c_i\dt \disc{A}_n)\disc{F}(\disc{V}_n)
                 + \dt \sum_{j=2,\dotsc,i-1} a_{i,j}(\dt \disc{A}_n)\disc{R}_n(\disc{v}_{n,j}), \\
\disc{V}_{n+1} &= \disc{V}_n + \phi_1(\dt \disc{A}_n)\disc{F}(\disc{V}_n)
                 + \dt \sum_{j=2,\dotsc,S} b_j(\dt\disc{A}_n)\disc{R}_n(\disc{v}_{n,j}).
\end{align*}
The two stage method from section~\ref{sec:exponential_int}
taken from~\cite{hochbruck2005explicit} is characterized by the Butcher tableau
\[
\begin{array}{c|cc}
0 & 0 & 0 \\
c_2 & c_2 \phi_1(c_2 \cdot) & 0 \\
\hline
  & \phi_1 - (1/c_2)\phi_2 & (1/c_2)\phi_2
\end{array}\quad,
\]
i.e.\ \(a_{2,1}(\cdot) = c_2 \phi_1(c_2 \cdot)\), \(b_1 = \phi_1 - (1/c_2)\phi_2\), and
\(b_2 = (1/c_2)\phi_2\), where \(c_2 \in (0,1)\).
A third-order three stage method (also taken from~\cite{hochbruck2005explicit}) is given by
the Butcher tableau
\[
\begin{array}{c|ccc}
0 & 0 & 0 & 0 \\
c_2 & c_2 \phi_1(c_2 \cdot) & 0 & 0\\
c_3 & c_3 \phi_1(c_3 \cdot) - a_{3,2} & \gamma c_2 \phi_2(c_2\cdot) + ({c_3^2}/{c_2})\phi_2(c_3\cdot) & 0\\
\hline
  & \phi_1 - b_2 - b_3 & ({\gamma}/({\gamma c_2 + c_3}))\,\phi_2 & ({1}/({\gamma c_2 + c_3}))\,\phi_2 
\end{array}\quad,
\]
where either \(c_3 = 2/3\), \(c_2 \in (0,1)\), and \(\gamma = 0\)
or \(c_2, c_3 \in (0,1)\setminus\{\,2/3\,\}\) and \(\gamma = - (3c_3^2 - 2c_3)/(3c_2^2 - 2c_2)\).


\bibliographystyle{siam}

\bibliography{Bibliography}

\end{document}